\documentclass[a4paper,reqno,11pt]{amsart}
\usepackage{amsmath,amsfonts,amssymb,amsthm}
\usepackage[usenames,dvipsnames]{color}
\usepackage[utf8]{inputenc}
\usepackage[T1]{fontenc}
\usepackage[margin = 1in]{geometry}
\usepackage{mathtools}
\usepackage{mathabx}
\usepackage{bbm}
\usepackage{enumerate}
\usepackage{textcomp}
\usepackage{comment}
\usepackage{graphicx}
\usepackage{cite}
\usepackage{dsfont}

\usepackage[hyphens]{url}
\usepackage[hidelinks, colorlinks, allcolors = black, breaklinks = true]{hyperref}
\usepackage{esint}

\allowdisplaybreaks

\makeatletter
\def\subsubsection{\@startsection{subsubsection}{3}
  \z@\z@{-\fontdimen2\font}
  {\normalfont\bfseries}}
\makeatother

\makeatletter
\def\paragraph{\@startsection{paragraph}{4}
  \z@\z@{-\fontdimen2\font}
  {\normalfont\bfseries}}
\makeatother

\makeatletter
\newcommand{\proofstep}[1]{
  \par
  \addvspace{\medskipamount}
  \textit{#1\@addpunct{.}}\enspace\ignorespaces
}
\makeatother

\newcommand{\trans}{{\ensuremath{\tt t}}}
\renewcommand{\d}{\ensuremath{\mathrm{d}}}
\newcommand{\R}{\ensuremath{\mathbb{R}}}
\newcommand{\E}{\ensuremath{\mathbb{E}}}
\newcommand{\PP}{\ensuremath{\mathbb{P}}}

\newcommand{\C}{\ensuremath{\mathbb{C}}}
\newcommand{\N}{\ensuremath{\mathbb{N}}}
\newcommand{\LL}{\ensuremath{\mathcal L}}

\newcommand{\eps}{\ensuremath{\varepsilon}}
\newcommand{\verti}[1]{\ensuremath{\left\lvert #1 \right\rvert}}
\newcommand{\vertii}[1]{\ensuremath{\left\lVert #1 \right\rVert}}
\newcommand{\vertiii}[1]{{\left\lvert\kern-0.25ex\left\lvert\kern-0.25ex\left\lvert #1
    \right\rvert\kern-0.25ex\right\rvert\kern-0.25ex\right\rvert}}
\renewcommand{\d}{\ensuremath{{\rm d}}}
\newcommand{\eoperp}{\ensuremath{{_{\sqrt\eps}\operp}}}

\newtheorem{theorem}{Theorem}[section]
\newtheorem{definition}[theorem]{Definition}
\newtheorem{proposition}[theorem]{Proposition}
\newtheorem{remark}[theorem]{Remark}
\newtheorem{lemma}[theorem]{Lemma}

\newtheorem{globalassumptions}[theorem]{Global Assumptions}
\newtheorem{globalassumption}[theorem]{Global Assumption}
\newtheorem{assumptions}[theorem]{Assumptions}

\numberwithin{equation}{section}
\def\R{\mathbb{R}}
\def\C{\mathbb{C}}
\def\N{\mathbb{N}}

\def\E{\mathbb{E}}

\newcommand{\be}{\begin{equation}}
\newcommand{\ee}{\end{equation}}
\newcommand{\bea}{\begin{eqnarray}}
\newcommand{\eea}{\end{eqnarray}}
\newcommand{\beann}{\begin{eqnarray*}}
\newcommand{\eeann}{\end{eqnarray*}}
\newcommand{\benn}{\begin{equation*}}
\newcommand{\eenn}{\end{equation*}}

\newcommand{\cA}{{\mathcal A}}  % calligraphic A
\newcommand{\cB}{{\mathcal B}}  % calligraphic B
\newcommand{\cF}{{\mathcal F}}  % calligraphic F
\newcommand{\cL}{{\mathcal L}}  % calligraphic L
\newcommand{\cM}{{\mathcal M}}  % calligraphic M
\newcommand{\cO}{{\mathcal O}}  % calligraphic O
\newcommand{\cR}{{\mathcal R}}  % calligraphic R
\newcommand{\cT}{{\mathcal T}}  % calligraphic T
\newcommand{\cU}{{\mathcal U}}  % calligraphic U
\newcommand{\cV}{{\mathcal V}}  % calligraphic V
\DeclareMathOperator{\argmin}{\mathrm{argmin}}
\DeclareMathOperator{\esssup}{\mathrm{ess-sup}}
\DeclareMathOperator{\ind}{\mathrm{ind}}
\DeclareMathOperator{\diag}{\mathrm{diag}}
\DeclareMathOperator{\id}{\mathrm{id}}

\newcommand{\dualpair}[4]{\ensuremath{\,_{#1\kern-0.2ex}\left\langle#2,#3\right\rangle_{#4}}}
\newcommand{\inner}[3]{\ensuremath{\left(#1,#2\right)_{#3}}}

\renewcommand{\Re}{\mathrm{Re\,}}
\renewcommand{\Im}{\mathrm{Im\,}}

\begin{document}
%%%%%%%%%%%%%%%%%%%%%%%%%%%%%%%%%%%%%%%%%%%
%
\title[Multiscale analysis for the stochastic FitzHugh-Nagumo equations]{Multiscale analysis for traveling-pulse solutions to the stochastic FitzHugh-Nagumo equations}
\thanks{MVG is grateful to Mark C.~Veraar for discussions. CK appreciates discussions with Wilhelm Stannat. The authors thank Alexandra Neamtu for advice and a careful reading of the manuscript. Several remarks of the anonymous reviewers have helped to improve the content and presentation of this revised version. KE acknowledges partial support from the European Union's Horizon 2020 research and innovation programme under the Marie Sk\l odowska-Curie grant agreement \# 754362. KE and MVG have been partially supported by the Deutsche Forschungsgemeinschaft (German Research Foundation -- DFG) under project \# 334362478. CK is supported by the Volkswagen-Stiftung through a Lichtenberg Professorship. KE appreciates the kind hospitality of Delft University of Technology. MVG appreciates the kind hospitality of Heidelberg University and the Technical University of Munich.\newline
\includegraphics[width=1.5cm]{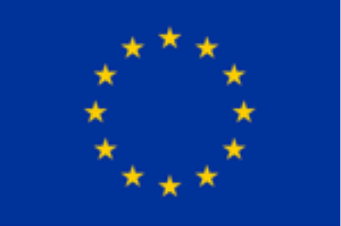}}
\date{\today}
\author{Katharina Eichinger}
\address[Katharina Eichinger]{CEREMADE, Universit\'e Paris Dauphine, PSL, Pl. de Lattre Tassigny, 75775 Paris Cedex 16, France and INRIA-Paris, MOKAPLAN, 2 Rue Simone IFF, 75012 Paris, France}
\email{eichinger@ceremade.dauphine.fr}
\author{Manuel V.~Gnann}
\address[Manuel V.~Gnann]{Delft Institute of Applied Mathematics, Faculty of Electrical Engineering, Mathematics and Computer Sciences, Delft University of Technology, Van Mourik Broekmanweg 6, 2628 XE Delft, Netherlands}
\email{M.V.Gnann@tudelft.nl}
\author{Christian Kuehn}
\address[Christian Kuehn]{Center for Mathematics, Technical University of Munich, Boltzmannstr.~3, 85747 Garching near Munich, Germany}
\email{ckuehn@ma.tum.de}
\begin{abstract}
We investigate the stability of traveling-pulse solutions to the stochastic FitzHugh-Nagumo equations with additive noise. Special attention is given to the effect of small noise on the classical deterministically stable fast traveling pulse. Our method is based on adapting the velocity of the traveling wave by solving a scalar stochastic ordinary differential equation (SODE) and tracking perturbations to the wave meeting a system of a scalar stochastic partial differential equation (SPDE) coupled to a scalar ordinary differential equation (ODE). This approach has been recently employed by Kr\"uger and Stannat [Nonlinear
Anal., 162:197--223, 2017] for scalar stochastic bistable reaction-diffusion equations such as the Nagumo equation. A main difference in our situation of an SPDE coupled to an ODE is that the linearization has essential spectrum parallel to the imaginary axis and thus only generates a strongly continuous semigroup. Furthermore, the linearization around the traveling wave is not self-adjoint anymore, so that fluctuations around the wave cannot be expected to be orthogonal in a corresponding inner product. We demonstrate that this problem can be overcome by making use of Riesz instead of orthogonal spectral projections as recently employed in a series of papers by Hamster and Hupkes in case of analytic semigroups. We expect that our approach can also be applied to traveling waves and other patterns in more general situations such as systems of SPDEs with linearizations only generating a strongly continuous semigroup. This provides a relevant generalization as these systems are prevalent in many applications.
\end{abstract}
\maketitle
\tableofcontents
%

%%%%%%%%%%%%%%%%%%%%%%%%%%%%%%%%%%%%%%%%%%%
\section{Introduction}
%%%%%%%%%%%%%%%%%%%%%%%%%%%%%%%%%%%%%%%%%%%
\subsection{The stochastic FitzHugh-Nagumo equations}
%%%%%%%%%%%%%%%%%%%%%%%%%%%%%%%%%%%%%%%%%%%
We consider the \emph{stochastic FitzHugh-Nagumo equations}
\begin{subequations}\label{sfhn}
 \begin{align}
  \d \tilde u(t,x) &= \left(\nu \partial_x^2 \tilde u(t,x) + f\left(\tilde u(t,x)\right) - \tilde v(t,x)\right) \d t + \sigma \, \d W(t,x) \quad \mbox{for} \quad (t,x) \in \R_+ \times \R, \label{sfhn_1}\\
  \d \tilde v(t,x) &= \eps \left(\tilde u(t,x) - \gamma \tilde v(t,x)\right) \d t \quad \mbox{for} \quad (t,x) \in \R_+ \times \R, \label{sfhn_2}
 \end{align}
\end{subequations}
in which the independent variables $t$ and $x$ denote time and position on a neural axon, respectively. The dependent variable $\tilde u$ denotes the \emph{electric potential} and $\tilde v$ is a \emph{gating variable}. The parameter $\nu > 0$ determines the strength of the \emph{diffusion} $\partial_x^2 \tilde u$, while the nonlinearity $f\left(\tilde u\right)$ is a reaction term which typically has the form $f\left(\tilde u\right) = \chi (\tilde u) \, \tilde u \left(1 - \tilde u\right) \left(\tilde u - a\right)$ with $0 < a < 1$ and is suitably cut off by the factor $\chi$. The parameter $\sigma > 0$ determines the strength of the noise $W$. Here, we assume that $W$ is an infinite-dimensional Wiener process taking values in a Hilbert space to be specified in what follows. The parameter $\eps > 0$ determines the strength of the coupling of the electric potential $\tilde u$ to the gating variable $\tilde v$ and is assumed to be sufficiently small. The parameter $\gamma > 0$ determines the decay of the gating variable $\tilde v$.\medskip 

The \emph{classical FitzHugh-Nagumo}~\cite{FitzHugh,Nagumo} \emph{partial differential equations} (PDEs) obtained for $\sigma=0$ form a simplified, yet qualitatively very similar, model for the \emph{Hodgkin-Huxley equations}~\cite{HodgkinHuxley4}, which was a key part of Hodgkin's and Huxley's Nobel prize awarded in 1963. By now, the FitzHugh-Nagumo system is a standard model for the generation and transmission of electrical signals in neuroscience~\cite{ErmentroutTerman,Izhikevich1}. The literature on the PDE version of the FitzHugh-Nagumo system ($\sigma=0$) is very large, see for example the recent papers~\cite{CarterSandstede,GuckenheimerKuehn1} and detailed references therein. For the ODE version ($\sigma=0$, $\nu=0$), the literature is vast~\cite{GGR}, mostly due to the crucial role played by bistable nonlinearities in all areas of nonlinear science and the commonly found multiple time scale structure of the FitzHugh-Nagumo system~\cite{KuehnBook}. Also the SODE variant for $\nu=0$ is quite well-studied, mainly due to a flurry of activity since the mid 1990s; see e.g.~\cite{BashkirtsevaRyashko,BerglundGentzKuehn1,BerglundLandon,Lindneretal,LindnerSchimansky-Geier,MuratovVanden-Eijnden}. Yet, the full SPDE variant~\eqref{sfhn} has only attracted major attention quite recently, including a large number of numerical studies~\cite{LordPowellShardlow,SauerStannat,SauerStannat2,Shardlow,Tuckwell1,Tuckwell2,TuckwellRodriguez} as well as analytical studies regarding existence, regularity, invariant measures and attractors~\cite{BerglundKuehn,BonaccorsiMastrogiacomo,KuehnNeamtuPein,LiYin,BatesLuWang,Wang4}. The question regarding stochastic stability of pulses for additive noise is far less studied. We refer to \cite{HamsterHupkes2019} for multiplicative noise with a regularized equation (diffusion in the second variable) and to the review~\cite{KuehnSPDEwaves} for the stochastic Nagumo case ($\eps=0$). For further biophysical motivation regarding various noise terms in the FitzHugh-Nagumo equation we refer to the review~\cite{Lindneretal}.

\medskip 

Here we contribute to a more detailed understanding of stochastic pulse stability of the FitzHugh-Nagumo equations \eqref{sfhn} exploiting the \emph{multiscale} nature of the problem. More precisely, our aim is to understand the dynamics of \eqref{sfhn} near a deterministically stable pulse in the regime, where the parameters $\sigma > 0$ (strength of the noise) and $\eps > 0$ (coupling to the gating variable) are small. The subsequent analysis generalizes the recent analysis of Kr\"uger and Stannat \cite{KruegerStannat2017} for corresponding scalar stochastic bistable reaction-diffusion equations such as the Nagumo equation. We remark that the idea of tracking small noise fluctuations for SPDEs around traveling waves via a multiscale SODE approximation goes back at least to the early 1980s and works by Ebeling, Mikhailov, and Schimansky-Geier \cite{MikhailovSchimanskyGeierEbeling,SchimanskyGeierMikhailovEbeling}. From the viewpoint of applications the extension from Nagumo to FitzHugh-Nagumo is a crucial generalization as the FitzHugh-Nagumo model \eqref{sfhn} is far more realistic than the one studied in \cite{KruegerStannat2017} since only \eqref{sfhn} allows for deterministically stable \emph{traveling-pulse solutions}. In fact, deterministically stable localized pulses model far better the real action potentials generated in neurons in comparison to the deterministically stable traveling fronts appearing in the Nagumo equation. From a mathematical viewpoint, our generalization is important as it does extend beyond the setting of treating the equation in a Hilbert space in which the linearization of the traveling wave is self-adjoint (this is the case in \cite{KruegerStannat2017}) or generates an analytic semigroup and thereby makes the methods of multiscale approximation available to a broad class of stochastic SPDE-ODE reaction-diffusion systems. 

%%%%%%%%%%%%%%%%%%%%%%%%%%%%%%%%%%%%%%%%%%%
\subsection{Traveling-pulse solutions and their stability\label{sec:tw}}
%%%%%%%%%%%%%%%%%%%%%%%%%%%%%%%%%%%%%%%%%%%
We recall some of the well-known results on existence of traveling-pulse solutions to the deterministic version of \eqref{sfhn} where $\sigma = 0$. These solutions have the form $\tilde u(t,x) = \hat u(\xi)$ and $\tilde v(t,x) = \hat v(\xi)$, where $\xi = x + s t$ and $s \in \R$ is the velocity of the traveling wave. The tuple $\left(\hat u, \hat v\right)$ therefore fulfills the set of equations
\begin{subequations}\label{fhn_tw}
 \begin{align}
  \tfrac{\d \hat u}{\d \xi} &= \hat u^\prime \quad \mbox{for} \quad \xi \in \R, \\
  \tfrac{\d \hat u^\prime}{\d \xi} &= \frac 1 \nu \left(s \hat u^\prime - f\left(\hat u\right) + \hat v\right) \quad \mbox{for} \quad \xi \in \R, \\
  \tfrac{\d \hat v}{\d \xi} &= \frac{\eps}{s} \left(\hat u - \gamma \hat v\right) \quad \mbox{for} \quad \xi \in \R.
 \end{align}
\end{subequations}
Traveling-pulse solutions to \eqref{fhn_tw} are solutions $\left(s, \hat u, \hat v\right)^{\trans}$ such that $\left(\hat u, \hat u^\prime, \hat v\right)^{\trans} \to (b_1,b_2,b_3)^{\trans}$ as $\xi \to \pm \infty$ where $(b_1,b_2,b_3)^{\trans} \in \R^3$ is a stationary solution to \eqref{fhn_tw} fulfilling
\[
b_2 = 0, \quad f(b_1) = b_3, \quad \mbox{and} \quad b_1 = \gamma b_3, 
\]
i.e., they are homoclinic orbits of the dynamical system \eqref{fhn_tw}. In what follows we will also have the convention that we mean non-trivial traveling pulses. Furthermore, we are only interested in solutions $\left(\hat u, \hat v\right)^{\trans}$ such that $\left(\hat u, \hat v\right)^{\trans} \to (0,0)^{\trans}$ as $\xi \to \pm \infty$. Indeed, this situation will occur if the equilibrium point $(b_1,b_2,b_3)^{\trans} = (0,0,0)^{\trans}$ is unique, which is the case e.g.~if $\gamma \ge 0$ is sufficiently small~\cite{GGR}. Further note that the velocity $s$ of the traveling wave is not a parameter but a functional of $f$ and $\eps$.

\medskip

Next, we give a very brief overview of the existing literature on the existence of homoclinic orbits of \eqref{fhn_tw} corresponding to traveling-pulse solutions. For a wave speed $s=0$ and $\eps=0$, we recover the planar ODE associated to traveling waves of the Nagumo equation. Using the resulting Hamiltonian structure of the ODE, it is easy to see that a homoclinic orbit exists for $s=0$ and $\eps=0$. A singular perturbation argument in combination with Melnikov's method~\cite{KuehnBook,Szmolyan1,Hastings1976} yields the existence of a \emph{slow pulse} with wave speed $s\approx 0$. Yet, an application of Sturm-Liouville theory~\cite{KapitulaPromislow2013,KuehnBook1} shows that the slow pulse is unstable. As it is deterministically already unstable, considering this pulse under the influence of noise is not expected to be biophysically relevant as the noisy small perturbations will be amplified exponentially near the slow pulse. Yet, there is also a \emph{fast pulse} corresponding to much higher wave speeds $s=s(f,\eps)$. Carpenter~\cite{Carpenter1974} and Conley~\cite{Conley1975} constructed these homoclinic orbits to \eqref{fhn_tw} employing the method of isolating blocks of the fast and slow subsystems~\cite{ConleyEaston1971}. See also~\cite{GardnerSmoller1983}, where the methods developed in~\cite{Conley1975} have been further improved. Then a generalization of the pulse construction to large classes of FitzHugh-Nagumo-like models has been provided in~\cite{EvansFenichelFeroe}. Later, a fully geometric construction of the fast pulse via the Exchange Lemma~\cite{JonesKopell,JonesKaperKopell,KuehnBook} was proved by Jones, Kopell and Langer~\cite{JonesKopellLanger} using differential forms. The connection in parameter space between slow and fast pulses has been proved in~\cite{KrupaSandstedeSzmolyan}. Then the parametric bifurcation structure has been analyzed in a more refined way in~\cite{CarterSandstede,GuckenheimerKuehn1,GuckenheimerKuehn3,Sneydetal}. A non-perturbative approach to construct traveling-pulse solutions of \eqref{fhn_tw} has been carried out by Arioli and Koch~\cite{ArioliKoch2015} for the value $\eps = 0.01$ using computer-assisted proofs.

\medskip

Stability of the fast traveling-pulse solutions to the deterministic version of \eqref{sfhn} with $\sigma = 0$ in the space of bounded uniformly continuous functions has been obtained by Jones~\cite{Jones1984} for the case  of a cubic polynomial, in the sense that solutions starting sufficiently close to a wave profile decay to a translate of
\[
\left(\hat u(\cdot+s t), \hat v(\cdot+s t)\right)^{\trans}.
\]
The proof relies on analysis developed by Evans~\cite{Evans1971_72,Evans1972_73,Evans1972_73_2,Evans1974_75,Evans1976,Evans1976_2}, where it is proved that in the space of bounded uniformly continuous functions that linear stability implies nonlinear stability and that the point spectrum of the linearization around the traveling pulse is determined by the roots of a function $D(\lambda)$, the \emph{Evans function}. Jones proves that instability can only occur due to eigenvalues of the linear operator near $\lambda = 0$. By calculating the winding number of $D(\lambda)$ for a small circle around $\lambda = 0$, it is shown that only two eigenvalues lie in it, one of which is $\lambda = 0$ (related to the translation invariance of the problem) and the other is negative because $\frac{\d D}{\d \lambda}(0) > 0$. For a more recent stream-lined stability analysis allowing for more general reaction terms $f$ but still restricted to the space of bounded uniformly continuous functions, we refer to \cite{Yanagida1985,ArioliKoch2015} while stability in $L^2(\R;\R^2)$ and $H^1(\R;\R^2)$ is proved in \cite{GhazaryanLatushkinSchecter,RottmannMatthes2010,Yurov} (see \S\ref{sec:lin_intro} for further details). Nonlinear stability of the fast FitzHugh-Nagumo pulse with oscillatory tails (instead of monotone tails as considered in this paper) has been proved by Carter, de Rijk, and Sandstede in \cite{CarterdeRijkSandstede}. For general introductions into the subject, we refer to \cite{ChiconeLatushkin,Sandstede2002}.

%%%%%%%%%%%%%%%%%%%%%%%%%%%%%%%%%%%%%%%%%%%
\subsection{An approach for computing the velocity correction\label{sec:ex_approaches}}
%%%%%%%%%%%%%%%%%%%%%%%%%%%%%%%%%%%%%%%%%%%
Our aim is to investigate the following decomposition
\begin{subequations}\label{decomp_sfhn}
 \begin{align}
  \tilde u(t,x) &= \hat u\left(x+st+\varphi(t)\right) + u_\varphi(t,x), \\
  \tilde v(t,x) &= \hat v\left(x+st+\varphi(t)\right) + v_\varphi(t,x)
 \end{align}
\end{subequations}
of solutions to \eqref{sfhn}, where the function $\varphi(t)$ is a random correction to the position of the wave front, and $u_\varphi(t,x)$ and $v_\varphi(t,x)$ denote lower-order fluctuations that are uniquely defined through \eqref{decomp_sfhn} for any choice of $\varphi = \varphi(t)$. Ideally, we would like to choose $\varphi$ to minimize the distance in the direction of the traveling wave between the solution $\tilde X := (\tilde u, \tilde v)^\trans$ of the FitzHugh-Nagumo SPDEs \eqref{sfhn} and the suitably translated traveling wave $\hat X = (\hat u, \hat v)^\trans$, i.e.,
\begin{equation}\label{def_phi}
 \varphi(t) \in \argmin_{\varphi \in \R}\vertii{\Pi_{st+\varphi}^0 \left(\tilde X(t,\cdot) - \hat X(\cdot + s t + \varphi)\right)}_H^2, \quad \mbox{with} \quad \tilde X := \begin{pmatrix} \tilde u \\ \tilde v \end{pmatrix}, \quad \hat X := \begin{pmatrix} \hat u \\ \hat v \end{pmatrix},
\end{equation}
where $\vertii{\cdot}_H$ is the norm in the spatial variable of a suitable underlying Hilbert space $H$ and $\Pi_{st+\varphi}^0$ is a suitable projection operator onto the traveling wave such that $\Pi^0_{st+\varphi} \tfrac{\d \hat X}{\d\xi}(\cdot+st+\varphi) = \tfrac{\d \hat X}{\d\xi}(\cdot+st+\varphi)$ (see Proposition~\ref{prop:frozen}~\eqref{item:riesz} and \eqref{pi_st_phi} further below). However, as the minimization problem \eqref{def_phi} is not necessarily convex, uniqueness of a minimizer is not ensured. We follow the approach in \cite{KruegerStannat2017} and replace \eqref{def_phi} by the weaker condition for a critical point of finding $\varphi = \varphi(t)$ such that
\begin{equation}\label{def_phi_2}
0 = \inner{\Pi^0_{st+\varphi} \left(\tilde X(t,\cdot) - \hat X(\cdot+st+\varphi)\right)}{\tfrac{\d \hat X}{\d\xi}(\cdot+st+\varphi)}{H},
\end{equation}
where $\inner{\cdot}{\cdot}{H}$ denotes the inner product of $H$. This approach has been employed by Inglis and MacLaurin in \cite{InglisMacLaurin2016} for more general classes of SPDE systems with the drawback that results only hold up to the first stopping time when the local minimum turns into a saddle. Here, we follow the work around proposed in \cite{KruegerStannat2017} in the sense that $\varphi(t)$ is approximated by a process $\varphi^m(t)$, which in our case fulfills the random ordinary differential equation (RODE)
\begin{equation} \label{eq_phi_ks}
 \tfrac{\d \varphi^m}{\d t}(t) = m \inner{\Pi^0_{st+\varphi^m(t)} \left(\tilde X(t,\cdot) - \hat X(\cdot+st+\varphi^m(t))\right)}{\tfrac{\d \hat X}{\d\xi}\left(\cdot+st+\varphi^m(t)\right)}{H}
\end{equation}
for given initial condition and a relaxation parameter $m > 0$ that is chosen sufficiently large. Notably, approximating $\varphi$ with $ \varphi^m$ through \eqref{eq_phi_ks} implies differentiability while this is in general not true for \eqref{def_phi_2}. By further analyzing \eqref{eq_phi_ks} we will construct $\varphi$ as a solution of an SODE and $\left(u_{\varphi}, v_{\varphi}\right)^{\trans}$ as the solution of a system of an SPDE coupled to an SODE in such a way that $\left(u_{\varphi}, v_{\varphi}\right)^{\trans}$ fulfill estimates in suitable function spaces, making precise what the notion `lower order' means.

\medskip

In \cite[\S3.4]{KruegerStannat2017}, it has been pointed out that with the strategy described above the first-order fluctuations around the traveling wave are orthogonal in their suitably chosen inner product. In fact, in the case of of second-order bistable reaction-diffusion equations, the spatial linearization around the traveling wave, the \emph{frozen-wave operator}, is of Sturm-Liouville type. Hence, one can always find a weighted $L^2$-inner product in which it is self-adjoint, so that one can replace the projection operator used in \eqref{def_phi} with an orthogonal projection induced by this inner product. One cannot expect a similar approach to be applicable in our situation \eqref{sfhn} of an SPDE coupled to an ODE or other more complicated systems of SPDEs. Instead, we will build up our analysis on Riesz spectral projections of the frozen-wave operator that do not require a self-adjoint structure and yield a partition into two subspaces invariant under the linearized flow. Note that non-orthogonal Riesz spectral projections have also been employed by Hamster and Hupkes in \cite{HamsterHupkes2019,HamsterHupkes2020SIMA,HamsterHupkes2020SIADS,HamsterHupkes2020PhysD} by projecting onto the eigenvector of the adjoint of the frozen-wave operator. Their SODE to determine $\varphi$ is more involved compared to \eqref{eq_phi_ks}. Their method applies to a variety of (systems) of reaction-diffusion equations with special forms of multiplicative noise (partially only including a one-dimensional Wiener process). However, in all situations treated there, the frozen-wave operator is sectorial with spectral angle larger than $\frac \pi 2$ and therefore generates an analytic semigroup. Specifically, in \cite[(1.2)]{HamsterHupkes2019}, \cite[(1.3)]{HamsterHupkes2020SIMA}, \cite[(1.17)]{HamsterHupkes2020PhysD} the second component of the FitzHugh-Nagumo system is regularized by adding $c \partial_x^2 \tilde v$ with some $c > 0$ to the right-hand side of \eqref{sfhn_2} (see \cite{ChenChoi2015,CornwellJones} for existence and stability of the pulse in this case), which is different from the only partly parabolic FitzHugh-Nagumo system without diffusion in the second component as treated for instance in \cite{Carpenter1974,Conley1975,GardnerSmoller1983,JonesKopellLanger,Jones1984,Yanagida1985,ArioliKoch2015}. In our setting, the frozen-wave operator has essential spectrum parallel to the imaginary axis (see \S\ref{sec:lin_intro} and \S\ref{sec:lin_proof} below) and therefore is so far not covered by this approach.

\medskip

Historically, one can track back stability and fluctuation analysis of bistable scalar SPDEs, such as the Nagumo SPDE, at least to early works by Ebeling, Mikhailov, and Schimansky-Geier in \cite{MikhailovSchimanskyGeierEbeling,SchimanskyGeierMikhailovEbeling}, where it was recognized that the deterministic reference wave speed should be corrected by a stochastic term, which in turn satisfies an SODE. Many further works followed, e.g., using a more rigid/frozen stochastic frame in combination with numerical simulations by Lord and Th\"ummler in \cite{LordThuemmler2012}, employing functional inequalities to establish stability bounds by Stannat in \cite{Stannat2014}, the adaptation of a rigorous multiscale expansion by Kr\"uger and Stannat in \cite{KruegerStannat2017} originally motivated by work on the related problem of stochastic traveling waves for bistable neural field equations by the same authors in \cite{KruegerStannat2014} and by Inglis and MacLaurin in \cite{InglisMacLaurin2016}, and work on long-time stochastic stability tracking for suitable multiplicative noise by Hamster and Hupkes in \cite{HamsterHupkes2019} via stochastic convolution estimates by the same authors in \cite{HamsterHupkes2020SIADS}. We also remark that in case of stochastic dispersive PDEs, the random modulation of soliton solutions (i.e., standing or traveling wave packages) has been investigated by de Bouard and Debussche for the stochastic Korteweg-de Vries equation in \cite{BouardDebussche2007} and by de Bouard and Fukuizumi for the Gross-Pitaevskii equation in \cite{BouardFukuizumi2009}. Furthermore, multiscale expansions also frequently appear in the context of SPDE amplitude equations at bifurcation points as treated for instance by Bl\"omker in \cite{Bloemker} and by Bl\"omker, Hairer, and Pavliotis in \cite{BloemkerHairerPavliotis}.

%%%%%%%%%%%%%%%%%%%%%%%%%%%%%%%%%%%%%%%%%%%
\subsection{Outline}
%%%%%%%%%%%%%%%%%%%%%%%%%%%%%%%%%%%%%%%%%%%
We continue with the setting and auxiliary results in \S\ref{sec:set_prel}, followed by the main results in \S\ref{sec:main}. The proofs of auxiliary results are contained in Appendix~\ref{sec:proofs_prel} while the proofs of the main results can be found in \S\ref{sec:proofs_main}. In \S\ref{sec:evol} and \S\ref{sec:var_construct} we construct solutions to \eqref{sfhn} using the variational approach for equations with locally monotone coefficients~\cite{LiuRoeckner2010,LiuRoeckner2015}. In \S\ref{sec:lin_intro} and \S\ref{sec:lin_proof}, we then formulate the SPDE of perturbations around the traveling wave which to leading-order is governed by the linearization around the deterministically-stable fast FHN pulse. Results on the deterministic linearized evolution of perturbations around the traveling wave are provided in Proposition~\ref{prop:family} and Proposition~\ref{prop:frozen} below. Afterwards, in \S\ref{sec:corr_intro} and \S\ref{sec:phase}, we derive an SODE approximating the correction of the wave velocity (cf.~Proposition~\ref{prop:path_ode} below). The leading-order part of this SODE is an Ornstein-Uhlenbeck-type process with a linear damping in the drift due to the relaxation method of the frame and additive stochastic fluctuations obtained from projecting the infinite-dimensional noise onto the deterministic translation-invariant mode. Practically, this entails that the deterministic reference wave has phase diffusion along the translation direction. Subsequently, in \S\ref{sec:reduced_intro} and \S\ref{sec:leading}, we prove a multiscale expansion in terms of the linearized evolution (cf.~Theorem~\ref{th:multi} below), which is further investigated in  \S\ref{sec:immediate_intro} and \S\ref{sec:immediate} in the limit as $m \to \infty$ of immediate relaxation (cf.~Theorem~\ref{th:immediate} and Proposition~\ref{prop:moment} below). In particular, our results yield bounds on
\begin{itemize}
\item the first exit time where the multiscale decomposition cannot be guaranteed to hold anymore and 
\item on the second moment of fluctuations transverse to the traveling wave mode after correcting the wave velocity.
\end{itemize}
Concluding remarks and an outlook on future research can be found in \S\ref{sec:conclusions}.

%%%%%%%%%%%%%%%%%%%%%%%%%%%%%%%%%%%%%%%%%%%
\section{Setting and auxiliary results\label{sec:set_prel}}
%%%%%%%%%%%%%%%%%%%%%%%%%%%%%%%%%%%%%%%%%%%
For what follows, we fix a \emph{stochastic basis}, that is, a \emph{complete filtered probability space}
\[
\left(\Omega,\cF,(\cF_t)_{t \in [0,T]},\PP\right),
\]
with a complete and right-continuous filtration $(\cF_t)_{t \in [0,T]}$, where $T \in (0,\infty)$ is arbitrary.

%%%%%%%%%%%%%%%%%%%%%%%%%%%%%%%%%%%%%%%%%%%
\subsection{Existence and uniqueness of solutions using the variational approach\label{sec:evol}}
%%%%%%%%%%%%%%%%%%%%%%%%%%%%%%%%%%%%%%%%%%%
In line with the classical findings for the deterministic FitzHugh-Nagumo PDE discussed above, we make the following assumption on \eqref{sfhn} or \eqref{fhn_tw}, respectively. In particular, we are only going to study stochastic perturbations to the deterministically stable fast pulse solution:
%%%%%%%%%%%%%%%%%%%%%%%%%%%%%%%%%%%%%%%%%%%
\begin{globalassumption}\label{ass:tw}
In a right-neighborhood of $\eps = 0$ the system \eqref{fhn_tw} has a non-trivial homoclinic orbit of the point $(0,0,0)^\trans$ in phase space. The corresponding traveling-pulse solution $\left(\hat u, \hat v\right)^\trans$ of the FitzHugh-Nagumo equations \eqref{sfhn} with $\sigma = 0$ is locally asymptotically stable up to translation.
\end{globalassumption}
%%%%%%%%%%%%%%%%%%%%%%%%%%%%%%%%%%%%%%%%%%%
We will assume certain properties on the reaction term $f(w)$ that are fulfilled for instance by the choice $f(w) = \chi(w) \, w (1-w) (w-a)$, where $a \in (0,1)$ and $\chi \in C^\infty(\R)$ meets $\chi_{|[-c_1,\infty)} \equiv 1$ and $\chi(w) = \frac{c_2^2}{w^2}$ for $w \in (-\infty,-c_2]$, where $1 < c_1 < c_2$ sufficiently large (cut-off at $-\infty$), for which existence of a traveling-pulse solution (cf.~Global~Assumption~\ref{ass:tw}) is guaranteed provided $\eps > 0$ is small. Here, we directly make the usual abstract bi-stability assumptions~\cite{Chen1,Yanagida1985,KruegerStannat2017} for $f$ so that the nonlinearity effectively behaves like the classical cubic nonlinearity chosen for the Nagumo and FitzHugh-Nagumo equations.

%%%%%%%%%%%%%%%%%%%%%%%%%%%%%%%%%%%%%%%%%%%
\begin{globalassumptions}\label{ass:reaction}
We have $f \in C^3(\R)$ and there exist $a \in (0,1)$ and $d > 0$ with 
 \begin{subequations}
  \begin{align}
   f(0) &= f(a) = f(1) = 0, \\
   f(w) &< 0 \quad \mbox{for} \quad w \in (0,a) \cup (1,\infty), \\
   f(w) &> 0 \quad \mbox{for} \quad w \in (-\infty,0) \cup (a,1), \\
   f'(0) &< 0, \quad f'(a) > 0, \quad f'(1) < 0, \label{sign_f} \\
   f(w) - \frac w d &\ne 0 \quad \mbox{for} \quad w \in \R \setminus \{0\}, \\
   \int_0^1 f(w) \, \d w &> 0, \\
   \eta_1 &:= \sup_{w \in \R} f'(w) < \infty, \label{cond_f_prime} \\
   \verti{f(w_1 + w_2) - f(w_1) - f'(w_1) w_2} &\le \eta_2 \left(1+\verti{w_1}+\verti{w_2}\right) \verti{w_2}^2 \quad \mbox{for} \quad w_1, w_2 \in \R, \label{cond_f_diff_2} \\
   \verti{f'(w)} &\le \eta_3 \left(1+\verti{w}^2\right) \quad \mbox{for} \quad w \in \R, \label{cond_f_diff_3} \\
   \verti{f(w_1) - f(w_2)} &\le \eta_4 \verti{w_1-w_2} \left(1+\verti{w_1}^2+\verti{w_2}^2\right) \quad \mbox{for} \quad w_1, w_2 \in \R, \label{cond_f_diff_4} \\
   \verti{f'(w_1 + w_2) - f'(w_1) - f''(w_1) w_2} &\le \eta_5 \verti{w_2}^2 \quad \mbox{for} \quad w_1, w_2 \in \R, \label{cond_f_diff_5} \\
   \verti{f'(w_1) - f'(w_2)} &\le \eta_6 \verti{w_1 - w_2} \left(1+\verti{w_1}+\verti{w_2}\right) \quad \mbox{for} \quad w_1, w_2 \in \R, \label{cond_f_diff_6} \\
   \verti{f''(w_1) - f''(w_2)} &\le \eta_7 \verti{w_1 - w_2} \quad \mbox{for} \quad w_1, w_2 \in \R, \label{cond_f_diff_7}
  \end{align}
 \end{subequations}
 where $\eta_2, \eta_3, \eta_4, \eta_5, \eta_6, \eta_7 < \infty$.
\end{globalassumptions}
%%%%%%%%%%%%%%%%%%%%%%%%%%%%%%%%%%%%%%%%%%%
Note that the linearization of \eqref{fhn_tw} in $(0,0,0)^\trans$ only depends on $f'(0)$ and because of \eqref{sign_f} this point is hyperbolic and hence $\tfrac{\d^j \hat u}{\d\xi^j}$ for $j \in \{0,1,2,3,4,5\}$ and $\tfrac{\d^j\hat v}{\d\xi^j}$ for $j \in \{0,1,2,3,4\}$ are exponentially decaying as $\verti{\xi} \to \pm \infty$. In particular,
\[
\vertii{\tfrac{\d^j \hat u}{\d\xi^j}}_{L^\infty(\R)} < \infty \quad \mbox{and} \quad\vertii{\tfrac{\d^j \hat u}{\d\xi^j}}_{L^2(\R)} < \infty \quad \mbox{for all} \quad j \in \{0,1,2,3,4,5\},
\]
as well as
\[
\vertii{\tfrac{\d^j \hat v}{\d\xi^j}}_{L^\infty(\R)} < \infty \quad \mbox{and} \quad \vertii{\tfrac{\d^j \hat v}{\d\xi^j}}_{L^2(\R)} < \infty \quad \mbox{for all} \quad j \in \{0,1,2,3,4\}.
\]
For $\varphi \equiv 0$, we may use \eqref{decomp_sfhn} as a definition of $\left(u, v\right) := (u_0,v_0)$ and employing \eqref{fhn_tw}, we obtain the system
\begin{subequations}\label{spde_check}
 \begin{eqnarray}
  \d u(t,x) &=& \left(\nu \partial_x^2 u(t,x) + f\left(u(t,x) + \hat u(x + s t)\right) - f\left(\hat u(x + s t)\right) - v(t,x)\right) \d t \nonumber \\
  && + \sigma \, \d W(t,x), \\
  \d v(t,x) &=& \eps \left(u(t,x) - \gamma v(t,x)\right) \d t.
 \end{eqnarray}
\end{subequations}
In order to analyze \eqref{spde_check}, we use the following Hilbert space setting. Suppose that $L^2(\R)$ is the standard $L^2$-space of $\R$-valued functions on the real line,
\[
H^1(\R) := \left\{u \in L^2(\R): \, \partial_x u \in L^2(\R)\right\},
\]
where $\partial_x u$ denotes the distributional derivative, and $H^{-1}(\R)$ denotes the topological dual of $H^1(\R)$ relative to $L^2(\R)$. The Laplacian $\partial_x^2 u$ for $u \in H^1(\R)$ can be defined via its bilinear form as
\begin{equation}\label{def_laplacian}
 \dualpair{H^1(\R)}{w}{\partial_x^2 u}{H^{-1}(\R)} := - \int_\R (\partial_x w) \, (\partial_x u) \, \d x.
\end{equation}
Then, we may recast the system \eqref{spde_check} in form of the abstract stochastic evolution equation
\begin{equation}\label{stoch_evol}
 \d X(t,\cdot) = \cA(t,X(t,\cdot)) \, \d t + \cB(t,X(t,\cdot)) \, \d W_U(t,\cdot)
\end{equation}
and introduce the \emph{rigged space triple} (\emph{Gelfand triple})
\begin{equation}\label{rigged_alt}
V := H^1(\R) \eoperp L^2(\R) \hookrightarrow H := L^2(\R) \eoperp L^2(\R) = H^* \hookrightarrow V^* = H^{-1}(\R) \eoperp L^2(\R),
\end{equation}
where
\begin{subequations}\label{hv_norm}
\begin{align}
\inner{Y_1}{Y_2}{H} &:= Z \int_\R \left(\eps w_1 w_2 + q_1 q_2\right) \d x, \label{h_norm} \\
\inner{Y_1}{Y_2}{V} &:= Z \int_\R \left(\eps w_1 w_2 + \eps (\partial_x w_1) (\partial_x w_2) + q_1 q_2\right) \d x, \label{v_norm}
\end{align}
\end{subequations}
with
\begin{equation}\label{yj_z}
Y_j := \begin{pmatrix} w_j \\ q_j \end{pmatrix} \quad \mbox{and} \quad Z := \left(\int_\R \left(\eps \left(\tfrac{\d \hat u}{\d \xi}\right)^2 + \left(\tfrac{\d \hat v}{\d \xi}\right)^2\right) \d \xi\right)^{-1},
\end{equation}
and where 
\begin{subequations}\label{def_vectors}
 \begin{align}
  X &:= \begin{pmatrix} u \\ v \end{pmatrix}, \\
  \cA \colon [0,T] \times V \times \Omega &\to V^*, \nonumber \\
  (t,X,\omega) &\mapsto \cA(t,X) := \begin{pmatrix} \nu \partial_x^2 u + f\left(u + \hat u(\cdot + s t)\right) - f\left(\hat u(\cdot + s t)\right) - v \\ \eps \left(u - \gamma v\right) \end{pmatrix}, \label{def_a}\\
  \cB \colon [0,T] \times V \times \Omega &\to L_2\left(U;H\right), \nonumber \\
  (t,X,\omega) &\mapsto \left(U \to H, \quad N \mapsto \cB(t,X) N:= \begin{pmatrix} \sigma \sqrt Q N \\ 0 \end{pmatrix}\right). \label{def_b}
 \end{align}
\end{subequations}
Here, we assume that $U := L^2(\R)$, $Q \in L(U)$ is a symmetric nonnegative-definite bounded linear operator in $U$ with finite trace, $L_2\left(U;H\right)$ denotes the space of \emph{Hilbert-Schmidt operators} $U \to H$, and $W_U$ is a $U$-valued $(\cF_t)_{t \in [0,T]}$-adapted \emph{cylindrical $\id_U$-Wiener process}. Note that $W = \sqrt Q W_U$ is in fact a $U$-valued $(\cF_t)_{t \in [0,T]}$-adapted \emph{$Q$-Wiener process}. We remark that the assumption of $Q$ having finite trace leads to a noise term in \eqref{stoch_evol} which is not translation invariant. However, since we study stability of a pulse with exponential tails traveling with finite speed on an axon (which in reality has finite length), this assumption appears to be acceptable regarding the relevance of our results in terms of applications. Further note that the scaling of the first component with $\sqrt\eps$ in \eqref{hv_norm} changes the geometry of the orthogonal sums indicated by the symbol $\eoperp$ rather than $\operp$ in \eqref{rigged_alt}. On the other hand, the normalization with $Z$ merely implies the convenient property that $\vertii{\tfrac{\d\hat X}{\d\xi}}_H = 1$ but leaves all angles between vectors unchanged.

\medskip

In order to show existence of solutions to \eqref{stoch_evol}, one could use the concept of mild solutions as employed for the FitzHugh-Nagumo SPDE with additive noise in~\cite{KuehnNeamtuPein}. The approach presented in~\cite{KruegerStannat2017}, which we build upon, uses variational solutions~\cite{KrylovRozovskii1979,PrevotRoeckner,LiuRoeckner2015} for equations with locally monotone coefficients~\cite{LiuRoeckner2010,LiuRoeckner2015}. As we will show in Proposition~\ref{prop:reg_var} below, the variational solution (constructed in Proposition~\ref{prop:ex_var} below) also turns out to be a mild solution. We remark that the authors of \cite{HamsterHupkes2019,HamsterHupkes2020PhysD,HamsterHupkes2020SIMA} use \cite{LiuRoeckner2010} to construct solutions, too, but as previously mentioned, their system is regularized by adding $c \partial_x^2 u$ with $c > 0$ to the second component in \eqref{def_a}, which also changes the Gelfand triple to $H^1(\R;\R^2) \hookrightarrow L^2(\R;\R^2) \hookrightarrow H^{-1}(\R;\R^2)$. Since the variational approach has not been directly worked out for the FitzHugh-Nagumo SPDE with additive noise and no regularization of the second component, we use this opportunity to fill this gap within this work as an auxiliary step.  Therefore, we use the concept of \emph{variational solutions} (cf.~\cite[Definition~1.1]{LiuRoeckner2010} and \cite[Definition~5.1.2]{LiuRoeckner2015}):
%%%%%%%%%%%%%%%%%%%%%%%%%%%%%%%%%%%%%%%%%%%
\begin{definition}\label{def:variational}
A variational solution to \eqref{stoch_evol} is a continuous $H$-valued $(\cF_t)_{t \ge 0}$-adapted process $(X(t,\cdot))_{t \in [0,T]}$ such that the $\d t \otimes \PP$-equivalence class $\check X$ meets $\check X \in L^2\left([0,T] \times \Omega, \d t \otimes \PP; V\right)$ and such that the solution formula
\begin{equation}\label{var_formula}
X(t,\cdot) = X(0,\cdot) + \int_0^t \cA\left(t',\bar X(t',\cdot)\right) \d t' + \int_0^t \cB\left(t',\bar X(t',\cdot)\right) \d W_U(t',\cdot) \quad \mbox{for} \quad t \in [0,t]
\end{equation}
is satisfied, $\PP$-almost surely, where $\bar X$ denotes any $V$-valued progressively measurable $\d t \otimes \PP$-version of $\check X$.
\end{definition}
%%%%%%%%%%%%%%%%%%%%%%%%%%%%%%%%%%%%%%%%%%%
Under the given hypothesis, we can prove existence of solutions and regularity in space under additional assumptions:
%%%%%%%%%%%%%%%%%%%%%%%%%%%%%%%%%%%%%%%%%%%
\begin{proposition}\label{prop:ex_var}
For $p \in [6,\infty)$ and $T > 0$, the stochastic evolution equation \eqref{stoch_evol} has for any $X(0,\cdot) = \left(u(0,\cdot),v(0,\cdot)\right)^\trans = X^{(0)} = \left(u^{(0)},v^{(0)}\right)^\trans \in L^p\left(\Omega,\cF_{0},\PP;H\right)$ a unique variational solution. The solution further satisfies $X \in L^p\left(\Omega,\cF,\PP;C^0\left([0,T];H\right)\right)$.
\end{proposition}
%%%%%%%%%%%%%%%%%%%%%%%%%%%%%%%%%%%%%%%%%%%

For the subsequent result, we define the stronger space
\begin{subequations}\label{strong_v}
\begin{equation}
\cV := H^1(\R) \eoperp H^1(\R)
\end{equation}
with inner product
\begin{eqnarray}\nonumber
\inner{Y_1}{Y_2}{\cV} &:=& Z \int_\R \left(\eps w_1 w_2 + \eps (\partial_x w_1) (\partial_x w_2) + q_1 q_2 + (\partial_x q_1) (\partial_x q_2)\right) \d x \\
&\stackrel{\eqref{h_norm}}{=}& \inner{Y_1}{Y_2}{H} + \inner{\partial_x Y_1}{\partial_x Y_2}{H}, \label{sv_norm}
\end{eqnarray}
\end{subequations}
where $Y_j$ and $Z$ are as in \eqref{yj_z}.
%%%%%%%%%%%%%%%%%%%%%%%%%%%%%%%%%%%%%%%%%%%
\begin{proposition}\label{prop:reg_var}
In the situation of Proposition~\ref{prop:ex_var}, the solution $X$ is also a mild solution in the sense of \cite{DaPratoZabczyk} with linear operator
\[
\begin{pmatrix} \nu \partial_x^2 & 0 \\ 0 & - \eps \gamma \end{pmatrix} \colon H \supseteq H^2(\R) \eoperp L^2(\R) \to H.
\]
If additionally $\sqrt Q \in L_2(U;H^1(\R))$, $u^{(0)} \in L^2(\R)$, and $v^{(0)} \in H^1(\R)$, then
\[
t u, v \in C^0\left([0,T];H^1(\R)\right), \quad \mbox{$\PP$-almost surely.}
\]
If further $u^{(0)} \in H^1(\R)$, i.e., $X^{(0)} \in \cV$, then $X \in C^0\left([0,T];\cV\right)$, $\PP$-almost surely.
\end{proposition}
%%%%%%%%%%%%%%%%%%%%%%%%%%%%%%%%%%%%%%%%%%%

We will give the proof of Proposition~\ref{prop:ex_var} and Proposition~\ref{prop:reg_var} in \S\ref{sec:var_construct}.

%%%%%%%%%%%%%%%%%%%%%%%%%%%%%%%%%%%%%%%%%%%
\subsection{Linearization around the traveling wave\label{sec:lin_intro}}
%%%%%%%%%%%%%%%%%%%%%%%%%%%%%%%%%%%%%%%%%%%
We recall the global Assumptions~\ref{ass:tw} and write $\hat X := \left(\hat u, \hat v\right)^\trans$ for the deterministically stable fast traveling pulse. Then we define
\begin{equation}\label{def_txxhx}
\tilde X(t,x) := \left(\tilde u(t,x),\tilde v(t,x)\right)^\trans = X(t,x) + \hat X(x + s t),
\end{equation}
and in line with \eqref{decomp_sfhn} set
\begin{eqnarray}\nonumber
X_\varphi(t,x) &:=& \left(u_\varphi(t,x),v_\varphi(t,x)\right)^\trans := \tilde X(t,x) - \hat X(x + s t + \varphi) \\
&=& X(t,x) + \hat X(x + s t) - \hat X(x + s t + \varphi), \label{x_phi}
\end{eqnarray}
where $\varphi = \varphi(t)$ is yet to be determined. We split the stochastic evolution equation \eqref{stoch_evol} around the traveling-wave profile $\hat X$ into a linear and nonlinear part using \eqref{fhn_tw}. This yields
\begin{equation}\label{evol_lin_tw}
 \d X_\varphi(t,\cdot) = \left(\cL_{s t + \varphi(t)} X_\varphi(t,\cdot) + R_{\varphi(t)}(t,X_\varphi(t,\cdot),\cdot) - \dot\varphi(t) \tfrac{\d \hat X}{\d \xi}(\cdot+st+\varphi(t))\right) \d t + \begin{pmatrix} \sigma \\ 0 \end{pmatrix} \d W(t,\cdot)
\end{equation}
where
\begin{equation}\label{lin_op}
 \cL_{st+\varphi} Y := \begin{pmatrix} \nu \partial_x^2 w + f'\left(\hat u(\cdot+s t+\varphi)\right) w - q \\ \eps \left(w - \gamma q\right) \end{pmatrix}, \quad Y := \begin{pmatrix} w \\ q \end{pmatrix},
\end{equation}
and
\begin{equation}\label{remainder}
 R_\varphi(t,Y,\cdot) := \begin{pmatrix} f\left(w + \hat u(\cdot+s t + \varphi)\right) - f\left(\hat u(\cdot+s t + \varphi)\right) - f'\left(\hat u(\cdot + s t + \varphi)\right) w \\ 0 \end{pmatrix}.
\end{equation}
We have the following result, which is proved in \S\ref{sec:family}.
%%%%%%%%%%%%%%%%%%%%%%%%%%%%%%%%%%%%%%%%%%%
\begin{proposition}[linearized evolution]\label{prop:family}
The family of operators $\left(\cL_{st}\right)_{t \ge 0}$, with 
\[
\cL_{st} \colon D\left(\cL_{st}\right) = H^3(\R) \eoperp H^1(\R) \to \cV \stackrel{\eqref{strong_v}}{=} H^1(\R) \eoperp H^1(\R),
\]
generates an evolution family $\left(P_{st,st'}\right)_{t \ge t' \ge 0}$ in $\cV$ with
\begin{subequations}\label{bound_evol_fam}
\begin{equation}\label{bound_pstst}
\vertii{P_{st,st'}}_{L(\cV)} \le e^{\beta (t-t')},
\end{equation}
where
\begin{equation}\label{def_beta}
\beta := \vertii{f'(\hat u) - f'(0)}_{W^{1,\infty}(\R)} - \min\left\{- f'(0),\eps \gamma\right\},
\end{equation}
\end{subequations}
and we have used the norm in the $\xi=x+st$ coordinate in the last expression.
\end{proposition}
%%%%%%%%%%%%%%%%%%%%%%%%%%%%%%%%%%%%%%%%%%%
We also write
\begin{equation}\label{fw_op}
 \cL^\# Y := \begin{pmatrix} \nu \partial_\xi^2 w + f'\left(\hat u\right) w - q - s \partial_\xi w \\ \eps \left(w - \gamma q\right) - s \partial_\xi q \end{pmatrix} = \begin{pmatrix} \nu \partial_\xi^2 + f'\left(\hat u\right) - s \partial_\xi & - 1 \\ \eps & - \eps \gamma - s \partial_\xi \end{pmatrix} Y
\end{equation}
for the linearized evolution in the moving frame, i.e., with respect to $\xi=x+st$, and call $\cL^\#$ the \emph{frozen-wave operator}. Defining the \emph{translation operator} $\cT_c$ by
\begin{equation}\label{translation}
\cT_c Y := Y(\cdot+c) \quad \mbox{for any} \quad c \in \R,
\end{equation}
one may readily check that
\begin{equation}\label{lst_lf}
\partial_t - \cL_{st} = \cT_{st} \left(\partial_t - \cL^\#\right) \cT_{-st}.
\end{equation}
Furthermore, by differentiating \eqref{fhn_tw} it follows
\begin{equation}\label{frozen_tw_0}
\cL^\# \tfrac{\d \hat X}{\d \xi} = 0, \quad \mbox{where} \quad \hat X := \begin{pmatrix} \hat u \\ \hat v \end{pmatrix}.
\end{equation}
Equation~\eqref{frozen_tw_0} simply means that the derivative of the traveling wave is an eigenvector of the frozen-wave operator, which actually arises due to translation invariance~\cite{KuehnBook1}. Here and in what follows, we use the conventions of \cite[\S2.2.4, \S2.2.5, Definition~2.2.3, (4.1.11)]{KapitulaPromislow2013}:
%%%%%%%%%%%%%%%%%%%%%%%%%%%%%%%%%%%%%%%%%%%
\begin{definition}\label{def:spectrum}
For a complex Banach space $E$ and a linear operator $\cB \colon E \supseteq D(\cB) \to E$ we define
\begin{enumerate}[(a)]
\item\label{item:resolvent_set} the \emph{resolvent set} $\rho(\cB)$ as the set of all $\lambda \in \C$ such that $\lambda \id_E - \cB$ is invertible and $(\lambda \id_E-\cB)^{-1}$ is bounded; 
\item the \emph{spectrum} $\sigma(\cB) := \C \setminus \rho(\cB)$;
\item\label{item:def_point} the \emph{point spectrum} $\sigma_\mathrm{p}(\cB)$ as the set of all $\lambda \in \C$ such that the \emph{index} $\ind(\lambda \id_E - \cB)$ of $\lambda \id_E - \cB$ satisfies $\ind(\lambda \id_E - \cB) = 0$ but $\lambda \id_E - \cB$ is not invertible;
\item\label{item:essential} the \emph{essential spectrum} $\sigma_\mathrm{ess}(\cB)$ as the set of all $\lambda \in \C$ such that $\lambda \id_E - \cB$ is not a Fredholm operator with $\ind(\lambda \id_E-\cB) = 0$.
\item\label{item:numeric} for a $\C$-Hilbert space $E$ with inner product $\inner{\cdot}{\cdot}{E}$ the \emph{numerical range}
\[
\cR_E(\cB) := \left\{\inner{\cB Y}{Y}{E} \colon Y \in D(\cB), \, \vertii{Y}_E = 1\right\}.
\]
\end{enumerate}
\end{definition}
%%%%%%%%%%%%%%%%%%%%%%%%%%%%%%%%%%%%%%%%%%%

Next, we are going to list (spectral) properties of the frozen-wave operator in the statements below, which make use of a complex Hilbert space
\begin{subequations}
\begin{equation}\label{hilbert_c}
H_\C := L^2(\R;\C) \eoperp L^2(\R;\C),
\end{equation}
endowed with the inner product
\begin{equation}\label{inner_c}
\inner{Y_1}{Y_2}{H_\C} := Z \int_\R \left(\eps \overline{w_1} w_2 + \overline{q_1} q_2\right) \d\xi, \quad Y_j := \begin{pmatrix} w_j \\ q_j \end{pmatrix}, \quad Z \stackrel{\eqref{yj_z}}{=} \left(\int_\R \left(\eps \left(\tfrac{\d\hat u}{\d\xi}\right)^2 + \left(\tfrac{\d\hat v}{\d\xi}\right)^2\right)\d\xi\right)^{-1}.
\end{equation}
\end{subequations}
As mentioned in \S\ref{sec:tw} already, the spectral properties of $\cL^\#$ have in fact been studied by Jones~\cite{Jones1984} (based on the framework developed in \cite{Evans1971_72,Evans1972_73,Evans1972_73_2,Evans1974_75,Evans1976,Evans1976_2}) with shorter proofs given by Yanagida in~\cite{Yanagida1985} and by Arioli and Koch in \cite[\S3]{ArioliKoch2015} in the space of bounded uniformly continuous functions. Furthermore, in \cite[\S4.1]{ArioliKoch2015}, bounds on the eigenvalues for square integrable functions on the torus have been derived by a scaling of the components analogous to the one used in \eqref{hv_norm} or \eqref{inner_c}. Here, we adapt these approaches in order to give corresponding results in $H_\C$ in \S\ref{sec:frozen}, a choice being compatible with the Hilbert-space setting of \S\ref{sec:evol}. Directly applying the deterministic results \cite{Jones1984,Yanagida1985,ArioliKoch2015} would require to adapt our stochastic framework to Banach spaces, which we briefly comment on in \S\ref{sec:conclusions}.

\medskip

We emphasize that the following stability result, Proposition~\ref{prop:frozen}, in the space of square-integrable functions is (up to a scaling of the components of $Y$) in more generality proved by Ghazaryan, Latushkin, and Schecter in \cite{GhazaryanLatushkinSchecter} and by Yurov in \cite{Yurov}. These proofs mainly rely on applying the Gearhart-Pr\"uss theorem (see e.g.~\cite[Theorem~4.1.5]{KapitulaPromislow2013}) and Palmer's theorem \cite{BenArtziGohberg} and are more complicated than our relatively elementary arguments provided in \S\ref{sec:frozen}. See also \cite{RottmannMatthes2010} for an alternative proof using the Laplace transform by Rottmann-Matthes and \cite{RottmannMatthes2011} by the same author, where the method has been applied to first-order hyperbolic PDEs.

%%%%%%%%%%%%%%%%%%%%%%%%%%%%%%%%%%%%%%%%%%%
\begin{proposition}[properties of the frozen-wave operator $\cL^\#$]\label{prop:frozen}
For the frozen-wave operator
\[
\cL^\# \colon D(\cL^\#) := H^2(\R;\C) \eoperp H^1(\R;\C) \to H_\C \stackrel{\eqref{hilbert_c}}{=} L^2(\R;\C) \eoperp L^2(\R;\C)
\]
it holds
\begin{enumerate}[(a)]
\item\label{item:frozen_semi} $\cL^\#$ [$\cL^\#|_{D(\cL^\#) \cap H}$] generates a $C_0$-semigroup $\left(P_{st}^\#\right)_{t \ge 0}$ [$\left(P_{st}^\#|_H\right)_{t \ge 0}$] in $H_\C$ [$H$].
\item\label{item:semi_2} We have $P_{st,st'} = \cT_{st} P_{s (t-t')}^\#|_{\cV} \cT_{-st'}$ with $P_{st,st'}$ as in Proposition~\ref{prop:family}.
\item\label{item:spectrum} The \emph{spectrum} $\sigma(\cL^\#) = \sigma_\mathrm{ess}(\cL^\#) \dot\cup \sigma_\mathrm{p}(\cL^\#)$ consists of
\begin{enumerate}[(i)]
\item\label{item:essential-frozen} \emph{essential spectrum} $\sigma_\mathrm{ess}(\cL^\#) \subseteq \left\{ \lambda \in \C: \Re \lambda \le - \kappa\right\}$ with
\begin{equation}\label{def_kappa}
\kappa := \min \left\{-f'(0), \eps \gamma \right\}
\end{equation}
and
\item\label{item:point} \emph{point spectrum} $\sigma_\mathrm{p}(\cL^\#)$ consisting of an isolated eigenvalue $0$ of multiplicity $1$ with eigenvector $\frac{\d \hat X}{\d \xi}$ such that $\sigma_\mathrm{p}(\cL^\#) \setminus \{0\} \subseteq \left\{\lambda \in \C \colon \Re \lambda \le \lambda^*(\eps)\right\}$ for an eigenvalue $\lambda^*(\eps) < 0$, which we will call the \emph{Jones eigenvalue}~\cite{Jones1984} (see \cite{Yanagida1985} for more general reaction terms $f$ as treated here).
\end{enumerate}
\item\label{item:riesz} Defining the \emph{Riesz spectral projections} using \emph{Dunford calculus} as
\begin{equation}\label{riesz_projection}
\Pi^{\#,0} := \frac{1}{2 \pi i} \ointctrclockwise_{\verti{\lambda} = r} \left(\lambda \id_{H_\C} - \cL^\#\right)^{-1} \d\lambda \quad \mbox{and} \quad \Pi^\# := \id_{H_\C} - \Pi^{\#,0},
\end{equation}
where $r := \frac 1 2 \min\left\{\verti{\lambda^*(\eps)}, \kappa\right\}$, it follows that for any $\vartheta < \min\{\kappa,-\lambda^*(\eps)\}$ there exists $C_\vartheta \in [0,\infty)$ independent of $t \in [0,\infty)$ such that 
\begin{equation}\label{est_semigroupFrozenWave}
\vertii{P_{st}^\# \Pi^\#}_{L\left(H_\C\right)} \le C_{\vartheta} e^{-\vartheta t},
\end{equation}
where $\vertii{\cdot}_{L\left(H_\C\right)}$ denotes the \emph{operator norm} of bounded linear operators
\[
L^2(\R;\C) \eoperp L^2(\R;\C) \to L^2(\R;\C) \eoperp L^2(\R;\C).
\]
\item\label{item:riesz2} The operators $\left.\Pi^{\#,0}\right|_H, \left.\Pi^\#\right|_H \colon H \to H$ are well-defined and bounded, too, with $\vertii{\Pi^{\#,0}}_{L\left(H\right)} \le \vertii{\Pi^{\#,0}}_{L\left(H_\C\right)}$, $\vertii{\Pi^\#}_{L\left(H\right)} \le \vertii{\Pi^\#}_{L\left(H_\C\right)}$, and \eqref{est_semigroupFrozenWave} holds true with $H_\C$ replaced by $H$.
\end{enumerate}
\end{proposition}
%%%%%%%%%%%%%%%%%%%%%%%%%%%%%%%%%%%%%%%%%%%

%%%%%%%%%%%%%%%%%%%%%%%%%%%%%%%%%%%%%%%%%%%
\section{Main results\label{sec:main}}
%%%%%%%%%%%%%%%%%%%%%%%%%%%%%%%%%%%%%%%%%%%
For the following results, we can choose $T \in [0,\infty)$ arbitrarily.
%%%%%%%%%%%%%%%%%%%%%%%%%%%%%%%%%%%%%%%%%%%
\subsection{Correction of the wave velocity\label{sec:corr_intro}}
%%%%%%%%%%%%%%%%%%%%%%%%%%%%%%%%%%%%%%%%%%%
In accordance with \eqref{eq_phi_ks} we define the \emph{translated spectral Riesz projections}
\begin{equation}\label{pi_st_phi}
\Pi^0_{st + \varphi} := \cT_{st + \varphi} \Pi^{\#,0} \cT_{-st-\varphi}, \qquad \Pi_{st + \varphi} := \cT_{st + \varphi} \Pi^{\#} \cT_{-st-\varphi},
\end{equation}
and postulate the RODE as an approximation for the noise-induced wave velocity change
\begin{subequations}\label{eq_phi_xm}
\begin{equation}\label{eq_phi}
\dot\varphi^m(t) = \tfrac{\d \varphi^m}{\d t}(t) = m \inner{\Pi^0_{st+\varphi^m(t)} X^m(t,\cdot)}{\tfrac{\d\hat X}{\d \xi}(\cdot+st+\varphi^m(t))}{H},
\end{equation}
where $m > 0$ and
\begin{eqnarray*}
X^m(t,x) &:=& X_{\varphi^m}(t,x) \stackrel{\eqref{x_phi}}{=} \tilde X(t,x) - \hat X(x + s t + \varphi^m(t)) \\
&\stackrel{\eqref{x_phi}}{=}& X(t,x) + \hat X(x + s t) - \hat X(x + s t + \varphi^m(t)).
\end{eqnarray*}
%
%%%%%%%%%%%%%%%%%%%%%%%%%%%%%%%%%%%%%%%%%%%
\begin{remark}\label{rem_optcondphi}
Indeed, from \eqref{def_phi} with the choices \eqref{pi_st_phi} and $H \stackrel{\eqref{rigged_alt}}{=} L^2(\R) \eoperp L^2(\R)$, for $X^{(0)} \in \cV$ we obtain the necessary condition for a minimum to be
\begin{eqnarray*}
0 &=& \partial_\varphi \vertii{\Pi_{st+\varphi}^0 X_\varphi(t,\cdot)}_H^2 \\
&\stackrel{\eqref{x_phi}}{=}& \partial_\varphi \vertii{\Pi_{st+\varphi}^0 \left(\tilde X(t,\cdot) - \hat X(\cdot + s t + \varphi)\right)}_H^2 \\
&\stackrel{\eqref{pi_st_phi}}{=}& \partial_\varphi \vertii{\Pi^{\#,0} \left(\tilde X(t,\cdot - s t - \varphi) - \hat X\right)}_H^2 \\
&=& - 2 \inner{\Pi^{\#,0} \left(\tilde X\left(t,\cdot - s t - \varphi\right) - \hat X\right)}{\Pi^{\#,0}\left(\left(\partial_x \tilde X\right)\left(t,\cdot - s t - \varphi\right)\right)}{H} \\
&=& - \underbrace{\int_\R \partial_\xi \verti{\diag\left(\sqrt\eps,1\right) \cdot \left(\Pi^{\#,0} \left(\tilde X\left(t,\cdot - s t - \varphi\right) - \hat X\right)\right)(\xi)}^2 \d\xi}_{= 0} \\
&& - 2 \inner{\Pi^{\#,0} \left(\tilde X\left(t,\cdot - s t - \varphi\right) - \hat X\right)}{\tfrac{\d \hat X}{\d\xi}}{H} \\
&=& - 2 \inner{\Pi^0_{st+\varphi} \left(\tilde X(t,\cdot) - \hat X(\cdot+st+\varphi)\right)}{\tfrac{\d\hat X}{\d \xi}(\cdot+st+\varphi)}{H} \\
&\stackrel{\eqref{x_phi}}{=}& - 2 \inner{\Pi^0_{st+\varphi} X_\varphi(t,\cdot)}{\tfrac{\d\hat X}{\d \xi}(\cdot+st+\varphi)}{H},
\end{eqnarray*}
so that \eqref{def_phi_2} holds true, while \eqref{eq_phi_ks} or \eqref{eq_phi} correspond to a relaxation of this condition with numerical parameter $m$.
\end{remark}
%%%%%%%%%%%%%%%%%%%%%%%%%%%%%%%%%%%%%%%%%%%
By choosing $\varphi^m$ from \eqref{eq_phi} as the velocity adaption, the SPDE \eqref{evol_lin_tw} becomes
\begin{eqnarray}
 \d X^m(t,\cdot) &=& \left(\cL_{s t + \varphi^m(t)} X^m(t,\cdot) + R^m(t,X^m(t,\cdot),\cdot) - \dot\varphi^m(t) \tfrac{\d \hat X}{\d \xi}(\cdot+st+\varphi^m(t))\right) \d t \nonumber \\
 && + \begin{pmatrix} \sigma \\ 0 \end{pmatrix} \d W(t,\cdot),\label{eq_xm}
\end{eqnarray}
\end{subequations}
where
\begin{subequations}
\begin{eqnarray}
R^m(t,Y,\cdot) &:=& R_{\varphi^m}(t,Y,\cdot) \nonumber \\
&\stackrel{\eqref{remainder}}{=}& \begin{pmatrix} f\left(w + \hat u(\cdot+s t + \varphi^m)\right) - f\left(\hat u(\cdot+s t + \varphi^m)\right) - f'\left(\hat u(\cdot + s t + \varphi^m)\right) w \\ 0 \end{pmatrix}, \nonumber \\
&& \\
Y &=& \begin{pmatrix} w \\ q \end{pmatrix}.
\end{eqnarray}
\end{subequations}
The proof of the following proposition is contained in \S\ref{sec:phase}.
%%%%%%%%%%%%%%%%%%%%%%%%%%%%%%%%%%%%%%%%%%%
\begin{proposition}\label{prop:path_ode}
Suppose $\sqrt Q \in L_2(L^2(\R);H^1(\R))$ and $u^{(0)}, v^{(0)} \in H^1(\R)$.
\begin{enumerate}[(a)]
\item\label{item:path_exist} $\PP$-almost surely, there exists a unique $(\cF_t)_{t \ge 0}$-adapted solution $\varphi^m \in C^1([0,T])$ to the path-wise ODE \eqref{eq_phi} subject to $\varphi^m(0) = 0$.
\item\label{item:path_eq} The correction $\dot\varphi^m$ to the velocity $s$ of the traveling pulse satisfies the SODE
\begin{eqnarray}
\d\dot\varphi^m(t) &=& - m \dot\varphi^m(t) \left(1 + \inner{\Pi^0_{st+\varphi^m(t)} \partial_x X^m(t,\cdot)}{\tfrac{\d\hat X}{\d\xi}\left(\cdot+st+\varphi^m(t)\right)}{H}\right) \d t \nonumber \\
&& + \sigma m \inner{\Pi^0_{st+\varphi^m(t)} \left(1,0\right)^\trans \d W(t,\cdot)}{\tfrac{\d\hat X}{\d\xi}\left(\cdot+st+\varphi^m(t)\right)}{H} \nonumber \\
&& + m \inner{\Pi^0_{st+\varphi^m(t)} R^m(t,X^m(t,\cdot),\cdot)}{\tfrac{\d \hat X}{\d\xi}\left(\cdot+st+\varphi^m(t)\right)}{H} \d t. \label{sde_vel}
\end{eqnarray}
\end{enumerate}
\end{proposition}
%%%%%%%%%%%%%%%%%%%%%%%%%%%%%%%%%%%%%%%%%%%

%%%%%%%%%%%%%%%%%%%%%%%%%%%%%%%%%%%%%%%%%%%
\subsection{Reduced stochastic dynamics and multiscale analysis\label{sec:reduced_intro}}
%%%%%%%%%%%%%%%%%%%%%%%%%%%%%%%%%%%%%%%%%%%
Since for small values of $\sigma$, $X^m$ and $\dot\varphi^m$ are expected to be small, we may also consider the reduced set of equations (in which by linearity $\sigma$ can be scaled out)
\begin{subequations}\label{eq_phi_xm0}
\begin{eqnarray}
\d\dot\varphi_0^m(t) &=& - m \dot\varphi_0^m(t) \, \d t + m \inner{\Pi^0_{st} \left(1,0\right)^\trans \d W(t,\cdot)}{\tfrac{\d \hat X}{\d\xi}(\cdot+st)}{H}, \label{eq_phim0} \\
\d X_0^m(t,\cdot) &=& \left(\cL_{st} X_0^m(t,\cdot) - \dot\varphi_0^m(t) \tfrac{\d \hat X}{\d\xi}(\cdot+st)\right) \d t + \left(1,0\right)^\trans \d W(t,\cdot). \label{eq_xm0}
\end{eqnarray}
\end{subequations}
Note that the SPDE~\eqref{eq_xm0} describes the stochastic fluctuations around the stable fast pulse with stochastically adapted wave velocity. Due to the construction (cf.~the definition of $\mathcal{L}_{st}$ in~\eqref{lin_op}), the aforementioned SPDE arises from a multiscale argument and represents a leading-order approximation to track the fluctuations around the deterministically stable fast pulse.

\medskip

The following statements establish existence of solutions to \eqref{eq_phi_xm0} and yield a multiscale expansion of the solution $\tilde X = \left(\tilde u, \tilde v\right)^\trans$ to the original equation \eqref{sfhn}. Their proof is contained in \S\ref{sec:leading}. 
%%%%%%%%%%%%%%%%%%%%%%%%%%%%%%%%%%%%%%%%%%%
\begin{theorem}\label{th:multi}
For
\[
\sqrt Q \in L_2(L^2(\R);H^1(\R)), \quad X^m_0(0,\cdot) = X^{(0)}_0 \in \cV, \quad \dot\varphi^m_0(0) = m \inner{\Pi^{\#,0} X^{(0)}}{\tfrac{\d\hat X}{\d\xi}}{H}, \quad \varphi_0^m(0) = 0,
\]
we have:
\begin{enumerate}[(a)]
\item\label{item:multi_mild} The system \eqref{eq_phi_xm0} has a unique mild solution~\cite[Definition~1.1]{Seidler} given by
\begin{subequations}\label{mild_phi_x_0_m}
\begin{eqnarray}
\varphi_0^m(t) &=& \left(1-e^{-mt}\right) \inner{\Pi^{\#,0} X^{(0)}}{\tfrac{\d\hat X}{\d\xi}}{H} \nonumber \\
&& + \int_0^t \left(1-e^{-m (t-t')}\right) \inner{\Pi_{st'}^0 \left(1,0\right)^\trans \d W(t',\cdot)}{\tfrac{\d \hat X}{\d\xi}(\cdot+st')}{H}, \label{mild_phi_0_m} \\
X_0^m(t,\cdot) &=& P_{st,0} X^{(0)}_0 - \varphi_0^m(t) \tfrac{\d \hat X}{\d\xi}(\cdot+st) + \int_0^t P_{st,st'} \left(1,0\right)^\trans \d W(t',\cdot), \label{mild_x_0_m}
\end{eqnarray}
\end{subequations}
$\PP$-almost surely, where $\left(P_{st,st'}\right)_{t \ge t' \ge 0}$ is given by Proposition~\ref{prop:family}.
\item\label{item:multi_decomp} Let X be defined as in \eqref{x_phi} using Proposition~\ref{prop:ex_var} and Proposition~\ref{prop:reg_var}. Define the stopping times 
\begin{subequations}\label{stopping_times}
 \begin{align}
  \tau_{q,\sigma} &:= \inf\left(\left\{t \in [0,T]: \, \vertii{X(t,\cdot)}_\cV \ge \sigma^{1-q}\right\} \cup \{T\}\right),\label{stopping_time_X} \\
  \tau_{q,\sigma}^m &:= \inf\left(\left\{t \in [0,T]: \, \verti{\varphi^m_0(t)} \ge \sigma^{-q}\right\}\cup \{T\}\right),\label{stopping_time_phi}
 \end{align}
\end{subequations}
where $q \in \left[0,\frac 1 2\right)$.  Then, on $\left\{\min\left\{\tau_{q,\sigma},\tau_{q,\sigma}^m\right\} = T\right\}$ we have
\begin{equation}\label{decomp_x_sx_s}
\tilde X(t,\cdot) =: \hat X\left(\cdot+st+\sigma\varphi_0^m(t)\right) + \sigma X_0^m(t,\cdot) + \sigma S^m(t,\cdot),
\end{equation}
where the remainder $S^m$ meets the estimate
\begin{equation}\label{est_remainder_ms}
\vertii{S^m(t,\cdot)}_\cV \le C \sigma^{1-2q}\left(1 + \sigma^{1-q}\right), \quad \mbox{$\PP$-almost surely},
\end{equation}
with $C < \infty$ being independent of $\sigma$.
\item\label{item:multi_lim} For $q \in \left[0,\frac 1 2\right)$ it holds $\PP\left[\min\left\{\tau_{q,\sigma}, \tau_{q,\sigma}^m\right\} = T\right] \ge 1 - C \sigma^{2q} \to 1$ as $\sigma \searrow 0$ for a constant $C < \infty$.
\end{enumerate}
\end{theorem}
Note that by scaling out $\sigma$ beforehand, we have $X^{(0)}_0 := \sigma^{-1}X^{(0)}$, where $X^{(0)}$ is the initial condition for the unscaled equations \eqref{stoch_evol} describing the fluctuations around the traveling wave.
%%%%%%%%%%%%%%%%%%%%%%%%%%%%%%%%%%%%%%%%%%%

%%%%%%%%%%%%%%%%%%%%%%%%%%%%%%%%%%%%%%%%%%%
\subsection{Immediate relaxation\label{sec:immediate_intro}}
%%%%%%%%%%%%%%%%%%%%%%%%%%%%%%%%%%%%%%%%%%%
In the limit as $m \to \infty$ (\emph{immediate relaxation}), the system \eqref{eq_phi_xm0} further simplifies and we expect the solution $\left(X_0^m, \varphi_0^m\right)$ to \eqref{eq_phi_xm0} to converge to
\begin{subequations}\label{phi_x_inf_0}
\begin{align}
\varphi_0^\infty(t) &:= \inner{\Pi^{\#,0} X^{(0)}_0}{\tfrac{\d \hat X}{\d\xi}}{H} + \int_0^t \inner{\Pi^0_{s t'} \left(1,0\right)^\trans \d W(t',\cdot)}{\tfrac{\d\hat X}{\d\xi}(\cdot+st')}{H}, \label{phi_inf_0}\\
X_0^\infty(t,\cdot) &:= P_{st,0} \Pi^\# X^{(0)}_0 + \int_0^t P_{st,st'} \Pi_{s t'} \left(1,0\right)^\trans \d W(t',\cdot), \label{x_inf_0}
\end{align}
\end{subequations}
where $\left(P_{st,st'}\right)_{t \ge t' \ge 0}$ is given by Proposition~\ref{prop:family}.

\medskip

The proof of the following statement is contained in \S\ref{sec:immediate}.
%%%%%%%%%%%%%%%%%%%%%%%%%%%%%%%%%%%%%%%%%%%
\begin{theorem}\label{th:immediate}
Suppose $\sqrt Q \in L_2(L^2(\R);H^1(\R))$ and $X^{(0)}_0 \in \cV$. Let $\varphi^m$ be given as in Proposition~\ref{prop:path_ode}~\eqref{item:path_exist}, $X^m := X_{\varphi^m}$ be defined as in \eqref{x_phi} using Proposition~\ref{prop:ex_var} and Proposition~\ref{prop:reg_var}, and $\varphi_0^m$ and $X_0^m$ be given by Theorem~\ref{th:multi}.
\begin{enumerate}[(a)]
\item\label{item:immediate_xinf0_0} We have $\Pi_{st}^0 X_0^\infty(t,\cdot) = 0$ or equivalently $\Pi_{st} X_0^\infty(t,\cdot) = X_0^\infty(t,\cdot)$ for all $t \in [0,T]$, $\PP$-almost surely.
\item\label{item:immediate_lim} For any $\delta > 0$, we have
\begin{subequations}\label{limit_minf}
\begin{align}
\E\left[\sup_{t \in [\delta,T]} \verti{\varphi_0^m(t) - \varphi_0^\infty(t)}\right] \to 0 \quad &\mbox{as} \quad m \to \infty, \label{limit_phi_minf} \\
\E\left[\sup_{t \in [\delta,T]} \vertii{X_0^m(t,\cdot) - X_0^\infty(t,\cdot)}_\cV\right] \to 0 \quad &\mbox{as} \quad m \to \infty. \label{limit_x_minf}
\end{align}
\end{subequations}
\item\label{item:immediate_multi} Suppose $q \in \left[0,\frac 1 2\right)$, let $\tau_{q,\sigma}$ be defined as in \eqref{stopping_time_X}, and
\begin{equation}\label{t_q_sig_inf}
\tau_{q,\sigma}^\infty := \inf\left(\left\{ t \in [0,T] \colon \verti{\varphi_0^\infty(t)} \ge \sigma^{-q} \right\} \cup \{T\}\right).
\end{equation}
On $\left\{\min\left\{\tau_{q,\sigma}, \tau_{q,\sigma}^\infty\right\} = T \right\}$ the stochastic traveling wave has the multiscale decomposition
\begin{equation}\label{multi_imm_relax}
\tilde X(t,\cdot) =: \hat X \left(\cdot + st + \sigma \varphi_0^\infty(t)\right) + \sigma X_0^\infty(t,\cdot) + \sigma S^\infty(t,\cdot),
\end{equation}
with
\begin{equation}\label{est_s_inf}
\vertii{S^\infty(t,\cdot)}_\cV \le C \sigma^{1-2q} \left(1 + \sigma^{1-q}\right), \quad \mbox{$\PP$-almost surely,}
\end{equation}
where $C < \infty$ is independent of $\sigma$.
\item\label{item:immediate_stop} For $q \in \left[0,\frac 1 2\right)$ it holds $\PP\left[\min\left\{\tau_{q,\sigma}, \tau_{q,\sigma}^\infty\right\} = T\right] \ge 1 - C \sigma^{2q} \to 1$ as $\sigma \searrow 0$ for some $C < \infty$.
\item\label{item:immediate_min} On $\left\{\min\left\{\tau_{q,\sigma}, \tau_{q,\sigma}^\infty\right\} = T \right\}$, where $q \in \left[0,\frac 1 2\right)$, the function
\[
\R \owns \varphi \mapsto \vertii{\Pi_{st}^0 \left(\tilde X (t,\cdot) - \hat X\left(\cdot + st + \sigma \varphi\right)\right)}_H^2 \in \R
\]
is for $0 \leq t \leq T$ fixed, $\PP$-almost surely, locally approximately minimal at $\varphi = \varphi_0^\infty(t)$ in the sense that
\begin{subequations}\label{min_0_inf}
\begin{equation}\label{min_0_inf_1}
\left. \partial_\varphi \vertii{\Pi_{st}^0 \left(\tilde X (t,\cdot) - \hat X\left(\cdot + st + \sigma \varphi\right)\right)}_H^{2} \right|_{\varphi = \varphi_0^\infty(t)} = \cO\left(\sigma^{3-2q}\right) = o\left(\sigma^2\right),
\end{equation}
and
\begin{equation}\label{min_0_inf_2}
\left. \partial_\varphi^2 \vertii{\Pi_{st}^0 \left(\tilde X (t,\cdot) - \hat X(\cdot + st + \sigma \varphi)\right)}_H^2 \right|_{\varphi = \varphi_0^\infty(t)} = \sigma^2 \left(2 \vertii{\tfrac{\d\hat X}{\d\xi}}_H^2 + \cO\left(\sigma^{1-2q}\right)\right) > 0
\end{equation}
\end{subequations}
as $\sigma \searrow 0$, $\PP$-almost surely.
\end{enumerate}
\end{theorem}
%%%%%%%%%%%%%%%%%%%%%%%%%%%%%%%%%%%%%%%%%%%
Note that Theorem~\ref{th:immediate}~\eqref{item:immediate_xinf0_0} is a generalization of the orthogonality property in \cite[\S3.3]{KruegerStannat2017}.
Theorem~\ref{th:immediate}~\eqref{item:immediate_multi} and Theorem~\ref{th:immediate}~\eqref{item:immediate_stop} yield a multiscale expansion in the immediate-relaxation limit. Finally, Theorem~\ref{th:immediate}~\eqref{item:immediate_min} relates to our original motivation in \eqref{def_phi} to obtain the correction of the translation of the wave through a minimization argument. The following result is proved in in \S\ref{sec:immediate}.

%%%%%%%%%%%%%%%%%%%%%%%%%%%%%%%%%%%%%%%%%%%
\begin{proposition}\label{prop:moment}
Suppose that $\sqrt Q \in L_2(L^2(\R);L^2(\R))$ and $X^{(0)}_0 \in H$. Then the \emph{second moment} of $X_0^\infty$ satisfies the bound
\begin{equation}\label{moment_bound}
\E\left[\vertii{X_0^\infty(t,\cdot)}_H^2\right] \le 2 C_{\vartheta}^{2} e^{-2\vartheta t} \vertii{X^{(0)}_0}_H^2 + C_{\vartheta}^{2} \frac{1-e^{-2\vartheta t}}{\vartheta} \vertii{\Pi^\#}_{L(H)}^{2} \eps Z \vertii{\sqrt Q}_{L_2(L^2(\R))}^2
\end{equation}
with $\vartheta < \min\left\{\kappa,-\lambda^*(\eps)\right\}$, where $\kappa$ and the Jones eigenvalue $\lambda^*(\eps)$ have been introduced in Proposition~\ref{prop:frozen}, and the constant $C_\vartheta \in [0,\infty)$ only depends on $\vartheta$ (and the parameters $\gamma$, $\eps$, $f$, and $\nu$ of the deterministic FitzHugh-Nagumo system \eqref{sfhn} with $\sigma = 0$) but is independent of $t$, $T$, and $\sigma$.
\end{proposition}
%%%%%%%%%%%%%%%%%%%%%%%%%%%%%%%%%%%%%%%%%%%

The inequality~\eqref{moment_bound} is crucial for theoretical and application purposes as it bounds the expected size of fluctuations/deviations from the deterministic pulse. It shows that the second moment of the remaining deviations from the pulse, after correcting the wave velocity, are bounded by a sum of two contributions. The first term is simply due to the initial data $X^{(0)}_0$ and is exponentially decaying as in \eqref{est_semigroupFrozenWave} of Proposition~\ref{prop:frozen}~\eqref{item:riesz}. The second term is due to noise around the traveling wave, where once more the decay constant of \eqref{est_semigroupFrozenWave} enters. From a theoretical viewpoint the multiscale estimate for the second moment can then be useful in Doob/Markov-type inequalities to control the probabilities of individual sample paths~\cite{AntonopoulouBatesBloemkerKarali,BerglundGentz10,GnannKuehnPein}.

\medskip

Note that the noise contribution is of lower order compared to the second moment of fluctuations around the corresponding deterministic traveling wave (i.e., without stochastic velocity adaptation) because we have shifted appropriately to minimize the deviations in direction of the derivative of the traveling wave (which corresponds to the eigenvector with eigenvalue $0$ of the frozen-wave operator). This is made more precise in the following proposition, in which we compute the second moment of the fluctuations in direction of the derivative of the traveling wave on the event $\left\{\min\left\{\tau_{q,\sigma}, \tau_{q,\sigma}^\infty\right\} = T \right\}$ where the multiscale decomposition holds true. The proof can be found in \S\ref{sec:immediate}.
\begin{proposition}\label{prop:moment-without-adaption}
There exists a sequence $(Q_N)_{N \in \N}$ with $\sqrt{Q_N} \in L_2(L^2(\R);H^1(\R))$ such that for the deviations without adapting the wave velocity $X := X_0$ of \eqref{x_phi} (i.e., with $\varphi = 0$) with respect to $Q_N$ in the definition of the noise \eqref{def_b}, the second moment of deviations in direction of the traveling wave satisfies
\begin{eqnarray*}
\lefteqn{\E\left[ \inner{\Pi^0_{st} X(t,\cdot)}{\tfrac{\d \hat X}{\d\xi}(\cdot + st)}{H}^2 \mathds{1}_{\left\{\min\left\{\tau_{q,\sigma}, \tau_{q,\sigma}^\infty\right\} = T \right\}} \right]} \\
&=& \sigma^2 \inner{\Pi^{\#,0} X^{(0)}_0}{\tfrac{\d \hat X}{\d\xi}}{H}^2
+ \sigma^2 \eps Z  \left(t \left\| (1,0)^\trans \left(\Pi^{\#,0}\right)^*\tfrac{\d \hat X}{\d\xi}\right\|_{H}^2 + o(N^0)\right) + o(\sigma^2),
\end{eqnarray*}
where $\left\| (1,0)^\trans \left(\Pi^{\#,0}\right)^*\tfrac{\d \hat X}{\d\xi}\right\|_{H}^2 > 0$, $o(N^0)$ can depend on $t$ but not on $\sigma$, and $o(\sigma^2)$ can depend on $t$ and $N$. In particular, asymptotically the in $\sigma$ leading contribution of this moment grows linearly in time.
\end{proposition}
%

%%%%%%%%%%%%%%%%%%%%%%%%%%%%%%%%%%%%%%%%%%%
\section{Proofs of main results\label{sec:proofs_main}}
%%%%%%%%%%%%%%%%%%%%%%%%%%%%%%%%%%%%%%%%%%%
\subsection{Correction of the wave velocity\label{sec:phase}}
%%%%%%%%%%%%%%%%%%%%%%%%%%%%%%%%%%%%%%%%%%%
In this section, we prove Propositions~\ref{prop:path_ode}.

%%%%%%%%%%%%%%%%%%%%%%%%%%%%%%%%%%%%%%%%%%%
\begin{proof}[Proof of Proposition~\ref{prop:path_ode}~\eqref{item:path_exist}]
We use the path-wise definition
\[
F \colon [0,T] \times \R \mapsto \R, \quad (t,\varphi) \mapsto m \inner{\Pi^0_{st+\varphi} \left(\tilde X(t,\cdot) - \hat X(\cdot+st+\varphi)\right)}{\tfrac{\d\hat X}{\d \xi}(\cdot+st+\varphi)}{H}
\]
and note that by employing translation invariance of integrals, we have
\[
F(t,\varphi) \stackrel{\eqref{pi_st_phi}}{=} m \inner{\Pi^{\#,0} \left(\tilde X(t,\cdot - st - \varphi) - \hat X\right)}{\tfrac{\d\hat X}{\d \xi}}{H} \quad \mbox{for} \quad (t,\varphi) \in [0,T] \times \R.
\]
Then, we can compute for $(t_1,\varphi_1), (t_2,\varphi_2) \in [0,T] \times \R$,
\[
\verti{F(t_1,\varphi_1) - F(t_2,\varphi_2)} \le m \vertii{\Pi^{\#,0}}_{L(H)} \vertii{\tilde X(t_1,\cdot-st_1-\varphi_1) - \tilde X(t_2,\cdot-st_2-\varphi_2)}_H \vertii{\tfrac{\d\hat X}{\d\xi}}_H
\]
and on noting that
\begin{eqnarray*}
\lefteqn{\vertii{\tilde X(t_1,\cdot-st_1-\varphi_1) - \tilde X(t_2,\cdot-st_2-\varphi_2)}_H} \\
&\le& \vertii{\tilde X(t_1,\cdot-st_1-\varphi_1) - \tilde X(t_2,\cdot-st_1-\varphi_1)}_H \\
&& + \vertii{\tilde X(t_2,\cdot-st_1-\varphi_1) - \tilde X(t_2,\cdot-st_2-\varphi_2)}_H,
\end{eqnarray*}
with
\begin{eqnarray*}
\lefteqn{\vertii{\tilde X(t_1,\cdot-st_1-\varphi_1) - \tilde X(t_2,\cdot-st_1-\varphi_1)}_H} \\
&\stackrel{\eqref{def_txxhx}}{\le}& \vertii{X(t_1,\cdot-st_1-\varphi_1) - X(t_2,\cdot-st_1-\varphi_1)}_H \\
&& + \vertii{\hat X(\cdot-\varphi_1) - \hat X\left(\cdot + s (t_2 - t_1) - \varphi_1\right)}_H \\
&=& \vertii{X(t_1,\cdot) - X(t_2,\cdot)}_H + \vertii{\int_{s (t_1-t_2) + \varphi_1}^{\varphi_1} \tfrac{\d\hat X}{\d\xi}(\cdot-\xi) \d\xi}_H \\
&\le& \vertii{X(t_1,\cdot) - X(t_2,\cdot)}_H + \verti{s} \verti{t_1-t_2} \vertii{\tfrac{\d\hat X}{\d\xi}}_H
\end{eqnarray*}
and
\begin{eqnarray*}
\lefteqn{\vertii{\tilde X(t_2,\cdot-st_1-\varphi_1) - \tilde X(t_2,\cdot-st_2-\varphi_2)}_H} \\
&=& \vertii{\int_{st_1+\varphi_1}^{st_2+\varphi_2} \partial_\xi \tilde X(t_2,\cdot-\xi) \d\xi}_H \le \verti{\int_{st_1+\varphi_1}^{st_2+\varphi_2} \vertii{\partial_\xi \tilde X(t_2,\cdot-\xi)}_H \d\xi} \\
&\le& \left(\verti{s} \verti{t_1-t_2} + \verti{\varphi_1-\varphi_2}\right) \sup_{t \in [0,T]} \vertii{\partial_\xi \tilde X(t,\cdot)}_H \\
&\stackrel{\eqref{def_txxhx}}{\le}& \left(\verti{s} \verti{t_1-t_2} + \verti{\varphi_1-\varphi_2}\right) \left(\sup_{t \in [0,T]} \vertii{\partial_\xi X(t,\cdot)}_H + \vertii{\tfrac{\d\hat X}{\d\xi}}_H\right),
\end{eqnarray*}
and applying Proposition~\ref{prop:reg_var}, we obtain $\sup_{t \in [0,T]} \vertii{\partial_\xi X(t,\cdot)}_H < \infty$, $\PP$-almost surely, so that $F$ is, $\PP$-almost surely, continuous and globally Lipschitz continuous in the second component. Making use of the Picard-Lindel\"of theorem finishes the proof.
\end{proof}
%%%%%%%%%%%%%%%%%%%%%%%%%%%%%%%%%%%%%%%%%%%

%%%%%%%%%%%%%%%%%%%%%%%%%%%%%%%%%%%%%%%%%%%
\begin{proof}[Proof of Proposition~\ref{prop:path_ode}~\eqref{item:path_eq}]
We can rewrite \eqref{eq_phi} according to
\begin{eqnarray*}
\dot\varphi^m(t) &\stackrel{\eqref{eq_phi}}{=}& m \inner{\Pi^0_{st+\varphi^m(t)} X^m(t,\cdot)}{\tfrac{\d\hat X}{\d\xi}\left(\cdot+st+\varphi^m(t)\right)}{H} \\
&\stackrel{\eqref{pi_st_phi}}{=}& m \inner{\Pi^{\#,0} X^m\left(t,\cdot-st-\varphi^m(t)\right)}{\tfrac{\d\hat X}{\d\xi}}{H}, \quad \mbox{$\PP$-almost surely}.
\end{eqnarray*}
Taking the differential yields
\begin{eqnarray*}
\d\dot\varphi^m(t) &=& m \inner{\Pi^{\#,0} (\d X^m)\left(t,\cdot-st-\varphi^m(t)\right)}{\tfrac{\d\hat X}{\d\xi}}{H} \\
&& - \left(s+\dot\varphi^m(t)\right) m \inner{\Pi^{\#,0} \partial_\xi X^m\left(t,\cdot-st-\varphi^m(t)\right)}{\tfrac{\d\hat X}{\d\xi}}{H} \d t \\
&\stackrel{\eqref{pi_st_phi}}{=}& m \inner{\Pi^0_{st+\varphi^m(t)} \d X^m\left(t,\cdot\right)}{\tfrac{\d\hat X}{\d\xi}\left(\cdot+st+\varphi^m(t)\right)}{H} \d t \\
&& - \left(s+\dot\varphi^m(t)\right) m \inner{\Pi^0_{st+\varphi^m(t)} \partial_\xi X^m\left(t,\cdot\right)}{\tfrac{\d\hat X}{\d\xi}\left(\cdot+st+\varphi^m(t)\right)}{H} \d t \\
&\stackrel{\eqref{eq_xm}}{=}& m \inner{\Pi^0_{st+\varphi^m(t)} \cL_{st+\varphi^m(t)} X^m\left(t,\cdot\right)}{\tfrac{\d\hat X}{\d\xi}\left(\cdot+st+\varphi^m(t)\right)}{H} \d t \\
&& + m \inner{\Pi^0_{st+\varphi^m(t)} R^m\left(t,X^m\left(t,\cdot\right),\cdot\right)}{\tfrac{\d\hat X}{\d\xi}\left(\cdot+st+\varphi^m(t)\right)}{H} \d t \\
&& - m \dot\varphi^m(t) \inner{\Pi^0_{st+\varphi^m(t)} \tfrac{\d\hat X}{\d\xi}\left(\cdot+st+\varphi^m(t)\right)}{\tfrac{\d\hat X}{\d\xi}\left(\cdot+st+\varphi^m(t)\right)}{H} \d t \\
&& + \sigma m \inner{\Pi^0_{st+\varphi^m(t)} \left(1,0\right)^\trans \d W(t,\cdot)}{\tfrac{\d\hat X}{\d\xi}\left(\cdot+st+\varphi^m(t)\right)}{H} \\
&& - m \left(s+\dot\varphi^m(t)\right) \inner{\Pi^0_{st+\varphi^m(t)} \partial_x X^m\left(t,\cdot\right)}{\tfrac{\d\hat X}{\d\xi}\left(\cdot+st+\varphi^m(t)\right)}{H} \d t,
\end{eqnarray*}
$\PP$-almost surely. Now, we note that
\begin{eqnarray*}
\lefteqn{\inner{\Pi^0_{st+\varphi^m(t)} \tfrac{\d\hat X}{\d\xi}\left(\cdot+st+\varphi^m(t)\right)}{\tfrac{\d\hat X}{\d\xi}\left(\cdot+st+\varphi^m(t)\right)}{H}} \\
&\stackrel{\eqref{pi_st_phi}}{=}& \inner{\Pi^{\#,0} \tfrac{\d\hat X}{\d\xi}}{\tfrac{\d\hat X}{\d\xi}}{H} \stackrel{\eqref{frozen_tw_0},\eqref{riesz_projection}}{=} \vertii{\tfrac{\d\hat X}{\d\xi}}_H^2\stackrel{\eqref{h_norm}}{=} 1,
\end{eqnarray*}
so that we obtain the simplification
\begin{eqnarray*}
\d\dot\varphi^m(t) &=&m \inner{\Pi^0_{st+\varphi^m(t)} \left(\cL_{st+\varphi^m(t)} - s \partial_x\right) X^m\left(t,\cdot\right)}{\tfrac{\d\hat X}{\d\xi}\left(\cdot+st+\varphi^m(t)\right)}{H} \d t \\
&& - m \dot\varphi^m(t) \left(1 + \inner{\Pi^0_{st+\varphi^m(t)} \partial_x X^m\left(t,\cdot\right)}{\tfrac{\d\hat X}{\d\xi}\left(\cdot+st+\varphi^m(t)\right)}{H}\right) \d t \\
&& + \sigma m \inner{\Pi^0_{st+\varphi^m(t)} \left(1,0\right)^\trans \d W(t,\cdot)}{\tfrac{\d\hat X}{\d\xi}\left(\cdot+st+\varphi^m(t)\right)}{H} \\
&& + m \inner{\Pi^0_{st+\varphi^m(t)} R^m\left(t,X^m\left(t,\cdot\right),\cdot\right)}{\tfrac{\d\hat X}{\d\xi}\left(\cdot+st+\varphi^m(t)\right)}{H} \d t,
\end{eqnarray*}
$\PP$-almost surely. Next, we use \eqref{lin_op}, \eqref{fw_op}, \eqref{translation}, \eqref{riesz_projection}, and \eqref{pi_st_phi} to conclude that
\[
\Pi^0_{st+\varphi^m(t)} \left(\cL_{st+\varphi^m(t)} - s \partial_x\right) = \cT_{st+\varphi^m(t)} \Pi^{\#,0} \cL^\# \stackrel{\eqref{riesz_projection}}{=} \cT_{st+\varphi^m(t)} \cL^\# \Pi^{\#,0}
\]
and therefore
\begin{eqnarray*}
\lefteqn{\inner{\Pi^0_{st+\varphi^m(t)} \left(\cL_{st+\varphi^m(t)} - s \partial_x\right) X^m\left(t,\cdot\right)}{\tfrac{\d\hat X}{\d\xi}\left(\cdot+st+\varphi^m(t)\right)}{H}} \\
&=& \inner{\cL^\# \Pi^{\#,0} X^m\left(t,\cdot-st-\varphi^m(t)\right)}{\tfrac{\d\hat X}{\d\xi}}{H} \\
&\stackrel{\eqref{h_norm}, \eqref{yj_z}}{=}& \inner{\cL^\# \tfrac{\d \hat X}{\d\xi}}{\tfrac{\d\hat X}{\d\xi}}{H} \inner{\Pi^{\#,0} X^m\left(t,\cdot-st-\varphi^m(t)\right)}{\tfrac{\d\hat X}{\d\xi}}{H} \stackrel{\eqref{frozen_tw_0}}{=} 0,
\end{eqnarray*}
$\PP$-almost surely, so that we end up with \eqref{sde_vel}.
\end{proof}
%%%%%%%%%%%%%%%%%%%%%%%%%%%%%%%%%%%%%%%%%%%

%%%%%%%%%%%%%%%%%%%%%%%%%%%%%%%%%%%%%%%%%%%
\subsection{Reduced stochastic dynamics and multiscale analysis\label{sec:leading}}
%%%%%%%%%%%%%%%%%%%%%%%%%%%%%%%%%%%%%%%%%%%
In this section, we prove Theorem~\ref{th:multi}. We note that the existence and uniqueness of mild solutions for all non-autonomous linear SPDEs appearing in the next proof is guaranteed by standard results, see for instance \cite[Theorem~1.3]{Seidler} or the very general approach in~\cite{Veraar2010}.

%%%%%%%%%%%%%%%%%%%%%%%%%%%%%%%%%%%%%%%%%%%
\begin{proof}[Proof of Theorem~\ref{th:multi}~\eqref{item:multi_mild}] 
Since equation~\eqref{eq_phim0} decouples from \eqref{eq_xm0}, it forms a linear SDE for which we have a unique mild solution given by
\[
\dot\varphi_0^m(t) = m e^{-mt} \inner{\Pi^{\#,0} X^{(0)}_0}{\tfrac{\d\hat X}{\d\xi}}{H} + m \int_0^t e^{-m (t-t')} \inner{\Pi_{st'}^0 \left(1,0\right)^\trans \d W(t',\cdot)}{\tfrac{\d\hat X}{\d\xi}(\cdot+st')}{H},
\]
$\PP$-almost surely. Another integration using $\varphi_0^m(0) = 0$ yields
\begin{eqnarray*}
\varphi_0^m(t) &=& \left(1-e^{-mt}\right) \inner{\Pi^{\#,0} X^{(0)}_0}{\tfrac{\d\hat X}{\d\xi}}{H} \nonumber \\
&& + m \int_0^t \int_0^{t'} e^{-m (t'-t'')} \inner{\Pi_{st''}^0 \left(1,0\right)^\trans \d W(t'',\cdot)}{\tfrac{\d \hat X}{\d\xi}(\cdot+st'')}{H} \d t',
\end{eqnarray*}
$\PP$-almost surely. Utilizing
\begin{eqnarray*}
\lefteqn{m \int_0^t \int_0^{t'} e^{-m (t'-t'')} \inner{\Pi_{st''}^0 \left(1,0\right)^\trans \d W(t'',\cdot)}{\tfrac{\d \hat X}{\d\xi}(\cdot+st'')}{H} \d t'} \\
&=& \int_0^t \int_0^{t-t''} m e^{-m t'''} \, \d t''' \inner{\Pi_{st''}^0 \left(1,0\right)^\trans \d W(t'',\cdot)}{\tfrac{\d \hat X}{\d\xi}(\cdot+st'')}{H} \\
&=& \int_0^t \left(1-e^{-m (t-t'')}\right) \inner{\Pi_{st''}^0 \left(1,0\right)^\trans \d W(t'',\cdot)}{\tfrac{\d \hat X}{\d\xi}(\cdot+st'')}{H},
\end{eqnarray*}
we end up with \eqref{mild_phi_0_m}.

\medskip

Since $\varphi_0^m$ is already uniquely defined, we may uniquely solve \eqref{eq_xm0} with the mild-solution formula
\[
X_0^m(t,\cdot) = P_{st,0} X^{(0)}_0 - \int_0^t \dot\varphi_0^m(t') P_{st,st'} \tfrac{\d \hat X}{\d\xi}(\cdot+st') \, \d t' + \int_0^t P_{st,st'} \left(1,0\right)^\trans \d W(t',\cdot),
\]
$\PP$-almost surely. With help of Proposition~\ref{prop:frozen}~\eqref{item:frozen_semi} and \eqref{item:semi_2} it follows
\[
P_{st,st'} \tfrac{\d \hat X}{\d\xi}(\cdot+st') \stackrel{\eqref{translation}}{=} \cT_{st} P^\#_{s (t-t')} \tfrac{\d\hat X}{\d\xi} \stackrel{\eqref{translation}, \eqref{frozen_tw_0}}{=} \tfrac{\d\hat X}{\d\xi}(\cdot+st),
\]
so that
\[
\int_0^t \dot\varphi_0^m(t') P_{st,st'} \tfrac{\d \hat X}{\d\xi}(\cdot+st') \d t' = \int_0^t \dot\varphi_0^m(t') \, \d t' \, \tfrac{\d \hat X}{\d\xi}(\cdot+st) = \varphi_0^m(t) \tfrac{\d \hat X}{\d\xi}(\cdot+st)
\]
and we arrive at \eqref{mild_x_0_m}.
\end{proof}
%%%%%%%%%%%%%%%%%%%%%%%%%%%%%%%%%%%%%%%%%%%

%%%%%%%%%%%%%%%%%%%%%%%%%%%%%%%%%%%%%%%%%%%
\begin{proof}[Proof of Theorem~\ref{th:multi}~\eqref{item:multi_decomp}]
In what follows, we restrict ourselves to paths on $\left\{\min\left\{\tau_{q,\sigma}, \tau_{q,\sigma}^m\right\} = T \right\}$, which ensures smallness of $X(t,\cdot)$ and $\varphi_0^m(t)$.

\medskip

We consider the remainder
\begin{eqnarray*}
\sigma S^m(t,\cdot) &\stackrel{\eqref{decomp_x_sx_s}}{=}& \tilde X(t,\cdot) - \hat X\left(\cdot + s t + \sigma \varphi_0^m(t)\right) - \sigma X_0^m(t,\cdot) \\
 &\stackrel{\eqref{def_txxhx}}{=}& X(t,\cdot) + \hat X\left(\cdot + s t\right) - \hat X\left(\cdot + s t + \sigma \varphi_0^m(t)\right) - \sigma X_0^m(t,\cdot),
\end{eqnarray*}
so that with
\begin{eqnarray*}
 \d X(t,\cdot) &\stackrel{\eqref{evol_lin_tw},\eqref{remainder}}{=}& \cL_{st} X(t,\cdot) \, \d t + \begin{pmatrix} \sigma \\ 0 \end{pmatrix} \d W(t,\cdot), \\
 && + \begin{pmatrix} f\left(u(t,\cdot) + \hat u(\cdot+s t)\right) - f\left(\hat u(\cdot+s t)\right) - f'\left(\hat u(\cdot + s t)\right) u(t,\cdot) \\ 0 \end{pmatrix} \d t \nonumber \\
 \d \hat X\left(\cdot + s t\right) &\stackrel{\eqref{fhn_tw},\eqref{lin_op}}{=}& \cL_{st} \hat X\left(\cdot + s t\right) \d t + \begin{pmatrix} f\left(\hat u\left(\cdot + s t\right)\right) - f'\left(\hat u\left(\cdot + s t\right)\right) \hat u\left(\cdot + s t\right) \\ 0 \end{pmatrix} \d t, 
\end{eqnarray*}
as well as
\begin{eqnarray*}
 \lefteqn{\d \hat X\left(\cdot + s t + \sigma \varphi_0^m(t)\right)} \\
 &\stackrel{\eqref{fhn_tw},\eqref{lin_op}}{=}& \cL_{st} \hat X\left(\cdot + s t + \sigma \varphi_0^m(t)\right) \d t \\
 && + \begin{pmatrix} f\left(\hat u\left(\cdot + s t + \sigma \varphi_0^m(t)\right)\right) - f'\left(\hat u\left(\cdot + s t\right)\right) \hat u\left(\cdot + s t + \sigma \varphi_0^m(t)\right) \\ 0 \end{pmatrix}  \d t \\
 && + \sigma \dot\varphi_0^m(t) \tfrac{\d \hat X}{\d \xi}\left(\cdot + s t + \sigma \varphi_0^m(t)\right) \d t
\end{eqnarray*}
\mbox{and}
\[
 \d X_0^m(t,\cdot) \stackrel{\eqref{eq_xm0}}{=} \cL_{st} X_0^m\left(t,\cdot\right) \d t - \dot\varphi_0^m(t) \tfrac{\d \hat X}{\d \xi}\left(\cdot + s t\right) \d t + \begin{pmatrix} 1 \\ 0 \end{pmatrix} \d W(t,\cdot),
\]
we get with
\begin{equation}\label{def_tu0m}
 \tilde u_0^m(t,\cdot) := u(t,\cdot) + \hat u\left(\cdot + s t\right) - \hat u\left(\cdot + s t + \sigma \varphi_0^m(t)\right)
\end{equation}
that
\[
\d S^m(t,\cdot) - \cL_{st} S^m\left(t,\cdot\right) \d t = S_1^m(t,\cdot) + S_2^m(t,\cdot) + S_3^m(t,\cdot),
\]
where
\begin{subequations}\label{def_sjm}
\begin{eqnarray}
 S_1^m(t,\cdot) &:=& \sigma^{-1} \begin{pmatrix} f\left(\hat u\left(\cdot + s t + \sigma \varphi_0^m(t)\right) + \tilde u_0^m(t,\cdot)\right) - f\left(\hat u\left(\cdot + s t + \sigma \varphi_0^m(t)\right)\right) \\ 0 \end{pmatrix} \nonumber \\
 && - \sigma^{-1} \begin{pmatrix} f'\left(\hat u\left(\cdot + s t + \sigma \varphi_0^m(t)\right)\right) \tilde u_0^m(t,\cdot) \\ 0 \end{pmatrix}, \label{def_s1m} \\
S_2^m(t,\cdot) &:=& \sigma^{-1} \left(f'\left(\hat u\left(\cdot + s t + \sigma \varphi_0^m(t)\right)\right) - f'\left(\hat u\left(\cdot + s t\right)\right)\right) \begin{pmatrix} \tilde u_0^m(t,\cdot) \\ 0 \end{pmatrix}, \label{def_s2m} \\
S_3^m(t,\cdot) &:=& \dot\varphi_0^m(t) \left(\tfrac{\d \hat X}{\d \xi}\left(\cdot + s t\right) - \tfrac{\d \hat X}{\d \xi}\left(\cdot + s t + \sigma \varphi_0^m(t)\right)\right). \label{def_s3m}
\end{eqnarray}
\end{subequations}
Since $\left(\cL_{st}\right)_{t \ge 0}$ generates an evolution family $\left(P_{st,st'}\right)_{t \ge t' \ge 0}$ (cf.~Proposition~\ref{prop:family}), we find the mild-solution representation
\begin{equation} \label{duhamel_xm_x0m}
S^m(t,\cdot) = \int_0^t P_{st,st'} \left(S_1^m(t',\cdot) + S_2^m(t',\cdot) + S_3^m(t',\cdot)\right) \d t', \quad \mbox{$\PP$-almost surely}.
\end{equation}
Note that \eqref{duhamel_xm_x0m} can be rigorously justified by following the arguments detailed at the beginning of the proof of Proposition~\ref{prop:reg_var}. We continue by treating the terms $\int_0^t P_{s t, s t'} S_j^m(t',\cdot) \, \d t'$ for $j \in \{1,2,3\}$ separately.

\medskip

First observe that we have the point-wise estimate
\begin{eqnarray*}
\verti{S_1^m(t,x)} &\stackrel{\eqref{cond_f_diff_2},\eqref{def_s1m}}{\le}& \sigma^{-1} \eta_2 \left(1 + \verti{\hat u\left(x + s t + \sigma \varphi_0^m(t)\right)} + \verti{\tilde u_0^m(t,x)}\right) \verti{\tilde u_0^m(t,x)}^2 \\
&\le& \sigma^{-1} \eta_2 \left(1 + \vertii{\hat u}_{L^\infty(\R)} + \vertii{\tilde u_0^m(t,\cdot)}_{H^1(\R)}\right) \vertii{\tilde u_0^m(t,\cdot)}_{H^1(\R)} \verti{\tilde u_0^m(t,x)},
\end{eqnarray*}
where the Sobolev embedding in form of $\vertii{\tilde u_0^m(t,\cdot)}_{L^\infty(\R)} \le \vertii{\tilde u_0^m(t,\cdot)}_{H^1(\R)}$ has been applied. This already yields
\begin{equation}\label{est_s1m_1}
\vertii{S_1^m(t,\cdot)}_H \stackrel{\eqref{h_norm}}{\le} \sigma^{-1} \sqrt{\eps Z} \, \eta_2 \left(1 + \vertii{\hat u}_{L^\infty(\R)} + \vertii{\tilde u_0^m(t,\cdot)}_{H^1(\R)}\right) \vertii{\tilde u_0^m(t,\cdot)}_{H^1(\R)} \vertii{\tilde u_0^m(t,\cdot)}_{L^2(\R)}.
\end{equation}
On the other hand, a direct computation gives (suppressing the arguments in $S_1^m = S_1^m(t,x)$, $u_0^m = u_0^m(t,x)$, $\hat u = \hat u\left(x + s t + \sigma \varphi_0^m(t)\right)$, and $\tilde u_0^m = \tilde u_0^m(t,x)$)
\begin{eqnarray*}
\sigma \partial_x S_1^m &\stackrel{\eqref{def_s1m}}{=}& \left(f'\left(\hat u + \tilde u_0^m\right) \left(\partial_x \tilde u_0^m + \tfrac{\d \hat u}{\d\xi}\right) - f'\left(\hat u\right) \tfrac{\d \hat u}{\d\xi} - f''\left(\hat u\right) \tfrac{\d \hat u}{\d\xi} \tilde u_0^m - f'\left(\hat u\right) \partial_x \tilde u_0^m\right) (1,0)^\trans \\
&=& \left(\left(f'\left(\hat u + \tilde u_0^m\right) - f'\left(\hat u\right)  - f''\left(\hat u\right) \tilde u_0^m\right) \tfrac{\d \hat u}{\d\xi} + \left(f'\left(\hat u + \tilde u_0^m\right) - f'\left(\hat u\right)\right) \partial_x \tilde u_0^m\right) (1,0)^\trans.
\end{eqnarray*}
Now, we note that
\begin{eqnarray*}
\verti{\left(f'\left(\hat u+\tilde u_0^m\right) - f'\left(\hat u\right)  - f''\left(\hat u\right) \tilde u_0^m\right) \tfrac{\d \hat u}{\d\xi}} &\stackrel{\eqref{cond_f_diff_5}}{\le}& \eta_5 \left(\tilde u_0^m(t,x)\right)^2 \verti{\tfrac{\d \hat u}{\d\xi}\left(x + s t + \sigma \varphi_0^m(t)\right)} \\
&\le& \eta_5 \vertii{\tilde u_0^m(t,\cdot)}_{H^1(\R)} \vertii{\tfrac{\d\hat u}{\d\xi}}_{L^\infty(\R)} \verti{\tilde u_0^m(t,x)}, \\
\verti{\left(f'\left(\hat u + \tilde u_0^m\right) - f'\left(\hat u\right)\right) \partial_x \tilde u_0^m} &\stackrel{\eqref{cond_f_diff_6}}{\le}& \eta_6 \verti{\tilde u_0^m(t,x)} \verti{\partial_x \tilde u_0^m(t,x)} \\
&& \left(1 + 2 \verti{\hat u\left(x + s t + \sigma \varphi_0^m(t)\right)} + \verti{\tilde u_0^m(t,x)}\right) \\
&\le& \eta_6 \vertii{\tilde u_0^m(t,\cdot)}_{H^1(\R)} \verti{\partial_x \tilde u_0^m(t,x)} \\
&& \left(1 + 2 \vertii{\hat u}_{L^\infty(\R)} + \vertii{\tilde u_0^m(t,\cdot)}_{H^1(\R)}\right),\end{eqnarray*}
where $\vertii{\tilde u_0^m(t,\cdot)}_{L^\infty(\R)} \le \vertii{\tilde u_0^m(t,\cdot)}_{H^1(\R)}$ has been utilized once again. Hence,
\begin{eqnarray*}
\lefteqn{\vertii{\partial_x S_1^m(t,\cdot)}_H} \\
&\stackrel{\eqref{h_norm}}{\le}&  \sigma^{-1} \sqrt{\eps Z} \, \eta_5 \vertii{\tilde u_0^m(t,\cdot)}_{H^1(\R)} \vertii{\tfrac{\d\hat u}{\d\xi}}_{L^\infty(\R)} \vertii{\tilde u_0^m(t,\cdot)}_{L^2(\R)} \\
&& +  \sigma^{-1} \sqrt{\eps Z} \, \eta_6 \vertii{\tilde u_0^m(t,\cdot)}_{H^1(\R)} \vertii{\partial_x \tilde u_0^m(t,\cdot)}_{L^2(\R)} \left(1 + 2 \vertii{\hat u}_{L^\infty(\R)} + \vertii{\tilde u_0^m(t,\cdot)}_{H^1(\R)}\right).
\end{eqnarray*}
The combination with \eqref{est_s1m_1} yields
\begin{equation}\label{est_s1m}
\vertii{S_1^m(t,\cdot)}_\cV \stackrel{\eqref{sv_norm}}{\le} \sigma^{-1} \, C_1 \vertii{\tilde u_0^m(t,\cdot)}_{H^1(\R)}^2 \left(1+\vertii{\tilde u_0^m(t,\cdot)}_{H^1(\R)}\right),
\end{equation}
where $C_1 < \infty$ is independent of $\sigma$.

\medskip

Next, we estimate $S_2^m$. Notice that on one hand we have
\begin{eqnarray*}
\verti{f'\left(\hat u\left(x + s t + \sigma \varphi_0^m(t)\right)\right) - f'\left(\hat u(x+st)\right)} &\stackrel{\eqref{cond_f_diff_6}}{\le}& \eta_6 \verti{\hat u\left(x + s t + \sigma \varphi_0^m(t)\right) - \hat u(x+st)} \\
&& \left(1 + \verti{\hat u\left(x + s t + \sigma \varphi_0^m(t)\right)} + \verti{\hat u(x+st)}\right) \\
&\le& \eta_6 \vertii{\tfrac{\d\hat u}{\d\xi}}_{L^\infty(\R)} \left(1 + 2 \vertii{\hat u}_{L^\infty(\R)}\right) \sigma \verti{\varphi_0^m(t)},
\end{eqnarray*}
so that
\begin{equation}\label{est_s2m_1}
\vertii{S_2^m(t,\cdot)}_H \stackrel{\eqref{h_norm}, \eqref{def_s2m}}{\le}  \sqrt{\eps Z} \, \eta_6 \vertii{\tfrac{\d\hat u}{\d\xi}}_{L^\infty(\R)} \left(1 + 2 \vertii{\hat u}_{L^\infty(\R)}\right) \verti{\varphi_0^m(t)} \vertii{\tilde u_0^m(t,\cdot)}_{L^2(\R)}.
\end{equation}
On the other hand, for the derivative we get
\begin{eqnarray*}
\lefteqn{\partial_x \left[\left(f'\left(\hat u(x+st+\sigma \varphi_0^m(t))\right) - f'\left(\hat u(x+st)\right)\right) \tilde u_0^m(t,x)\right]} \\
&=& \left[f''\left(\hat u\left(x+st+\sigma \varphi_0^m(t)\right)\right) \tfrac{\d \hat u}{\d\xi}\left(x+st+ \sigma \varphi_0^m(t)\right) - f''\left(\hat u(x+st)\right) \tfrac{\d \hat u}{\d\xi}(x+st)\right] \tilde u_0^m(t,x) \\
&& + \left[f'\left(\hat u\left(x+st+\sigma \varphi_0^m(t)\right)\right) - f'\left(\hat u(x+st)\right)\right] \partial_x \tilde u_0^m(t,x) \\
&=& \left[f''\left(\hat u\left(x+st+ \sigma \varphi_0^m(t)\right)\right) - f''\left(\hat u(x+st)\right)\right] \tfrac{\d \hat u}{\d\xi}\left(x+st+ \sigma \varphi_0^m(t)\right) \tilde u_0^m(t,x) \\
&& + f''\left(\hat u(x+st)\right) \left[\tfrac{\d \hat u}{\d\xi}\left(x+st+ \sigma \varphi_0^m(t)\right) - \tfrac{\d \hat u}{\d\xi}(x+st)\right] \tilde u_0^m(t,x) \\
&& + \left[f'\left(\hat u(x+st+ \sigma \varphi_0^m(t))\right) - f'\left(\hat u(x+st)\right)\right] \partial_x \tilde u_0^m(t,x),
\end{eqnarray*}
so that
\begin{eqnarray*}
\lefteqn{\verti{\partial_x \left[\left(f'\left(\hat u(x+st+\sigma \varphi_0^m(t))\right) - f'\left(\hat u(x+st)\right)\right) \tilde u_0^m(t,x)\right]}} \\
&\stackrel{\eqref{cond_f_diff_6},\eqref{cond_f_diff_7}}{\le}& \eta_7 \verti{\hat u(x+st+\sigma \varphi_0^m(t)) - \hat u(x+st)} \verti{\tfrac{\d \hat u}{\d\xi}\left(x+st+ \sigma \varphi_0^m(t)\right)} \verti{\tilde u_0^m(t,x)} \\
&& + \verti{f''\left(\hat u(x+st)\right)} \verti{\tfrac{\d \hat u}{\d\xi}\left(x+st+ \sigma \varphi_0^m(t)\right) - \tfrac{\d \hat u}{\d\xi}(x+st)} \verti{\tilde u_0^m(t,x)} \\
&& + \eta_6 \verti{\hat u(x+st+ \sigma \varphi_0^m(t)) - \hat u(x+st)} \left(1 + \verti{\hat u(x+st+ \sigma \varphi_0^m(t))} + \verti{\hat u(x+st)}\right) \\
&& \phantom{+} \verti{\partial_x \tilde u_0^m(t,x)} \\
&\le& \eta_7 \vertii{\tfrac{\d\hat u}{\d\xi}}_{L^\infty(\R)}^2 \sigma \verti{\varphi_0^m(t)} \verti{\tilde u_0^m(t,x)} + \vertii{f''\left(\hat u\right)}_{L^\infty(\R)} \vertii{\tfrac{\d^2\hat u}{\d\xi^2}}_{L^\infty(\R)} \sigma \verti{\varphi_0^m(t)} \verti{\tilde u_0^m(t,x)} \\
&& + \eta_6 \vertii{\tfrac{\d\hat u}{\d\xi}}_{L^\infty(\R)} \left(1 + 2 \vertii{\hat u}_{L^\infty(\R)}\right) \sigma \verti{\varphi_0^m(t)} \verti{\partial_x \tilde u_0^m(t,x)}
\end{eqnarray*}
This gives
\begin{eqnarray*}
\lefteqn{\vertii{\partial_x \left[\left(f'\left(\hat u(\cdot+st+\sigma \varphi_0^m(t))\right) - f'\left(\hat u(\cdot+st)\right)\right) \tilde u_0^m(t,\cdot)\right]}_H} \\
&\stackrel{\eqref{h_norm}}{\le}& \sqrt{\eps Z} \, \eta_7 \vertii{\tfrac{\d\hat u}{\d\xi}}_{L^\infty(\R)}^2 \sigma \verti{\varphi_0^m(t)} \vertii{\tilde u_0^m(t,\cdot)}_{L^2(\R)} \\
&& + \sqrt{\eps Z} \vertii{f''\left(\hat u\right)}_{L^\infty(\R)} \vertii{\tfrac{\d^2\hat u}{\d\xi^2}}_{L^\infty(\R)} \sigma \verti{\varphi_0^m(t)} \vertii{\tilde u_0^m(t,\cdot)}_{L^2(\R)} \\
&& + \sqrt{\eps Z} \, \eta_6 \vertii{\tfrac{\d\hat u}{\d\xi}}_{L^\infty(\R)} \left(1 + 2 \vertii{\hat u}_{L^\infty(\R)}\right) \sigma \verti{\varphi_0^m(t)} \vertii{\partial_x \tilde u_0^m(t,\cdot)}_{L^2(\R)}
\end{eqnarray*}
and the combination with \eqref{est_s2m_1} yields
\begin{equation}\label{est_s2m}
\vertii{S_2^m(t,\cdot)}_\cV \stackrel{\eqref{sv_norm}}{\le} C_2 \verti{\varphi_0^m(t)} \vertii{\tilde u_0^m(t,\cdot)}_{H^1(\R)},
\end{equation}
where $C_2 < \infty$ is independent of $\sigma$.

\medskip

For estimating $\vertii{\int_0^t P_{s t,s t'} S_3^m(t',\cdot) \, \d t'}_\cV$, observe that
\begin{eqnarray*}
\lefteqn{(\partial_t - s \partial_x) \left(\hat X\left(x + s t + \sigma \varphi_0^m(t)\right) - \hat X(x + s t) - \tfrac{\d \hat X}{\d \xi}(x + s t) \sigma \varphi_0^m(t)\right)} \\
&=& \sigma \dot \varphi_0^m(t) \left(\tfrac{\d \hat X}{\d \xi}\left(x + s t + \sigma \varphi_0^m(t)\right) - \tfrac{\d \hat X}{\d \xi}(x+st)\right) \stackrel{\eqref{def_s3m}}{=} - \sigma S_3^m(t,\cdot).
\end{eqnarray*}
Using integration by parts, this gives
\begin{eqnarray*}
\lefteqn{\sigma \int_0^t P_{st,st'} S_3^m(t',\cdot) \d t' } \\
&=& - \int_0^t P_{st,st'} (\partial_{t'} - s \partial_x) \left(\hat X\left(\cdot + s t' + \sigma \varphi_0^m(t')\right) - \hat X(\cdot + s t') - \tfrac{\d \hat X}{\d \xi}(\cdot + s t') \sigma  \varphi_0^m(t')\right) \d t' \\
&=& - \hat X\left(\cdot + s t + \sigma \varphi_0^m(t)\right) + \hat X(\cdot + s t) + \tfrac{\d \hat X}{\d \xi}(\cdot + s t) \sigma \varphi_0^m(t) \\
&& + \int_0^t \left((\partial_{t'} P_{st,st'}) + P_{st,st'} s \partial_x\right) \left(\hat X\left(\cdot + s t' + \sigma \varphi_0^m(t')\right) - \hat X(\cdot + s t') - \tfrac{\d \hat X}{\d \xi}(\cdot + s t') \sigma \varphi_0^m(t')\right) \d t' \\
&=& - \hat X\left(\cdot + s t + \sigma \varphi_0^m(t)\right) + \hat X(\cdot + s t) + \tfrac{\d \hat X}{\d \xi}(\cdot + s t) \sigma \varphi_0^m(t) \\
&& - \int_0^t P_{st,st'} \left(\LL_{s t'} - s \partial_x\right) \left(\hat X\left(\cdot + s t' + \sigma \varphi_0^m(t')\right) - \hat X(\cdot + s t') - \tfrac{\d \hat X}{\d \xi}(\cdot + s t') \sigma \varphi_0^m(t')\right) \d t',
\end{eqnarray*}
$\PP$-almost surely, and therefore by boundedness of
\[
\left.\left(\LL_{st} - s \partial_x\right)\right|_{H^3(\R) \eoperp H^2(\R)} \colon \cU := H^3(\R) \eoperp H^2(\R) \to \cV
\]
(cf.~\eqref{lin_op}) and employing Proposition~\ref{prop:family}, we have
\begin{eqnarray*}
\lefteqn{\sigma \vertii{\int_0^t P_{st,st'} S_3^m(t',\cdot) \, \d t'}_\cV} \\
&\stackrel{\eqref{bound_pstst}}{\le}& \vertii{\hat X\left(\cdot + s t + \sigma \varphi_0^m(t)\right) - \hat X(\cdot + s t) - \tfrac{\d \hat X}{\d \xi}(\cdot + s t) \sigma \varphi_0^m(t)}_\cV \\
&& + C \int_0^t e^{\beta (t-t')} \vertii{\hat X\left(\cdot + s t' + \sigma \varphi_0^m(t')\right) - \hat X(\cdot + s t') - \tfrac{\d \hat X}{\d \xi}(\cdot + s t') \sigma \varphi_0^m(t')}_{\cU} \d t',
\end{eqnarray*}
$\PP$-almost surely, where $C < \infty$ and $\beta \stackrel{\eqref{def_beta}}{=} \vertii{f'(\hat u) - f'(0)}_{W^{1,\infty}(\R)} - \min\left\{\nu, - f'(0), \eps\gamma\right\}$. Now, we note that
\[
\vertii{\partial_x^j \left(\hat u\left(\cdot+st+\sigma \varphi_0^m(t)\right) - \hat u(\cdot+st) - \tfrac{\d\hat u}{\d\xi}(\cdot+st) \, \sigma \varphi_0^m(t)\right)}_{L^2(\R)} \le \frac 1 2 \vertii{\tfrac{\d^{j+2} \hat u}{\d\xi^{j+2}}}_{L^2(\R)} (\sigma \varphi_0^m(t))^2
\]
for $j \in \{0,1,2,3\}$ and in the same way
\[
\vertii{\partial_x^j \left(\hat v\left(\cdot+st+\sigma\varphi_0^m(t)\right) - \hat v(\cdot+st) - \tfrac{\d\hat v}{\d\xi}(\cdot+st) \, \sigma \varphi_0^m(t)\right)}_{L^2(\R)} \le \frac 1 2 \vertii{\tfrac{\d^{j+2} \hat v}{\d\xi^{j+2}}}_{L^2(\R)} (\sigma \varphi_0^m(t))^2
\]
for $j \in \{0,1,2\}$. As a result, we obtain
\begin{eqnarray}\nonumber
\vertii{\int_0^t P_{st,st'} S_3^m(t',\cdot) \, \d t'}_\cV &\le& \sigma^{-1} \, C_3 \left((\sigma \varphi_0^m(t))^2 + \int_0^t e^{\beta (t-t')} \left(\sigma \varphi_0^m(t')\right)^2 \d t'\right) \\
&\le& \sigma^{-1} \, C_3 \, \frac{e^{\beta t} + \beta - 1}{\beta} \max_{t' \in [0,t]} \left(\sigma \varphi_0^m(t')\right)^2, \label{est_s3m}
\end{eqnarray}
$\PP$-almost surely, where $C_3 < \infty$ is independent of $\sigma$.

\medskip

We collect \eqref{est_s1m}, \eqref{est_s2m}, and \eqref{est_s3m}, and obtain because of \eqref{bound_evol_fam} of Proposition~\ref{prop:family} and \eqref{duhamel_xm_x0m},
\begin{eqnarray}\nonumber
\lefteqn{\vertii{S^m(t,\cdot)}_\cV} \\
&\le& \sigma^{-1} \, C_1 \int_0^t e^{\beta (t-t')} \left(\vertii{\tilde u_0^m(t',\cdot)}_{H^1(\R)}^2 + \vertii{\tilde u_0^m(t',\cdot)}_{H^1(\R)}^3\right) \d t' \nonumber\\
&&+ \sigma^{-1} \, C_2 \int_0^t e^{\beta (t-t')} \sigma \verti{\varphi_0^m(t')} \vertii{\tilde u_0^m(t',\cdot)}_{H^1(\R)} \d t' \nonumber \\
&&+ \sigma^{-1} \, C_3 \, \frac{e^{\beta t} + \beta - 1}{\beta} \max_{t' \in [0,t]} \left(\sigma \varphi_0^m(t')\right)^2 \nonumber\\
&\le& \sigma^{-1} \, C_4 \, \max_{t' \in [0,t]} \left(\left(\vertii{\tilde u_0^m(t',\cdot)}_{H^1(\R)} + \sigma \verti{\varphi_0^m(t')}\right)^2 + \vertii{\tilde u_0^m(t',\cdot)}_{H^1(\R)}^3\right), \label{est_sm}
\end{eqnarray}
$\PP$-almost surely, where $C_4 < \infty$ is independent of $\sigma$. Now, because of \eqref{def_tu0m} we have
\begin{eqnarray*}
\vertii{\tilde u_0^m(t,\cdot)}_{H^1(\R)} &\le& \vertii{u(t,\cdot)}_{H^1(\R)} + \vertii{\hat u(\cdot+st) - \hat u\left(\cdot + st + \sigma \varphi_0^m(t)\right)}_{H^1(\R)} \\
&\le& \vertii{u(t,\cdot)}_{H^1(\R)} + \vertii{\tfrac{\d\hat u}{\d\xi}}_{L^\infty(\R)} \sigma \verti{\varphi_0^m(t)},
\end{eqnarray*}
which in combination with \eqref{stopping_times} and \eqref{est_sm} gives \eqref{est_remainder_ms}.
\end{proof}
%%%%%%%%%%%%%%%%%%%%%%%%%%%%%%%%%%%%%%%%%%%

%%%%%%%%%%%%%%%%%%%%%%%%%%%%%%%%%%%%%%%%%%%
\begin{proof}[Proof of Theorem~\ref{th:multi}~\eqref{item:multi_lim}]
We fix an element $\omega \in \Omega$. If for $\sigma \in (0,1]$ we have $t := \tau_{q,\sigma}(\omega) < T$, then
\begin{eqnarray*}
\sigma^{1-q} &\stackrel{\eqref{stopping_time_X}}{=}& \vertii{X(t,\cdot)}_\cV \stackrel{\eqref{def_txxhx}}{=} \vertii{\tilde X(t,\cdot) - \hat X(\cdot+st)}_\cV \\
&\stackrel{\eqref{decomp_x_sx_s}}{\le}& \sigma \vertii{S^m(t,\cdot)}_\cV + \vertii{\hat X\left(\cdot+s t + \sigma \varphi_0^m(t)\right) - \hat X(\cdot + s t) + \sigma X_0^m(t,\cdot)}_\cV \\
&\stackrel{\eqref{est_remainder_ms}}{\le}& C \sigma^{2-2q} + \vertii{\hat X\left(\cdot+s t + \sigma \varphi_0^m(t)\right) - \hat X(\cdot + s t) + \sigma X_0^m(t,\cdot)}_\cV,
\end{eqnarray*}
where $C < \infty$ is independent of $\sigma \in (0,1]$. This gives with help of Markov's inequality
\begin{eqnarray*}
\lefteqn{\PP\left[\tau_{q,\sigma} < T\right]} \\
&\le& \PP\left[\max_{t \in [0,T]} \vertii{X(t,\cdot)}_\cV \ge \sigma^{1-q}\right] \\
&\le& \PP\left[\max_{t \in [0,T]} \vertii{\hat X\left(\cdot+s t + \sigma \varphi_0^m(t)\right) - \hat X(\cdot + s t) + \sigma X_0^m(t_0,\cdot)}_\cV \ge \sigma^{1-q} \left(1 - C \sigma^{1-q}\right) \right] \\
&\le& \frac{2 \sigma^{2q-2}}{\left(1-C \sigma^{1-q}\right)^2} \, \E\left[\max_{t \in [0,T]} \vertii{\hat X\left(\cdot+s t + \sigma \varphi_0^m(t)\right) - \hat X(\cdot + s t)}_\cV^2\right] \\
&& + \frac{2 \sigma^{2q}}{\left(1-C \sigma^{1-q}\right)^2} \, \E\left[\max_{t \in [0,T]} \vertii{X_0^m(t,\cdot)}_\cV^2\right] \\
&\le& \frac{\tilde C \sigma^{2q}}{\left(1-C \sigma^{1-q}\right)^2} \left(\E\left[\max_{t \in [0,T]} \verti{\varphi_0^m(t)}^2\right] + \E\left[\max_{t \in [0,T]} \vertii{X_0^m(t,\cdot)}_\cV^2\right]\right),
\end{eqnarray*}
where $\tilde C < \infty$ is independent of $\sigma \in (0,1] \cap \left(0,C^{-\frac{1}{1-q}}\right)$. Hence, in order to conclude that $\PP\left[\tau_{q,\sigma} < T\right] \le C \sigma^{2q} \to 0$ as $\sigma \searrow 0$ for some $C < \infty$, it suffices to prove that
\[
\E\left[\max_{t \in [0,T]} \verti{\varphi_0^m(t)}^2\right] < \infty \quad \mbox{and} \quad \E\left[\max_{t \in [0,T]} \vertii{X_0^m(t,\cdot)}_\cV^2\right] < \infty.
\]
This follows from the representations \eqref{mild_phi_x_0_m} and the Burkholder-Davis-Gundy inequality~\cite[Theorem~3.28]{KaratzasShreve1991} or \cite[Lemma~7.7]{DaPratoZabczyk} to estimate the martingale. Indeed, we have
\begin{eqnarray*}
\lefteqn{\E\left[\max_{t \in [0,T]} \verti{\varphi_0^m(t)}^2\right]} \\
&\le& 2 \verti{\inner{\Pi^{\#,0} X^{(0)}_0}{\tfrac{\d \hat X}{\d\xi}}{H}}^2 \\
&& + 2 \E \max_{t \in [0,T]} \verti{\int_0^t \left(1-e^{-m(t-t')}\right) \inner{\Pi^0_{s t'} \left(1,0\right)^\trans \d W(t',\cdot)}{\tfrac{\d\hat X}{\d\xi}(\cdot+st')}{H}}^2 \\
&\stackrel{\eqref{pi_st_phi}}{\le}& 2 \vertii{\Pi^{\#,0}}_{L(H)}^2 \vertii{X^{(0)}_0}_H^2 \vertii{\tfrac{\d\hat X}{\d \xi}}_H^2 + 2 C_1 T \eps Z \vertii{\Pi^{\#,0}}_{L(H)}^2 \vertii{\tfrac{\d\hat X}{\d\xi}}_H^2 \vertii{\sqrt Q}_{L_2(L^2(\R))}^2 < \infty,
\end{eqnarray*}
as well as
\begin{eqnarray*}
\lefteqn{\E\left[\max_{t \in [0,T]} \vertii{X_0^m(t,\cdot)}_\cV^2\right]
} \\
&\le& 3 \max_{t \in [0,T]} \vertii{P_{st,0} X^{(0)}_0}_\cV^2 + 3 \vertii{\tfrac{\d\hat X}{\d\xi}}_\cV^2 \E \max_{t \in [0,T]} \verti{\varphi_0^m(t)}^2 \\
&& + 3 \E \max_{t \in [0,T]} \vertii{\int_0^t P_{st,st'} (1,0)^\trans \d W(t',\cdot)}_\cV^2 \\
&\stackrel{\eqref{bound_evol_fam}}{\le}& 3 e^{2 \beta T} \vertii{X^{(0)}_0}_\cV^2 + 3 \vertii{\tfrac{\d\hat X}{\d\xi}}_\cV^2 \E \max_{t \in [0,T]} \verti{\varphi_0^m(t)}^2 + 3 C_2 \frac{e^{2 \beta T} - 1}{2 \beta} \eps Z \vertii{\sqrt Q}_{L_2(L^2(\R))}^2 < \infty,
\end{eqnarray*}
where $C_1, C_2 < \infty$ are independent of $\sigma$ and Proposition~\ref{prop:family} has been employed.

\medskip

In the same way, we obtain
\[
\PP\left[\tau_{q,\sigma}^m < T\right] \stackrel{\eqref{stopping_time_phi}}{\le} \PP\left[\max_{t \in [0,T]} \verti{\varphi_0^m(t)} \ge \sigma^{-q}\right] \le \sigma^{2 q} \, \E\left[\max_{t \in [0,T]} \verti{\varphi_0^m(t)}^2\right] \to 0 \quad \mbox{as} \quad \sigma \searrow 0. \qedhere
\]
\end{proof}
%%%%%%%%%%%%%%%%%%%%%%%%%%%%%%%%%%%%%%%%%%%

%%%%%%%%%%%%%%%%%%%%%%%%%%%%%%%%%%%%%%%%%%%
\subsection{Immediate relaxation\label{sec:immediate}}
%%%%%%%%%%%%%%%%%%%%%%%%%%%%%%%%%%%%%%%%%%%
Here, we give the proofs of Theorem~\ref{th:immediate} and Proposition~\ref{prop:moment}.

%%%%%%%%%%%%%%%%%%%%%%%%%%%%%%%%%%%%%%%%%%%
\begin{proof}[Proof of Theorem~\ref{th:immediate}~\eqref{item:immediate_xinf0_0}]
Utilizing Proposition~\ref{prop:frozen}~\eqref{item:semi_2} and \eqref{item:riesz} yields
\begin{eqnarray*}
\Pi^0_{st} X_0^\infty(t,\cdot) &\stackrel{\eqref{x_inf_0}}{=}& \Pi^0_{st} P_{st,0} \Pi^\# X^{(0)}_0 + \int_0^t \Pi^0_{st} P_{st,st'} \Pi_{s t'} \left(1,0\right)^\trans \d W(t',\cdot) \\
&\stackrel{\eqref{pi_st_phi}}{=}& \cT_{st} \Pi^{\#,0} P^\#_{st} \Pi^\# X^{(0)}_0 + \int_0^t \cT_{st} \Pi^{\#,0} P^\#_{s(t-t')} \Pi^\# \cT_{-st'} \left(1,0\right)^\trans \d W(t',\cdot) \\
&\stackrel{\eqref{riesz_projection}}{=}& \cT_{st} P^\#_{st} \underbrace{\Pi^{\#,0} \Pi^\#}_{\stackrel{\eqref{riesz_projection}}{=} 0} X^{(0)}_0 + \int_0^t \cT_{st} P^\#_{s(t-t')} \underbrace{\Pi^{\#,0} \Pi^\#}_{\stackrel{\eqref{riesz_projection}}{=} 0} \cT_{-st'} \left(1,0\right)^\trans \d W(t',\cdot) = 0,
\end{eqnarray*}
$\PP$-almost surely, so that because of \eqref{riesz_projection} and \eqref{pi_st_phi} the claim follows.
\end{proof}
%%%%%%%%%%%%%%%%%%%%%%%%%%%%%%%%%%%%%%%%%%%

%%%%%%%%%%%%%%%%%%%%%%%%%%%%%%%%%%%%%%%%%%%
\begin{proof}[Proof of Theorem~\ref{th:immediate}~\eqref{item:immediate_lim}]
Using Theorem~\ref{th:multi}~\eqref{item:multi_mild}, we obtain for the difference of $\varphi_0^m$ and $\varphi_0^\infty$ the following form
\begin{eqnarray}
\left(\varphi_0^\infty - \varphi_0^m\right)(t) &\stackrel{\eqref{mild_phi_0_m}, \eqref{phi_inf_0}}{=}& e^{-mt} \inner{\Pi^{\#,0} X^{(0)}_0}{\tfrac{\d \hat X}{\d \xi}}{H} \nonumber \\
&& + \int_0^t e^{-m (t-t')} \inner{\Pi_{st'}^0 \left(1,0\right)^\trans \d W(t',\cdot)}{\tfrac{\d\hat X}{\d\xi}(\cdot+st')}{H}, \label{phiinf0_m0}
\end{eqnarray}
$\PP$-almost surely. On the other hand, we obtain with \eqref{riesz_projection}, \eqref{pi_st_phi}, \eqref{mild_x_0_m}, and \eqref{x_inf_0}
\[
\left(X_0^\infty - X_0^m\right)(t,\cdot) = - P_{st,0} \Pi^{\#,0} X^{(0)}_0 - \int_0^t P_{st,st'} \Pi^0_{st'} \left(1,0\right)^\trans \d W(t',\cdot) + \varphi_0^m(t) \tfrac{\d\hat X}{\d\xi}(\cdot+st)
\]
and with help of Proposition~\ref{prop:frozen}~\eqref{item:semi_2}, \eqref{item:point}, and \eqref{item:riesz} we conclude that
\begin{eqnarray*}
P_{st,0} \Pi^{\#,0} X^{(0)}_0 &\stackrel{\eqref{frozen_tw_0},  \eqref{riesz_projection}}{=}& \inner{\Pi^{\#,0} X^{(0)}_0}{\tfrac{\d\hat X}{\d\xi}}{H} \cT_{st} P_{st}^\# \tfrac{\d\hat X}{\d\xi} \\
&\stackrel{\eqref{translation}, \eqref{frozen_tw_0}}{=}& \inner{\Pi^{\#,0} X^{(0)}_0}{\tfrac{\d\hat X}{\d\xi}}{H} \tfrac{\d\hat X}{\d\xi}(\cdot+st), \\
P_{st,st'} \Pi^0_{st'} \left(1,0\right)^\trans \d W(t',\cdot) &\stackrel{\eqref{frozen_tw_0}, \eqref{riesz_projection}, \eqref{pi_st_phi}}{=}& \inner{\Pi^0_{st'} \left(1,0\right)^\trans \d W(t',\cdot-st')}{\tfrac{\d\hat X}{\d\xi}}{H} \cT_{st} P_{s(t-t')}^\# \tfrac{\d\hat X}{\d\xi} \\
&\stackrel{\eqref{translation}, \eqref{frozen_tw_0}}{=}& \inner{\Pi^0_{st'} \left(1,0\right)^\trans \d W(t',\cdot)}{\tfrac{\d\hat X}{\d\xi}(\cdot+st')}{H} \tfrac{\d\hat X}{\d\xi}(\cdot+st), 
\end{eqnarray*}
and thus
\begin{eqnarray}
\lefteqn{\left(X_0^\infty - X_0^m\right)(t,\cdot)} \nonumber \\
&=& \left(- \inner{\Pi^{\#,0} X^{(0)}_0}{\tfrac{\d\hat X}{\d\xi}}{H} - \int_0^t \inner{\Pi^0_{st'} \left(1,0\right)^\trans \d W(t',\cdot)}{\tfrac{\d\hat X}{\d\xi}(\cdot+st')}{H} + \varphi_0^m(t)\right) \tfrac{\d\hat X}{\d\xi}(\cdot+st) \nonumber \\
&\stackrel{\eqref{phi_inf_0}}{=}& - \left(\varphi_0^\infty(t) - \varphi_0^m(t)\right) \tfrac{\d\hat X}{\d\xi}(\cdot+st), \quad \mbox{$\PP$-almost surely}. \label{x0inf_x0m}
\end{eqnarray}
Hence, once \eqref{limit_phi_minf} is established, \eqref{limit_x_minf} is immediate from \eqref{x0inf_x0m}.

\medskip

In order to prove \eqref{limit_phi_minf}, notice that obviously
\[
\lim_{m \to \infty} \sup_{t \in [\delta,T]} e^{-mt} \verti{\inner{\Pi^{\#,0} X^{(0)}_0}{\tfrac{\d \hat X}{\d \xi}}{H}} = 0
\]
for any $\delta > 0$. For the remaining term we apply the simplification
\begin{eqnarray}
\lefteqn{\int_0^t e^{-m (t-t')} \inner{\Pi_{st'}^0 \left(1,0\right)^\trans \d W(t',\cdot)}{\tfrac{\d\hat X}{\d\xi}(\cdot+st')}{H}} \nonumber \\
&\stackrel{\eqref{translation}, \eqref{pi_st_phi}}{=}& \int_0^t e^{-m (t-t')} \inner{\Pi^{\#,0} \cT_{-st'} \left(1,0\right)^\trans \d W(t',\cdot)}{\tfrac{\d\hat X}{\d\xi}}{H} \nonumber \\
&\stackrel{\eqref{translation}}{=}& \inner{\Pi^{\#,0} \left(1,0\right)^\trans W(t,\cdot-st)}{\tfrac{\d\hat X}{\d\xi}}{H} \nonumber \\
&& - m \int_0^t e^{-m(t-t')} \inner{\Pi^{\#,0} \left(1,0\right)^\trans W(t',\cdot-st')}{\tfrac{\d\hat X}{\d\xi}}{H} \d t' \nonumber \\
&& + s \int_0^t e^{-m(t-t')} \inner{\Pi^{\#,0} \left(1,0\right)^\trans (\partial_x W)(t',\cdot-st')}{\tfrac{\d\hat X}{\d\xi}}{H} \d t' \label{phi_diff_int} \\
&=& e^{-mt} \inner{\Pi^{\#,0} \left(1,0\right)^\trans W(t,\cdot-st)}{\tfrac{\d\hat X}{\d\xi}}{H} \nonumber \\
&& + m \int_0^t e^{-m(t-t')} \inner{\Pi^{\#,0} \left(1,0\right)^\trans \left(W(t,\cdot-st)-W(t',\cdot-st')\right)}{\tfrac{\d\hat X}{\d\xi}}{H} \d t' \nonumber \\
&& + s \int_0^t e^{-m(t-t')} \inner{\Pi^{\#,0} \left(1,0\right)^\trans (\partial_x W)(t',\cdot-st')}{\tfrac{\d\hat X}{\d\xi}}{H} \d t', \quad \mbox{$\PP$-almost surely.} \nonumber
\end{eqnarray} 
Now, note that $[0,T] \owns t \mapsto W(t,\cdot) \in H^1(\R)$ is, $\PP$-almost surely, H\"older continuous with exponent $\alpha < \frac 1 2$, so that in particular
\begin{eqnarray*}
\lefteqn{\sup_{t,t' \in [0,T]} \frac{\verti{\inner{\Pi^{\#,0} \left(1,0\right)^\trans \left(W(t,\cdot-st)-W(t',\cdot-st)\right)}{\tfrac{\d\hat X}{\d\xi}}{H}}}{\verti{t-t'}^\alpha}} \\
&\stackrel{\eqref{h_norm}}{\le}& \vertii{\Pi^{\#,0}}_{L(H)} \vertii{\tfrac{\d\hat X}{\d\xi}}_H \sqrt{\eps Z} \sup_{t,t' \in [0,T]} \frac{\vertii{W(t,\cdot)-W(t',\cdot)}_{L^2(\R)}}{\verti{t-t'}^\alpha} < \infty, \quad \mbox{$\PP$-almost surely.}
\end{eqnarray*}
Furthermore,
\begin{eqnarray*}
\lefteqn{\verti{\inner{\Pi^{\#,0} \left(1,0\right)^\trans \left(W(t,\cdot-st)-W(t,\cdot-st')\right)}{\tfrac{\d\hat X}{\d\xi}}{H}}} \\
&=& \verti{s \int_{t'}^t \inner{\Pi^{\#,0} \left(1,0\right)^\trans (\partial_x W)(t,\cdot-st'')}{\tfrac{\d\hat X}{\d\xi}}{H} \d t''} \\
&\stackrel{\eqref{h_norm}}{\le}& \verti{s} \vertii{\Pi^{\#,0}}_{L(H)} \sqrt{\eps Z} \vertii{(\partial_x W)(t,\cdot)}_{L^2(\R)} \vertii{\tfrac{\d\hat X}{\d\xi}}_H \verti{t-t'},
\end{eqnarray*}
so that we may conclude
\[
M := \sup_{t,t' \in [0,T]} \frac{\verti{\inner{\Pi^{\#,0} \left(1,0\right)^\trans \left(W(t,\cdot-st)-W(t',\cdot-st')\right)}{\tfrac{\d\hat X}{\d\xi}}{H}}}{\verti{t-t'}^\alpha} < \infty,
\]
$\PP$-almost surely, i.e.,
\begin{align*}
&\verti{m \int_0^t e^{-m(t-t')} \inner{\Pi^{\#,0} \left(1,0\right)^\trans \left(W(t,\cdot-st)-W(t',\cdot-st')\right)}{\tfrac{\d\hat X}{\d\xi}}{H} \d t'} \\
& \quad \le m^{-\alpha} M \int_0^{mT} e^{- \tau} \tau^\alpha \, \d \tau \le m^{-\alpha} M \Gamma(1+\alpha) \to 0 \quad \mbox{as $m \to \infty$}, \quad \mbox{$\PP$-almost surely.}
\end{align*}
By continuity, furthermore
\begin{eqnarray*}
\lefteqn{\sup_{t \in [\delta,T]} e^{-mt} \verti{\inner{\Pi_{st}^0 \left(1,0\right)^\trans W(t,\cdot)}{\tfrac{\d\hat X}{\d\xi}(\cdot+st)}{H}}} \\
&\stackrel{\eqref{pi_st_phi}}{\le}& e^{- m \delta}  \vertii{\Pi^{\#,0}}_{L(H)} \sqrt{\eps Z} \sup_{t \in [\delta,T]} \vertii{W(t,\cdot)}_{L^2(\R)} \vertii{\tfrac{\d\hat X}{\d\xi}}_H \to 0 \quad \mbox{as $m \to \infty$, $\PP$-almost surely},
\end{eqnarray*}
and
\begin{eqnarray*}
\lefteqn{\sup_{t \in [\delta,T]} \verti{s \int_0^t e^{-m(t-t')} \inner{\Pi^{\#,0} \left(1,0\right)^\trans (\partial_x W)(t',\cdot-st')}{\tfrac{\d\hat X}{\d\xi}}{H} \d t'}} \\
&\le& m^{-1} \verti{s} \vertii{\Pi^{\#,0}}_{L(H)} \sqrt{\eps Z} \sup_{t \in [0,T]} \vertii{(\partial_x W)(t,\cdot)}_{L^2(\R)} \vertii{\tfrac{\d\hat X}{\d\xi}}_H \to 0 \quad \mbox{as $m \to \infty$, $\PP$-almost surely.}
\end{eqnarray*}
As a result, we infer that
\begin{equation}\label{point_diff_phi}
\sup_{t \in [\delta,T]} \verti{\left(\varphi_0^\infty-\varphi_0^m\right)(t)} \to 0 \quad \mbox{as $m \to \infty$}, \quad \mbox{$\PP$-almost surely.}
\end{equation}
\medskip

From \eqref{phiinf0_m0} and \eqref{phi_diff_int} we further deduce that $\sup_{t \in [0,T]} \verti{\left(\varphi_0^\infty - \varphi_0^m\right)(t)} \le g$, where
\begin{eqnarray*}
g &:=& \vertii{\Pi^{\#,0}}_{L(H)} \vertii{X^{(0)}_0}_H \vertii{\tfrac{\d\hat X}{\d\xi}}_H + \vertii{\Pi^{\#,0}}_{L(H)} \sqrt{\eps Z} \sup_{t \in [0,T]} \vertii{W(t,\cdot)}_{L^2(\R)} \vertii{\tfrac{\d\hat X}{\d\xi}}_H \\
&& + \verti{s} T \vertii{\Pi^{\#,0}}_{L(H)} \sqrt{\eps Z} \sup_{t \in [0,T]} \vertii{(\partial_x W)(t,\cdot)}_{L^2(\R)} \vertii{\tfrac{\d\hat X}{\d\xi}}_H.
\end{eqnarray*}
With help of the Burkholder-Davis-Gundy inequality~\cite[Lemma~7.7]{DaPratoZabczyk}, we infer that
\[
\E \sup_{t \in [0,T]} \vertii{W(t,\cdot)}_{H^1(\R)} \le C \sqrt T \vertii{\sqrt Q}_{L_2(L^2(\R);H^1(\R))} < \infty,
\]
i.e., $\E\left[g\right] < \infty$. Hence, with \eqref{point_diff_phi} and dominated convergence, we conclude that \eqref{limit_phi_minf} holds true.
\end{proof}
%%%%%%%%%%%%%%%%%%%%%%%%%%%%%%%%%%%%%%%%%%%

%%%%%%%%%%%%%%%%%%%%%%%%%%%%%%%%%%%%%%%%%%%
\begin{proof}[Proof of Theorem~\ref{th:immediate}~\eqref{item:immediate_multi}]
We introduce the auxiliary stopping time
\[
\tau_{q,\sigma,c}^\infty := \inf \left(\left\{ t \in [0,T] \colon \verti{\varphi_0^\infty(t)} \ge \sigma^{-q}-c \right\} \cup \{T\}\right),
\]
where $c \in (0,\sigma^{-q}]$, and start by proving that we can use the multiscale decomposition \eqref{decomp_x_sx_s} of Theorem~\ref{th:multi}~\eqref{item:multi_decomp} on $\left\{\min\left\{\tau_{q,\sigma}, \tau_{q,\sigma,c}^\infty\right\} = T\right\} \cap E^p$, $\PP$-almost surely, where $E^p \in \cF$ is an event with $\PP\left[E^p\right] \ge 1-p$, and where $p \in (0,1]$ is arbitrary. On $\left\{\min\left\{\tau_{q,\sigma}, \tau_{q,\sigma,c}^\infty\right\} = T\right\}$ we have
\[
\varphi_0^\infty(0) + \frac c 2 \stackrel{\eqref{phi_inf_0}}{=} \verti{\inner{\Pi^{\#,0} X^{(0)}_0}{\tfrac{\d \hat X}{\d \xi}}{H}} + \frac c 2 < \sigma^{-q}
\]
Now, note that for any $m>0$ we have
\begin{eqnarray*}
\varphi_0^m(t) &\stackrel{\eqref{mild_phi_0_m}}{=}& \left(1-e^{-mt}\right) \inner{\Pi^{\#,0} X^{(0)}_0}{\tfrac{\d\hat X}{\d\xi}}{H} \\
&& + \int_0^t \left(1-e^{-m (t-t')}\right) \inner{\Pi_{st'}^0 \left(1,0\right)^\trans \d W(t',\cdot)}{\tfrac{\d \hat X}{\d\xi}(\cdot+st')}{H} \\
&\stackrel{\eqref{translation}, \eqref{pi_st_phi}}{=}& \left(1-e^{-mt}\right) \inner{\Pi^{\#,0} X^{(0)}_0}{\tfrac{\d\hat X}{\d\xi}}{H} \\
&& + \int_0^t \left(1-e^{-m (t-t')}\right) \inner{\Pi^{\#,0} \cT_{-st'} \left(1,0\right)^\trans \d W(t',\cdot)}{\tfrac{\d \hat X}{\d\xi}}{H} \\
&\stackrel{\eqref{translation}}{=}& \left(1-e^{-mt}\right) \inner{\Pi^{\#,0} X^{(0)}_0}{\tfrac{\d\hat X}{\d\xi}}{H} \\
&& - m \int_0^t e^{-m (t-t')} \inner{\Pi^{\#,0} \left(1,0\right)^\trans W(t',\cdot-st')}{\tfrac{\d \hat X}{\d\xi}}{H} \d t' \\
&& + s \int_0^t \left(1-e^{-m (t-t')}\right) \inner{\Pi^{\#,0} \left(1,0\right)^\trans (\partial_x W)(t',\cdot-st')}{\tfrac{\d \hat X}{\d\xi}}{H} \d t',
\end{eqnarray*}
$\PP$-almost surely, where we have integrated by parts in the last step. Hence, for $\delta > 0$ it holds $\sup_{t \in [0,\delta]} \verti{\varphi_0^m(t)} \le g(\delta)$, $\PP$-almost surely, where
\begin{eqnarray*}
g(t) &:=& \verti{\inner{\Pi^{\#,0} X^{(0)}_0}{\tfrac{\d \hat X}{\d \xi}}{H}} + \vertii{\Pi^{\#,0}}_{L(H)} \sup_{t' \in [0,t]} \vertii{W(t',\cdot)}_H \vertii{\tfrac{\d\hat X}{\d\xi}}_H \\
&& + \verti{s} t \vertii{\Pi^{\#,0}}_{L(H)} \sqrt{\eps Z} \sup_{t' \in [0,t]} \vertii{(\partial_x W)(t',\cdot)}_{L^2(\R)} \vertii{\tfrac{\d\hat X}{\d\xi}}_H
\end{eqnarray*}
is $m$-independent. It is obvious that $g$ is, $\PP$-almost surely, non-decreasing and by the Burkholder-Davis-Gundy inequality~\cite[Lemma~7.7]{DaPratoZabczyk}, as in the proof of Theorem~\ref{th:immediate}~\eqref{item:immediate_lim}, we see that
\[
\E\left[\sup_{t \in [0,\delta]} \verti{\varphi_0^m(t)}\right] \le \E\left[g(\delta)\right] \to \verti{\inner{\Pi^{\#,0} X^{(0)}_0}{\tfrac{\d \hat X}{\d \xi}}{H}} \quad \mbox{as $\delta \searrow 0$.}
\]
Since
\[
g(\delta) \ge g(0) = \verti{\inner{\Pi^{\#,0} X^{(0)}_0}{\tfrac{\d \hat X}{\d \xi}}{H}},
\]
we conclude that
\[
\PP\left[\sup_{t \in [0,\delta]} \verti{\varphi_0^m(t)} \ge \verti{\inner{\Pi^{\#,0} X^{(0)}_0}{\tfrac{\d \hat X}{\d \xi}}{H}} + \frac c 2\right] \to 0 \quad \mbox{as $\delta \searrow 0$}
\]
and hence, for any $p \in (0,1]$ there exists $\delta \in [0,T]$ small enough such that
\[
\sup_{t \in [0,\delta]} \verti{\varphi_0^m(t)} \le \verti{\inner{\Pi^{\#,0} X^{(0)}_0}{\tfrac{\d \hat X}{\d \xi}}{H}} + \frac c 2 < \sigma^{-q}
\]
on an event $E_1^p \in \cF$ with $\PP\left[E_1^p\right] \ge 1 - \frac p 2$. Using the convergence \eqref{limit_phi_minf}, we infer that for any $p \in (0,1]$ we find $m \in (0,\infty)$ sufficiently large with
\[
\sup_{t \in [\delta,T]} \verti{\varphi_0^m(t)} \le \sup_{t \in [\delta,T]} \verti{\varphi_0^\infty(t)} + c \le \sigma^{-q} 
\]
on an event $E_2^p \in \cF$ with $\PP\left[E_2^p\right] \ge 1 - \frac p 2$. Define $E^p := E_1^p \cap E_2^p$, then
\[
\PP\left[E^p\right] = 1 - \PP\left[\Omega \setminus E^p\right] \ge 1 - \PP\left[\Omega \setminus E_1^p\right] - \PP\left[\Omega \setminus E_2^p\right] = - 1 + \PP\left[E_1^p\right] + \PP\left[E_2^p\right] \ge 1 - p.
\]
In total, on $\left\{\min\left\{\tau_{q,\sigma}, \tau_{q,\sigma,c}^\infty\right\} = T \right\} \cap E^p$ we can apply the multiscale decomposition \eqref{decomp_x_sx_s} of Theorem~\ref{th:multi}~\eqref{item:multi_decomp} in order to obtain
\begin{eqnarray*}
\lefteqn{\sigma \vertii{S^\infty(t,\cdot)}_\cV} \\
&\stackrel{\eqref{multi_imm_relax}}{=}& \vertii{\tilde X(t,\cdot) - \hat X (\cdot + st + \sigma \varphi_0(t)) - \sigma X_0^\infty(t,\cdot)}_\cV \\
&\stackrel{\eqref{decomp_x_sx_s}}{\le}& \sigma \vertii{S^m(t,\cdot)}_V + \sigma \vertii{X_0^\infty(t,\cdot) - X^m_0(t,\cdot)}_\cV \\ 
&& +\vertii{\hat X \left(\cdot + st + \sigma \varphi_0^\infty(t)\right) - \hat X \left(\cdot + st + \sigma \varphi^m_0(t)\right)}_\cV \\
&\stackrel{\eqref{est_remainder_ms}}{\le}& C \sigma^{2-2q} \left(1 + \sigma^{1-q}\right) + \sigma \vertii{X_0^{\infty}(t,\cdot) - X^m_0(t)}_\cV \\
&& + \sqrt{\sigma \verti{\varphi_0^\infty(t) - \varphi_0^m(t)}} \sqrt Z \sqrt{2 \eps \vertii{\tfrac{\d \hat u}{\d\xi}}_{W^{1,\infty}(\R)} \vertii{\hat u}_{W^{1,1}(\R)} + 2 \vertii{\tfrac{\d \hat v}{\d\xi}}_{W^{1,\infty}(\R)} \vertii{\hat v}_{W^{1,1}(\R)}} \\
&\to& C \sigma^{2-2q} \left(1 + \sigma^{1-q}\right) \quad \mbox{as $m \to \infty$}, \quad \mbox{$\PP$-almost surely,}
\end{eqnarray*}
i.e., \eqref{est_s_inf} holds true on $\left\{\min\left\{\tau_{q,\sigma}, \tau_{q,\sigma,c}^\infty\right\} = T \right\} \cap E^p$, where $C < \infty$ is independent of $\sigma$, $p$, and $c$. Setting $E := \bigcup_{p \in (0,1]} E^p$, it follows $\PP\left[E\right] = 1$ and we infer that the point-wise estimate \eqref{est_s_inf} holds true on $\left\{\min\left\{\tau_{q,\sigma}, \tau_{q,\sigma,c}^\infty\right\} = T \right\}$, too. Since
\[
\left\{\min\left\{\tau_{q,\sigma}, \tau_{q,\sigma}^\infty\right\} = T \right\} = \bigcup_{c \in (0,\sigma^{-q}]} \left\{\min\left\{\tau_{q,\sigma}, \tau_{q,\sigma,c}^\infty\right\} = T \right\},
\]
we infer that \eqref{est_s_inf} is also valid on $\left\{\min\left\{\tau_{q,\sigma}, \tau_{q,\sigma}^\infty\right\} = T\right\}$.
\end{proof}
%%%%%%%%%%%%%%%%%%%%%%%%%%%%%%%%%%%%%%%%%%%

%%%%%%%%%%%%%%%%%%%%%%%%%%%%%%%%%%%%%%%%%%%
\begin{proof}[Proof of Theorem~\ref{th:immediate}~\eqref{item:immediate_stop}]
Because of Theorem~\ref{th:multi}~\eqref{item:multi_lim}, it suffices to prove $\PP\left[\tau_{q,\sigma}^\infty < T\right] \le C \sigma^{2q} \to 0$ as $\sigma \searrow 0$ for some $C < \infty$. Indeed, by Markov's inequality, we obtain
\[
\PP\left[\tau_{q,\sigma}^\infty < T\right] \stackrel{\eqref{t_q_sig_inf}}{\le} \PP\left[\max_{t \in [0,T]} \verti{\varphi_0^\infty(t)} \ge \sigma^{-q}\right] \le \sigma^{2 q} \, \E\left[\max_{t \in [0,T]} \verti{\varphi_0^\infty(t)}^2\right].
\]
Now note that $\E\left[\max_{t \in [0,T]} \verti{\varphi_0^\infty(t)}^2\right] < \infty$ because of the representation \eqref{phi_inf_0}, where the martingale can be estimated using the Burkholder-Davis-Gundy inequality~\cite[Theorem~3.28]{KaratzasShreve1991} or \cite[Lemma~7.7]{DaPratoZabczyk}.
\end{proof}
%%%%%%%%%%%%%%%%%%%%%%%%%%%%%%%%%%%%%%%%%%%

%%%%%%%%%%%%%%%%%%%%%%%%%%%%%%%%%%%%%%%%%%%
\begin{proof}[Proof of Theorem~\ref{th:immediate}~\eqref{item:immediate_min}]
We use
\begin{eqnarray*}
\lefteqn{\left. \partial_\varphi \vertii{\Pi_{st}^0 \left(\tilde X (t,\cdot) - \hat X(\cdot + st + \sigma \varphi)\right)}_H^{2} \right|_{\varphi = \varphi_0^\infty(t)}} \\
&=& - 2 \sigma \inner{\Pi^0_{st} \left(\tilde X(t,\cdot) - \hat X\left(\cdot + st + \sigma \varphi_0^\infty(t)\right)\right)}{\tfrac{\d \hat X}{\d\xi}\left(\cdot + s t + \sigma \varphi_0^\infty(t)\right)}{H} \\
&\stackrel{\eqref{multi_imm_relax}}{=}& - 2 \sigma^2 \inner{\Pi^0_{st} \left(X_0^\infty(t,\cdot) + S^\infty(t,\cdot)\right)}{\tfrac{\d\hat X}{\d \xi}(\cdot + st + \sigma \varphi_0^\infty(t))}{H} \\
&=& - 2 \sigma^2 \inner{\Pi^0_{st} S^\infty(t,\cdot)}{\tfrac{\d\hat X}{\d \xi}(\cdot + st + \sigma \varphi_0^\infty(t))}{H}, \quad \mbox{$\PP$-almost surely},
\end{eqnarray*}
where we have used Theorem~\ref{th:immediate}~\eqref{item:immediate_xinf0_0} in the last step. Hence,
\begin{eqnarray*}
\verti{\left. \partial_\varphi \vertii{\Pi_{st+\varphi}^0 \left(\tilde X (t,\cdot) - \hat X(\cdot + st + \sigma \varphi)\right)}_H^{2} \right|_{\varphi = \varphi_0^\infty(t)}} &\stackrel{\eqref{pi_st_phi}}{\le}& 2 \sigma^2 \vertii{\Pi^{\#,0}}_{L(H)} \vertii{\tfrac{\d\hat X} {\d\xi}}_H \vertii{S^\infty(t,\cdot)}_H \\
&\stackrel{\eqref{est_s_inf}}{\le}& C \, \sigma^{3-2q} \left(1 + \sigma^{1-q}\right),
\end{eqnarray*}
$\PP$-almost surely, where $C < \infty$ is independent of $\sigma$. This proves \eqref{min_0_inf_1}.

\medskip

For the second derivative, we obtain
\begin{eqnarray*}
\lefteqn{\left. \partial_\varphi^2 \vertii{\Pi_{st}^0 \left(\tilde X (t,\cdot) - \hat X(\cdot + st + \sigma \varphi)\right)}_H^2 \right|_{\varphi = \varphi_0^\infty(t)}} \\
&=& - 2 \sigma \left. \partial_\varphi \inner{\Pi_{st}^0 \left(\tilde X(t,\cdot) - \hat X\left(\cdot+st+\sigma\varphi\right)\right)}{\tfrac{\d \hat X}{\d\xi}\left(\cdot+st+\sigma\varphi\right)}{H} \right|_{\varphi = \varphi_0^\infty(t)} \\
&=& 2 \sigma^2 \inner{\Pi_{st}^0 \tfrac{\d\hat X}{\d\xi}\left(\cdot+st+\sigma\varphi_0^\infty(t)\right)}{\tfrac{\d \hat X}{\d\xi}\left(\cdot+st+\sigma\varphi_0^\infty(t)\right)}{H} \\
&& - 2 \sigma^2 \inner{\Pi_{st}^0 \left(\tilde X(t,\cdot) - \hat X\left(\cdot+st+\sigma\varphi_0^\infty(t)\right)\right)}{\tfrac{\d^2 \hat X}{\d\xi^2}\left(\cdot+st+\sigma\varphi_0^\infty(t)\right)}{H} \\
&\stackrel{\eqref{multi_imm_relax}}{=}& 2 \sigma^2 \inner{\Pi_{st}^0 \tfrac{\d\hat X}{\d\xi}\left(\cdot+st+\sigma\varphi_0^\infty(t)\right)}{\tfrac{\d \hat X}{\d\xi}\left(\cdot+st+\sigma\varphi_0^\infty(t)\right)}{H} \\
&& - 2 \sigma^2 \inner{\Pi^0_{st} \left(X_0^\infty(t,\cdot) + S^\infty(t,\cdot)\right)}{\tfrac{\d^2\hat X}{\d \xi^2}(\cdot + st + \sigma \varphi_0^\infty(t))}{H} \\
&\stackrel{\eqref{pi_st_phi}}{=}& 2 \sigma^2 \inner{\Pi^{\#,0} \tfrac{\d\hat X}{\d\xi}\left(\cdot+\sigma\varphi_0^\infty(t)\right)}{\tfrac{\d \hat X}{\d\xi}\left(\cdot+\sigma\varphi_0^\infty(t)\right)}{H} \\
&& - 2 \sigma^2 \inner{\Pi^0_{st} S^\infty(t,\cdot)}{\tfrac{\d^2\hat X}{\d \xi^2}(\cdot + st + \sigma \varphi_0^\infty(t))}{H}, \quad \mbox{$\PP$-almost surely,}
\end{eqnarray*}
where Theorem~\ref{th:immediate}~\eqref{item:immediate_xinf0_0} was used in the last step once more. Further utilizing
\begin{eqnarray*}
\lefteqn{2 \sigma^2 \inner{\Pi^{\#,0} \tfrac{\d\hat X}{\d\xi}\left(\cdot+\sigma\varphi_0^\infty(t)\right)}{\tfrac{\d \hat X}{\d\xi}\left(\cdot+\sigma\varphi_0^\infty(t)\right)}{H}} \\
&=& 2 \sigma^2 \vertii{\tfrac{\d\hat X}{\d\xi}}_H^2 + 2 \sigma^3 \int_0^{\varphi_0^\infty(t)} \inner{\Pi^{\#,0} \tfrac{\d^2\hat X}{\d\xi^2}\left(\cdot+\sigma\varphi\right)}{\tfrac{\d \hat X}{\d\xi}}{H} \d\varphi \\
&& + 2 \sigma^3 \int_0^{\varphi_0^\infty(t)} \inner{\tfrac{\d \hat X}{\d\xi}}{\tfrac{\d^2 \hat X}{\d\xi^2}\left(\cdot+\sigma\varphi\right)}{H} \d\varphi \\
&& + 2 \sigma^4 \int_0^{\varphi_0^\infty(t)} \int_0^{\varphi_0^\infty(t)} \inner{\Pi^{\#,0} \tfrac{\d^2\hat X}{\d^2\xi}\left(\cdot+\sigma\varphi\right)}{\tfrac{\d^2 \hat X}{\d\xi^2}\left(\cdot+\sigma\varphi'\right)}{H} \d\varphi \, \d\varphi',
\end{eqnarray*}
we can estimate
\begin{align*}
\verti{\int_0^{\varphi_0^\infty(t)} \inner{\Pi^{\#,0} \tfrac{\d^2\hat X}{\d\xi^2}\left(\cdot+\sigma\varphi\right)}{\tfrac{\d \hat X}{\d\xi}}{H} \d\varphi} &\le \sigma^{-q} \vertii{\Pi^{\#,0}}_{L(H)} \vertii{\tfrac{\d^2 \hat X}{\d\xi^2}}_H \vertii{\tfrac{\d\hat X}{\d\xi}}_H, \\
\verti{\int_0^{\varphi_0^\infty(t)} \inner{\tfrac{\d \hat X}{\d\xi}}{\tfrac{\d^2 \hat X}{\d\xi^2}\left(\cdot+\sigma\varphi\right)}{H} \d\varphi} &\le \sigma^{-q} \vertii{\tfrac{\d\hat X}{\d\xi}}_H \vertii{\tfrac{\d^2 \hat X}{\d\xi^2}}_H,
\end{align*}
and
\[
\verti{\int_0^{\varphi_0^\infty(t)} \int_0^{\varphi_0^\infty(t)} \inner{\Pi^{\#,0} \tfrac{\d^2\hat X}{\d^2\xi}\left(\cdot+\sigma\varphi\right)}{\tfrac{\d^2 \hat X}{\d\xi^2}\left(\cdot+\sigma\varphi'\right)}{H} \d\varphi \, \d\varphi'} \le \sigma^{-2q} \vertii{\Pi^{\#,0}}_{L(H)} \vertii{\tfrac{\d^2 \hat X}{\d\xi^2}}_H ^2,
\]
as well as
\begin{eqnarray*}
\verti{\inner{\Pi^0_{st} S^\infty(t,\cdot)}{\tfrac{\d\hat X}{\d \xi}(\cdot + st + \sigma \varphi_0^\infty(t))}{H}} &\stackrel{\eqref{pi_st_phi}}{\le}& \vertii{\Pi^{\#,0}}_{L(H)} \vertii{\tfrac{\d^2\hat X}{\d\xi^2}}_H \vertii{S^\infty(t,\cdot)}_H \\
&\stackrel{\eqref{est_s_inf}}{\le}& C \, \sigma^{1-2q} \left(1 + \sigma^{1-q}\right),
\end{eqnarray*}
$\PP$-almost surely, where $C < \infty$ is independent of $\sigma$, we conclude that \eqref{min_0_inf_2} holds true, too.
\end{proof}
%%%%%%%%%%%%%%%%%%%%%%%%%%%%%%%%%%%%%%%%%%%

%%%%%%%%%%%%%%%%%%%%%%%%%%%%%%%%%%%%%%%%%%%
\begin{proof}[Proof of Proposition~\ref{prop:moment}]
From \eqref{x_inf_0} we obtain with help of \eqref{translation}, Proposition~\ref{prop:frozen}~\eqref{item:semi_2} and \eqref{item:riesz}, and \eqref{pi_st_phi}
\begin{eqnarray*}
\vertii{X_0^\infty(t,\cdot)}_H &\le& \vertii{P_{st}^\# \Pi^\# X^{(0)}_0}_H + \vertii{\int_0^t P_{s(t-t')}^\# \Pi^\# \cT_{-st'} \left(1,0\right)^\trans \d W(t',\cdot)}_H \\
&\stackrel{\eqref{est_semigroupFrozenWave}}{\le}& C_{\vartheta} e^{-\vartheta t} \vertii{X^{(0)}_0}_H + \vertii{\int_0^t P_{s(t-t')}^\# \Pi^\# \cT_{-st'} \left(1,0\right)^\trans \d W(t',\cdot)}_H,
\end{eqnarray*}
$\PP$-almost surely, where $\vartheta$ has been introduced in \eqref{def_kappa}. Therefore, using It\^o's isometry, the second moment can be bounded as follows
\[
\E\left[ \vertii{X_0^\infty(t,\cdot)}_H^{2} \right] \le 2 C_{\vartheta}^{2} e^{-2\vartheta t} \vertii{X^{(0)}_0}_H^2
+ 2 \int_0^t \vertii{P_{s(t-t')}^\# \Pi^\# \cT_{-st'} \left(1,0\right)^\trans \sqrt Q}_{L_2\left(L^2(\R);H\right)}^2 \d t'.
\]
For an orthonormal basis $(e_k)_{k \in \N}$ of $L^2(\R)$ we can write 
\begin{eqnarray*}
\vertii{P_{s(t-t')}^\# \Pi^\# \cT_{-st'} \left(1,0\right)^\trans \sqrt Q}_{L_2\left(L^2(\R);H\right)}^2 &=& \sum_{k=1}^\infty \vertii{P_{s(t-t')}^\# \Pi^\# \cT_{-st'} \left(1,0\right)^t \sqrt{Q} e_k}^2_H \\
&\leq& C_{\vartheta}^{2} e^{-2\vartheta(t-t')} \vertii{\Pi^\#}_{L(H)}^{2} \eps Z \sum_{k=1}^\infty \vertii{\sqrt Q e_k}^2_{L^2(\R)} \\
&\leq& C_{\vartheta}^{2} e^{-2\vartheta(t-t')} \vertii{\Pi^\#}_{L(H)}^{2} \eps Z \vertii{\sqrt Q}^2_{L_2\left(L^2(\R)\right)},
\end{eqnarray*}
so that eventually we arrive at \eqref{moment_bound}.
\end{proof}

\begin{proof}[Proof of Proposition \ref{prop:moment-without-adaption}]
Suppose $Q = Q_N \in L_2(L^2(\R);H^1(\R))$ to be specified further below. With \eqref{x_phi}--\eqref{remainder}, the deviations around the traveling wave without stochastic velocity adaption ($\varphi = 0$) satisfy the following mild-solution formula
\[
X(t,\cdot) = \sigma P_{st,0} X^{(0)}_0 + \int_0^t P_{st,st'}R_0(t',X(t',\cdot),\cdot) \d t' + \sigma \int_0^t P_{st,st'}  (1,0)^\trans \d W(t',\cdot),
\]
which can be justified with arguments analogous to those given at the beginning of the proof of Proposition~\ref{prop:reg_var}. Thanks to Proposition~\ref{prop:frozen}~\eqref{item:semi_2} and \eqref{translation}, in the moving frame this corresponds to 
\[
X(t,\cdot-st) = \sigma P_{st}^\# X^{(0)}_0 + \int_0^t P_{s(t-t')}^\#  R_0(t',X(t',\cdot-st'),\cdot-st') \d t' + \sigma \int_0^t P_{s(t-t')}^\#  (1,0)^\trans \d W(t',\cdot-st').
\]
We are interested in the deviations in direction of the derivative of the traveling wave which is given by the projection $\Pi^{\#,0}$ defined in \eqref{riesz_projection}. Note that by \eqref{frozen_tw_0} we have $\Pi^{\#,0} P_{st}^\# = P_{st}^\# \Pi^{\#,0} = \Pi^{\#,0}$, so that
\[
\Pi^{\#,0}X(t,\cdot-st) = \sigma \Pi^{\#,0} X^{(0)}_0 + \int_0^t \Pi^{\#,0}  R_0(t',X(t',\cdot-st'),\cdot-st') \d t' + \sigma \int_0^t \Pi^{\#,0}  (1,0)^\trans \d W(t',\cdot-st').
\]
Hence,
\begin{align*}
\inner{\Pi^{\#,0} X(t,\cdot-st)}{\tfrac{\d \hat X}{\d\xi}}{H} 
=&\sigma\inner{\Pi^{\#,0} X^{(0)}_0}{\tfrac{\d \hat X}{\d\xi}}{H} \\
&+ \underbrace{\int_0^t \inner{\Pi^{\#,0} R_0(t',X(t',\cdot-st'),\cdot-st')}{\tfrac{\d \hat X}{\d\xi}}{H} \d t'}_{=: I_1}\\
&+ \underbrace{ \sigma\int_0^t\inner{\Pi^{\#,0} (1,0)^\trans \d W(t',\cdot-st')}{\tfrac{\d \hat X}{\d\xi}}{H}}_{=: I_2}.
\end{align*}
In order to compute $\E\left[ \inner{\Pi^{\#,0} X(t,\cdot-st)}{\tfrac{\d \hat X}{\d\xi}}{H}^2\mathds{1}_{\left\{\min\left\{\tau_{q,\sigma}, \tau_{q,\sigma}^\infty\right\} = T \right\}}\right]$ we develop the square and consider the terms separately. For the first term note that
\begin{align*}
&\E\left[\sigma^2\inner{\Pi^{\#,0} X^{(0)}_0}{\tfrac{\d \hat X}{\d\xi}}{H}^2\mathds{1}_{\left\{\min\left\{\tau_{q,\sigma}, \tau_{q,\sigma}^\infty\right\} = T \right\}}\right]\\
&= \sigma^2\inner{\Pi^{\#,0} X^{(0)}_0}{\tfrac{\d \hat X}{\d\xi}}{H}^2 - \PP\left(\min\left\{\tau_{q,\sigma}, \tau_{q,\sigma}^\infty\right\} < T \right) \sigma^2\inner{\Pi^{\#,0} X^{(0)}_0}{\tfrac{\d \hat X}{\d\xi}}{H}^2 \\
&= \sigma^2\inner{\Pi^{\#,0} X^{(0)}_0}{\tfrac{\d \hat X}{\d\xi}}{H}^2 + o(\sigma^2),
\end{align*}
since by Theorem \ref{th:immediate} \eqref{item:immediate_stop} we have $\PP\left(\min\left\{\tau_{q,\sigma}, \tau_{q,\sigma}^\infty\right\} < T \right) \leq C \sigma^{2q}$ for some $C<\infty$ depending on $t$ and $N$ but independent of $\sigma$.

\medskip

Now note that using \eqref{remainder}, with the same argumentation as in the estimate of $S_1^m$ (cf.~\eqref{def_s1m}) in the proof of Theorem~\ref{th:multi}~\eqref{item:multi_decomp} it holds
\[
\E\left[I_1^2\mathds{1}_{\left\{\min\left\{\tau_{q,\sigma}, \tau_{q,\sigma}^\infty\right\} = T \right\}}\right] \leq C^2 \sigma^2 \sigma^{2-4q}(1+\sigma^{1-q})^2
\]
for a constant $C < \infty$ depending on $t$ and $N$ but independent of $\sigma$. For the third term we write again
\[
\E\left[I_2^2\mathds{1}_{\left\{\min\left\{\tau_{q,\sigma}, \tau_{q,\sigma}^\infty\right\} = T \right\}}\right] = \E\left[I_2^2\right] - \E\left[I_2^2\mathds{1}_{\left\{\min\left\{\tau_{q,\sigma}, \tau_{q,\sigma}^\infty\right\} < T \right\}}\right],
\]
and we will recognize that $\E\left[I_2^2\mathds{1}_{\left\{\min\left\{\tau_{q,\sigma}, \tau_{q,\sigma}^\infty\right\} < T \right\}}\right] = o(\sigma^2)$ depending on $t$ and $N$ by dominated convergence together with Theorem~\ref{th:immediate}~\eqref{item:immediate_stop} and the fact that $\sigma^{-2} I_2^2 \mathds{1}_{\left\{\min\left\{\tau_{q,\sigma}, \tau_{q,\sigma}^\infty\right\} < T \right\}} \leq \sigma^{-2}I_2^2,$ where the latter is a random variable independent of $\sigma$ whose expectation is finite  by the computations that follow.

\medskip

Take an orthonormal basis $(e_k)_{k \in \N}$ of $L^2(\R)$ with $e_k \in C_\mathrm{c}^\infty(\R)$ for each $k\in \N$ (existence of such a basis follows by applying the Gram-Schmidt algorithm to a countable dense subset of $L^2(\R)$ in $C_\mathrm{c}^\infty(\R)$). With It\^o's isometry we obtain
\begin{equation*}
\E\left[I_2^2\right] = \sigma^2 \int_0^t \sum_{k=1}^\infty \inner{\Pi^{\#,0} \cT_{-st'} (1,0)^\trans \sqrt{Q_N}e_k}{\tfrac{\d \hat X}{\d\xi}}{H}^2 \d t'.
\end{equation*} 
Now take the sequence of Hilbert Schmidt operators $(Q_N)_{N \in \N}$ with
\[
\sqrt{Q_N}e_k = \left\{
\begin{matrix}
e_k & : k \leq N,\\ 
0 & : k > N,
\end{matrix}
\right.
\]
and compute the with $N$ increasing second moment of $I_2$ for each $N \in \N$
\begin{eqnarray*}
\E\left[I_2^2\right] 
&=& \sigma^2 \int_0^t \sum_{k=1}^N \inner{\Pi^{\#,0} (1,0)^\trans e_k(\cdot-st')}{\tfrac{\d \hat X}{\d\xi}}{H}^2 \d t' \\
&\stackrel{\eqref{h_norm}}{=}& \sigma^2 \eps^2 Z^2 \int_0^t \sum_{k=1}^N \inner{e_k(\cdot-st')}{ \left(\left(\Pi^{\#,0}\right)^*\tfrac{\d \hat X}{\d\xi}\right)_1}{L^2(\R)}^2 \d t',
\end{eqnarray*}
where $\left(\left(\Pi^{\#,0}\right)^*\tfrac{\d \hat X}{\d\xi}\right)_1$ denotes the first component of $\left(\Pi^{\#,0}\right)^*\tfrac{\d \hat X}{\d\xi}.$
Taking the limit $N \rightarrow \infty$, we obtain with Parseval's identity
\[
\E\left[I_2^2\right] \rightarrow \sigma^2 \eps^2 Z^2 \int_0^t  \left\| \left(\left(\Pi^{\#,0}\right)^*\tfrac{\d \hat X}{\d\xi}\right)_1\right\|_{L^2(\R)}^2 \d t' = \sigma^2 \eps Z t \left\| (1,0)^\trans \left(\Pi^{\#,0}\right)^*\tfrac{\d \hat X}{\d\xi}\right\|_{H}^2.
\]
By the previous estimates, the mixed terms $\E\left[ 2 \sigma\inner{\Pi^{\#,0} X^{(0)}_0}{\tfrac{\d \hat X}{\d\xi}}{H} I_1\mathds{1}_{\left\{\min\left\{\tau_{q,\sigma}, \tau_{q,\sigma}^\infty\right\} = T \right\}}\right]$ and $\E[2I_1 I_2\mathds{1}_{\left\{\min\left\{\tau_{q,\sigma}, \tau_{q,\sigma}^\infty\right\} = T \right\}}]$ are of order $o(\sigma^2)$ depending on $t$ and $N$ by the Cauchy-Schwarz inequality. For the third mixed term, we write again
\begin{eqnarray*}
\lefteqn{\E\left[ 2\sigma\inner{\Pi^{\#,0} X^{(0)}_0}{\tfrac{\d \hat X}{\d\xi}}{H} I_2 \mathds{1}_{\left\{\min\left\{\tau_{q,\sigma}, \tau_{q,\sigma}^\infty\right\} = T \right\}} \right]} \\
&=& \E\left[ 2\sigma\inner{\Pi^{\#,0} X^{(0)}_0}{\tfrac{\d \hat X}{\d\xi}}{H} I_2 \right] - \E\left[ 2\sigma\inner{\Pi^{\#,0} X^{(0)}_0}{\tfrac{\d \hat X}{\d\xi}}{H} I_2 \mathds{1}_{\left\{\min\left\{\tau_{q,\sigma}, \tau_{q,\sigma}^\infty\right\} < T \right\}} \right] \\
&=& o(\sigma^2)
\end{eqnarray*}
depending on $t$ and $N$ since the first term in the second line is equal to $0$ because $I_2$ is a martingale (and equal to $0$ for $t=0$) and the second term in the second line is of order $o(\sigma^2)$ by the Cauchy-Schwarz inequality in conjunction with dominated convergence as $\sigma \searrow 0.$

\medskip

In order to ensure linear growth in time, we prove that $(1,0)^\trans \left(\Pi^{\#,0}\right)^*\tfrac{\d \hat X}{\d\xi} \neq 0$ by contradiction. Assume that $(1,0)^\trans \left(\Pi^{\#,0}\right)^*\tfrac{\d \hat X}{\d\xi} = 0$. Then for $\phi \in L^2(\R)$ we have
\[
0 = \eps Z \inner{\left(\left(\Pi^{\#,0}\right)^*\tfrac{\d \hat X}{\d\xi}\right)_1}{\phi}{L^2(\R)} = \inner{\tfrac{\d \hat X}{\d\xi}}{{\Pi^{\#,0}}(1,0)^\trans\phi}{H}.
\]
Since $\Pi^{\#,0}$ is a projection on the linear subspace generated by  $\tfrac{\d \hat X}{\d\xi}$ (see \eqref{riesz_projection}), we obtain
$\Pi^{\#,0}(1,0)^\trans\phi = 0.$ Hence, for any $Y \in H$ it holds
\[
0 = \inner{Y}{{\Pi^{\#,0}}(1,0)^\trans\phi}{H} = \eps Z \inner{\left(\left(\Pi^{\#,0}\right)^*Y\right)_1}{\phi}{L^2(\R)}.
\]
The above equality is in particular valid for a non-trivial eigenfunction $Y = (w,q)^\trans \neq 0$ of $\left(\cL^\#\right)^*$ with eigenvalue $0$, that is, $\left(\left(\Pi^{\#,0}\right)^*Y\right)_1 = 0$. Note that such an eigenfunction exists because of \eqref{frozen_tw_0} or Proposition~\ref{prop:frozen}~\eqref{item:point}, respectively. Indeed, we have $\ind(\cL^\#) = \ind\left(\left(\cL^\#\right)^*\right) = 0$ by Definition~\ref{def:spectrum}~\eqref{item:def_point} and because for any $\tilde Y \in H_\C$ we have
\[
0 = \inner{\tilde Y}{\cL^\# \tfrac{\d \hat X}{\d\xi}}{H_\C} = \inner{\left(\cL^\#\right)^* \tilde Y}{\tfrac{\d \hat X}{\d\xi}}{H_\C},
\]
the range of $\left(\cL^\#\right)^*$ is orthogonal to $\tfrac{\d \hat X}{\d\xi}$, so that the dimension of the kernel of $\left(\cL^\#\right)^*$ is at least $1$. Interchanging the roles of $\cL^\#$ and $\left(\cL^\#\right)^*$ shows that the range of $\cL^\#$ is orthogonal to the kernel of $\left(\cL^\#\right)^*$. Since the range of $\cL^\#$ has codimension $1$ we conclude that the dimension of the kernel of $\left(\cL^\#\right)^*$ is $1$.

\medskip

For such $Y$ we additionally have
\begin{align*}
\left(\Pi^{\#,0}\right)^*Y &\stackrel{\eqref{riesz_projection}}{=} \left(-\frac{1}{2 \pi i} \ointctrclockwise_{\verti{\lambda} = r} \left(\overline{\lambda} \id_{H_\C} - \left(\cL^\#\right)^*\right)^{-1} \d\overline{\lambda}\right) Y \\
&= \left(\frac{1}{2 \pi i} \ointctrclockwise_{\verti{\lambda} = r} \left(\lambda \id_{H_\C} - \left(\cL^\#\right)^*\right)^{-1} \d\lambda\right) Y = Y.
\end{align*}
With $\left(\left(\Pi^{\#,0}\right)^*Y\right)_1 = 0$ this implies immediately that $w=0$. In order to determine the second component, we compute $\left(\cL^\#\right)^*Y$ explicitly. Let $Y_j = (w_j,q_j)^\trans \in C^\infty_\mathrm{c}(\R;\C)^2$ for $j=1,2$, then by \eqref{fw_op} we obtain through integration by parts and regrouping terms
\begin{eqnarray*}
\inner{\cL^\# Y_1}{Y_2}{H_\C} &=& \eps Z \nu \int_\R (\partial_\xi^2 \overline{w_1}) w_2 \d \xi + \eps Z  \int_\R f'(\hat u)\overline{w_1}w_2 \d \xi - \eps Z  \int_\R \overline{q_1} w_2 \d \xi \\
&& - \eps Z s \int_\R (\partial_\xi \overline{w_1})w_2 \d \xi + \eps Z  \int_\R \overline{w_1}q_2 \d \xi - \eps Z \gamma  \int_\R \overline{q_1} q_2 \d \xi - Z s \int_\R (\partial_\xi \overline{q_1}) q_2 \d \xi \\
&=& \inner{Y_1}{\begin{pmatrix}\nu \partial_\xi^2 + f'\left(\hat u\right) + s \partial_\xi & 1 \\ -\eps & - \eps \gamma + s \partial_\xi\end{pmatrix}Y_2}{H_\C} \\
&=& \inner{Y_1}{\left(\cL^\#\right)^*Y_2}{H_\C}.
\end{eqnarray*}
Now $\left(\left(\Pi^{\#,0}\right)^*Y\right)_1 =0$ is equivalent to
\[
\nu \partial_\xi^2 w + f'(\hat u) w + s \partial_\xi w +\eps q = 0.
\]
With $w=0$ this implies $q =0,$ a contradiction to $Y$ being non-trivial.
\end{proof}
%

%%%%%%%%%%%%%%%%%%%%%%%%%%%%%%%%%%%%%%%%%%%

%%%%%%%%%%%%%%%%%%%%%%%%%%%%%%%%%%%%%%%%%%%
\section{Conclusions \& Outlook\label{sec:conclusions}}
%%%%%%%%%%%%%%%%%%%%%%%%%%%%%%%%%%%%%%%%%%%

In this work, we have shown how to derive a multiscale decomposition near a deterministically stable traveling wave for the FitzHugh-Nagumo SPDE-ODE system. This decomposition into a component along the translation invariant family and into its complement exploits the small noise and small time scale separation parameters to derive leading-order dynamics. More precisely, the stochastically adjusted wave speed is given by an SODE to account for stochastic phase dynamics along the deterministically neutral translation mode. The fluctuations in the infinitely many remaining modes are captured by an SPDE system. Locally, near the wave, these two equations can be linearized to provide a relatively explicit solution representation which approximates well the deviations in the corresponding modes. In particular, our approach does not require an analytic semigroup generated by the linearization and it applies to much wider classes of SPDE systems as well as other patterns.

\medskip

Natural generalizations of our work could be to treat the case of a cylindrical Wiener process allowing for translation-invariant (white) noise or to allow for multiplicative noise. Furthermore, it would be a desirable goal to obtain optimal estimates on the relevant stopping times, where the approach breaks down which has been addressed for the FitzHugh-Nagumo system with a regularizing Laplacian in the second component in \cite{HamsterHupkes2020PhysD} and for stochastic neural field equations in \cite{MacLaurinBressloff2020}. These could then be compared with high-accuracy numerical simulations; cf.~the references in numerics in the introduction. Of course, other classes of traveling wave patterns and effects could also be tackled via multiscale decomposition, e.g., periodic wave trains, deterministically chaotic waves, stochastic pulse splitting effects induced by large deviations or propagation failure, just to name a few. Many of these effects are somewhat understood for special examples of scalar SPDEs (see~\cite{KuehnSPDEwaves} and the references therein) but it is clear that the dynamics can be far more complicated for SPDE systems. From a technical viewpoint, we have already pointed out that applying stochastic slow manifold methods and related fast-slow sample path estimates would be natural directions for future research. Furthermore, a generalization from the Hilbert-space setting to Banach spaces appears possible as the variational approach also works if $V$ and $V^*$ are Banach spaces (cf.~Assumptions~\ref{ass:liu_roeckner}), we do not require any orthogonality, and Riesz spectral projections are available in Banach spaces, too. Notably, related techniques in the deterministic setting, such as Lyapunov-Schmidt or center-manifold reductions work in Banach spaces as well. In summary, it seems evident that rigorous multiscale methods for pattern formation in SPDEs have already been very successful but still need additional development.     

%%%%%%%%%%%%%%%%%%%%%%%%%%%%%%%%%%%%%%%%%%%
\appendix
%%%%%%%%%%%%%%%%%%%%%%%%%%%%%%%%%%%%%%%%%%%
\section{Proofs of auxiliary results\label{sec:proofs_prel}}
%%%%%%%%%%%%%%%%%%%%%%%%%%%%%%%%%%%%%%%%%%%
\subsection{Existence and uniqueness of solutions using the variational approach\label{sec:var_construct}}
%%%%%%%%%%%%%%%%%%%%%%%%%%%%%%%%%%%%%%%%%%%
In this section, we prove Proposition~\ref{prop:ex_var} and Proposition~\ref{prop:reg_var}. Note that the formulation \eqref{stoch_evol} for general operators $\cA \colon [0,T] \times V \times \Omega \to V^*$ and $\cB \colon [0,T] \times V \times \Omega \to L_2(U;H)$ can be found in \cite{LiuRoeckner2010} and existence of variational solutions has been established under the following conditions:

\medskip

%%%%%%%%%%%%%%%%%%%%%%%%%%%%%%%%%%%%%%%%%%%
\begin{assumptions}\label{ass:liu_roeckner}
 The Hilbert space $H$ is separable, $V$ is a reflexive Banach space which is continuously and densly embedded into $H$, and $(W_U(t),t \ge 0)$ is a cylindrical Wiener process on a separable Hilbert space $U$ with respect to a complete filtered probability space
\[
\left(\Omega,\cF,\left(\cF_t\right)_{t \in [0,T]},\PP\right)
\]
with a complete and right-continuous filtration $\left(\cF_t\right)_{t \in [0,T]}$. Furthermore, there exist $\alpha > 1$, $\beta \ge 0$, $\theta > 0$, $C_1, C_2 < \infty$, a positive $\left(\cF_t\right)_{t \in [0,T]}$-adapted process $F \in L^{\frac p 2}\left([0,T] \times \Omega; \d t \times \PP\right)$ with $p \ge \beta+2$, and $g: \, V \to [0,\infty)$ measurable and locally bounded, such that the following conditions (LR\ref{item:hem})--(LR\ref{item:bound2}) hold for all $Y,Y_1,Y_2 \in V$ and $(t,\omega) \in [0,T] \times \Omega$:
 \begin{enumerate}[(LR1)]
  \item\label{item:hem} \emph{Hemicontinuity:} the map $\R \owns \tau \mapsto \dualpair{V^*}{\cA(t,Y_1 + \tau Y_2)}{Y}{V} \in \R$ is continuous.
  \item\label{item:mon} \emph{Local monotonicity:}
  \begin{align*}
  &2 \dualpair{V^*}{\cA(t,Y_1) - \cA(t,Y_2)}{Y_1-Y_2}{V} + \vertii{\cB(t,Y_1) - \cB(t,Y_2)}_{L_2(U;H)}^2 \\
  & \quad \le \left(C_1+g(Y_2)\right) \vertii{Y_1 - Y_2}_H^2.
  \end{align*}
  \item\label{item:coercive} \emph{Coercivity:} $2 \dualpair{V^*}{\cA(t,Y)}{Y}{V} + \vertii{\cB(t,Y)}_{L_2(U;H)}^2 + \theta \vertii{Y}_V^\alpha \le F(t) + C_1 \vertii{Y}_H^2$.
  \item\label{item:growth} \emph{Growth:} $\vertii{\cA(t,Y)}^{\frac{\alpha}{\alpha-1}}_{V^*} \le \left(F(t) + C_1 \vertii{Y}_V^\alpha\right) \left(1 + \vertii{Y}_H^\beta\right)$.
  \item\label{item:bound1} $\vertii{\cB(t,Y)}_{L_2(U;H)}^2 \le C_2 \left(F(t) + \vertii{Y}_H^2\right)$.
  \item\label{item:bound2} $g(Y) \le C_2 \left(1+\vertii{Y}_V^\alpha\right) \left(1+\vertii{Y}_H^\beta\right)$.
 \end{enumerate}
 %
 %%%%%%%%%%%%%%%%%%%%%%%%%%%%%%%%%%%%%%%%%%
\end{assumptions}
%%%%%%%%%%%%%%%%%%%%%%%%%%%%%%%%%%%%%%%%%%%
Note that conditions (LR\ref{item:hem}) and (LR\ref{item:coercive}) are the same as the classical ones from~\cite{KrylovRozovskii1979}, while conditions (LR\ref{item:mon}) and (LR\ref{item:growth}) are weaker. A main theorem presented in~\cite{LiuRoeckner2010} reads as follows:
%%%%%%%%%%%%%%%%%%%%%%%%%%%%%%%%%%%%%%%%%%%
\begin{theorem}[Liu~and~R\"ockner~\cite{LiuRoeckner2010}]\label{th:liu_roeckner}
Under Assumptions~\ref{ass:liu_roeckner}, for any initial datum
\[
X(0) = X^{(0)} \in L^p\left(\Omega,\cF_{0},\PP;H\right),
\]
equation~\eqref{stoch_evol} has a unique variational solution $\left(X(t,\cdot)\right)_{t \in [0,T]}$ (cf.~Definition~\ref{def:variational}), which additionally satisfies
 \[
  \E\left[\sup_{t \in [0,T]} \vertii{X(t,\cdot)}_H^p + \int_0^T \vertii{X(t,\cdot)}_V^\alpha \d t\right] < \infty.
 \]
\end{theorem}
%%%%%%%%%%%%%%%%%%%%%%%%%%%%%%%%%%%%%%%%%%%
We will use Theorem~\ref{th:liu_roeckner} to prove Proposition~\ref{prop:ex_var} by verifying Assumptions~\ref{ass:liu_roeckner}.

%%%%%%%%%%%%%%%%%%%%%%%%%%%%%%%%%%%%%%%%%%%
\begin{proof}[Proof of Proposition~\ref{prop:ex_var}]
The Hilbert spaces $U = L^2(\R)$ and $H \stackrel{\eqref{rigged_alt}}{=} L^2(\R) \eoperp L^2(\R)$ are obviously separable, $V \stackrel{\eqref{rigged_alt}}{=} H^1(\R) \eoperp L^2(\R)$ is reflexive, and as the test functions are dense in $H$ and $V$, also $V$ is dense in $H$. Next, we concentrate on verifying (LR\ref{item:hem})--(LR\ref{item:bound2}) with the choices 
\begin{subequations}\label{choices_var}
\begin{align}
\alpha &:= 2, \label{choice_alpha} \\
\beta &:= 4, \label{choice_beta} \\
\theta &:= \nu, \label{choice_theta} \\
C_1 &:= \max\left\{2 \eta_1, 1,\eta_1 + \nu, 4 \nu^2, \frac{64 \eta_4^2}{\eps^2 Z^2}\right\}, \label{choice_c1} \\
C_2 &:= 1\label{choice_c2}, \\
g &:= 0 \label{choice_g}, \\
F(t) &:= \max\left\{1 + \left(2 \sqrt\eps + \eps \gamma + \eta_4 \left(1+3 \vertii{\hat u}_{L^\infty(\R)}^2\right)\right)^4,\eps Z \sigma \vertii{\sqrt Q}_{L_2(U)}^2\right\}, \label{choice_f} \\
p &\ge 6. \label{choice_p}
\end{align}
\end{subequations}
%

 %%%%%%%%%%%%%%%%%%%%%%%%%%%%%%%%%%%%%%%%%%
 \proofstep{Proof of (LR\ref{item:hem})}
 %%%%%%%%%%%%%%%%%%%%%%%%%%%%%%%%%%%%%%%%%%
 We have for $Y = \left(w,q\right)^\trans, Y_1 = \left(w_1,q_1\right)^\trans, Y_2 = \left(w_2,q_2\right)^\trans \in V$,
 \begin{eqnarray*}
  \lefteqn{\dualpair{V^*}{\cA(t,Y_1 + \tau Y_2)}{Y}{V}} \\
  &\stackrel{\eqref{def_a}}{=}& \eps Z \dualpair{H^{-1}(\R)}{\nu \partial_x^2 \left(w_1 + \tau w_2\right)}{w}{H^1(\R)} \\
  && + \eps Z \dualpair{H^{-1}(\R)}{f\left(w_1 + \tau w_2 + \hat u(\cdot + s t)\right) - f\left(\hat u(\cdot + s t)\right)}{w}{H^1(\R)} \\
  && - \eps Z \dualpair{H^{-1}(\R)}{q_1 + \tau q_2}{w}{H^1(\R)} \\
  && + \eps Z \dualpair{L^2(\R)}{w_1 + \tau w_2 - \gamma \left(q_1 + \tau q_2\right)}{q}{L^2(\R)} \\
  &=& \tau \eps Z \left(- \nu \inner{\partial_x w_2}{\partial_x w}{L^2(\R)} - \inner{q_2}{w}{L^2(\R)} + \inner{w_2 - \gamma q_2}{q}{L^2(\R)}\right) \\
  && + \eps Z \inner{f\left(w_1 + \tau w_2 + \hat u(\cdot + s t)\right) - f\left(\hat u(\cdot + s t)\right)}{w}{L^2(\R)} \\
  && - \eps Z \nu \inner{\partial_x w_1}{\partial_x w}{L^2(\R)} - \eps Z \inner{q_1}{w}{L^2(\R)} + \eps Z \inner{w_1 - \gamma q_1}{q}{L^2(\R)}.
 \end{eqnarray*}
 Hence, hemicontinuity follows if
 \[
  h \colon \R \to \R, \quad \tau \mapsto h(\tau) := \eps Z \inner{f\left(w_1 + \tau w_2 + \hat u(\cdot + s t)\right) - f\left(\hat u(\cdot + s t)\right)}{w}{L^2(\R)}
 \]
 is continuous. This is true because for $\tau_1, \tau_2 \in \R$ we have
 \begin{eqnarray*}
  \verti{h(\tau_1) - h(\tau_2)} &=& \eps Z \verti{\inner{f\left(w_1 + \tau_1 w_2 + \hat u(\cdot + s t)\right) - f\left(w_1 + \tau_2 w_2 + \hat u(\cdot + s t)\right)}{w}{L^2(\R)}} \\
  &\stackrel{\eqref{cond_f_diff_4}}{\le}& \eps Z \eta_4 \left(1 + 6 \vertii{\hat u}_{L^\infty(\R)}^2 + 6 \vertii{w_1}_{L^\infty(\R)}^2 + 3 \left(\tau_1^2 + \tau_2^2\right) \vertii{w_2}_{L^\infty(\R)}^2\right) \\
  && \vertii{w_2}_{L^2(\R)} \verti{\tau_1 - \tau_2} \vertii{w}_{L^2(\R)} \\
  &\stackrel{\eqref{v_norm}}{\le}& \eta_4 \left(1 + 6 \vertii{\hat u}_{L^\infty(\R)}^2 + \frac{6}{\eps Z} \vertii{Y_1}_V^2 + \frac{3 (\tau_1^2 + \tau_2^2)}{\eps Z} \vertii{Y_2}_V^2\right) \\
  &&\vertii{Y_2}_H \vertii{Y}_H \verti{\tau_1 - \tau_2},
 \end{eqnarray*}
 where the Sobolev embedding theorem in form of $\vertii{w_j}_{L^\infty(\R)} \le \vertii{w_j}_{H^1(\R)}$ has been applied.
 
 %%%%%%%%%%%%%%%%%%%%%%%%%%%%%%%%%%%%%%%%%%
 \proofstep{Proof of (LR\ref{item:mon})}
 %%%%%%%%%%%%%%%%%%%%%%%%%%%%%%%%%%%%%%%%%%
 For proving local monotonicity, observe that for $Y_1 = (w_1,q_1)^\trans, Y_2 = (w_2,q_2)^\trans \in V$,
 \begin{eqnarray*}
  \lefteqn{2 \dualpair{V^*}{\cA(t,Y_1) - \cA(t,Y_2)}{Y_1 - Y_2}{V}} \\
  &\stackrel{\eqref{def_a}}{=}& - 2 \eps Z \nu  \vertii{\partial_x \left(w_1 - w_2\right)}_{L^2(\R)}^2 \\
  && + 2 \eps Z \inner{f\left(w_1 + \hat u\left(\cdot + s t\right)\right) - f\left(w_2 + \hat u\left(\cdot + s t\right)\right)}{w_1 - w_2}{L^2(\R)} \\
  && - 2 \eps Z \gamma \vertii{q_1 - q_2}_{L^2(\R)}^2.
 \end{eqnarray*}
 On noting that
 \begin{eqnarray*}
 2 \eps Z \inner{f\left(w_1 + \hat u\left(\cdot + s t\right)\right) - f\left(w_2 + \hat u\left(\cdot + s t\right)\right)}{w_1 - w_2}{L^2(\R)} &\stackrel{\eqref{cond_f_prime}}{\le}& 2 \eps Z \eta_1 \vertii{w_1 - w_2}_{L^2(\R)}^2 \\
  &\stackrel{\eqref{v_norm}}{\le}& 2 \eta_1 \vertii{Y_1 - Y_2}_H^2,
 \end{eqnarray*}
 and $\cB\left(t,Y_1\right) - \cB\left(t,Y_2\right) \stackrel{\eqref{def_b}}{=} 0$, this results in
 \begin{eqnarray*}
  2 \dualpair{V^*}{\cA(t,Y_1) - \cA(t,Y_2)}{Y_1 - Y_2}{V} + \vertii{\cB\left(t,Y_1\right) - \cB\left(t,Y_2\right)}_{L_2(U;H)}^2 &\le& 2 \eta_1 \vertii{Y_1 - Y_2}_H^2 \\
  &\stackrel{\eqref{choice_c1}}{\le}& C_1 \vertii{Y_1 - Y_2}_H^2.
 \end{eqnarray*}
Note that we have verified the classical monotonicity assumption in~\cite{KrylovRozovskii1979}, i.e., local monotonicity is not needed here.
 
%%%%%%%%%%%%%%%%%%%%%%%%%%%%%%%%%%%%%%%%%%
 \proofstep{Proof of (LR\ref{item:coercive})}
 %%%%%%%%%%%%%%%%%%%%%%%%%%%%%%%%%%%%%%%%%%
 For proving coercivity, we note that for $Y = (w,q)^\trans \in V$,
 \begin{eqnarray*}
  \dualpair{V^*}{\cA(t,Y)}{Y}{V} &\stackrel{\eqref{def_a}}{=}& - \eps Z \nu \vertii{\partial_x w}_{L^2(\R)}^2 + \eps Z \inner{f\left(w + \hat u\left(\cdot + s t\right)\right) - f\left(\hat u\left(\cdot + s t\right)\right)}{w}{L^2(\R)} \\
  && - \eps Z \gamma \vertii{q}_{L^2(\R)}^2.
 \end{eqnarray*}
 Using
 \[
 \eps Z \inner{f\left(w + \hat u\left(\cdot + s t\right)\right) - f\left(\hat u\left(\cdot + s t\right)\right)}{w}{L^2(\R)} \stackrel{\eqref{cond_f_prime}}{\le} \eps Z \eta_1 \vertii{w}_{L^2(\R)}^2 \stackrel{\eqref{h_norm}}{\le} \eta_1 \vertii{Y}_H^2
 \]
 and $\vertii{\cB(t,Y)}_{L_2(U;H)}^2 \stackrel{\eqref{h_norm}, \eqref{def_b}}{=} \eps Z \sigma \vertii{\sqrt Q}_{L_2(U)}^2$, we obtain the estimate
 \begin{eqnarray*}
  \dualpair{V^*}{\cA(t,Y)}{Y}{V} + \vertii{\cB(t,Y)}_{L_2(U;H)}^2 &\le& \eps Z \sigma \vertii{\sqrt Q}_{L_2(U)}^2 - \eps Z \nu \vertii{\partial_x w}_{L^2(\R)}^2 + \eta_1 \vertii{Y}_H^2 \\
  &\stackrel{\eqref{hv_norm}}{=}& \eps Z \sigma \vertii{\sqrt Q}_{L_2(U)}^2 + \left(\eta_1 + \nu\right) \vertii{Y}_H^2 - \nu \vertii{Y}_V^2 \\
  &\stackrel{\eqref{choices_var}}{\le}& F(t) + C_1 \vertii{Y}_H^2 - \theta \vertii{Y}_V^2,
 \end{eqnarray*}
which yields (LR\ref{item:coercive}).
  
 %%%%%%%%%%%%%%%%%%%%%%%%%%%%%%%%%%%%%%%%%%
 \proofstep{Proof of (LR\ref{item:growth})}
 %%%%%%%%%%%%%%%%%%%%%%%%%%%%%%%%%%%%%%%%%%
 We have for $Y = (w,q)^\trans, \tilde Y = \left(\tilde w, \tilde q\right)^\trans \in V$,
 \begin{eqnarray*}
  \lefteqn{\dualpair{V^*}{\cA(t,Y)}{\tilde Y}{V}} \\
  &\stackrel{\eqref{def_a}}{=}& - \eps Z \nu \inner{\partial_x w}{\partial_x \tilde w}{L^2(\R)} + \eps Z \inner{f\left(w + \hat u(\cdot + s t)\right) - f\left(\hat u(\cdot + s t)\right)}{\tilde w}{L^2(\R)} - \eps Z \inner{q}{\tilde w}{L^2(\R)} \\
  && + \eps Z \inner{w}{\tilde q}{L^2(\R)} - \eps Z \gamma \inner{q}{\tilde q}{L^2(\R)} \\
  &\stackrel{\eqref{cond_f_diff_4}}{\le}& \eps Z \nu \vertii{\partial_x w}_{L^2(\R)} \vertii{\partial_x \tilde w}_{L^2(\R)} + \eps Z \eta_4 \left(1 + 3 \vertii{\hat u}_{L^\infty(\R)}^2 +  2 \vertii{w}_{L^\infty(\R)}^2\right) \vertii{w}_{L^2(\R)} \vertii{\tilde w}_{L^2(\R)} \\
  && + \eps Z \vertii{q}_{L^2(\R)} \vertii{\tilde w}_{L^2(\R)} + \eps Z \vertii{w}_{L^2(\R)} \vertii{\tilde q}_{L^2(\R)} + \eps Z \gamma \vertii{q}_{L^2(\R)} \vertii{\tilde q}_{L^2(\R)} \\
  &\stackrel{\eqref{hv_norm}}{\le}& \left(\vertii{Y}_V \left(\nu + \frac{4 \eta_4}{\eps Z} \vertii{Y}_H^2\right) + \vertii{Y}_H \left(2 \sqrt\eps + \eps \gamma + \eta_4 \left(1+3 \vertii{\hat u}_{L^\infty(\R)}^2\right)\right)\right) \vertii{\tilde Y}_V,
 \end{eqnarray*}
 where we have used that
 \[
 \vertii{w}_{L^\infty(\R)}^2 \le 2 \vertii{w}_{L^2(\R)} \vertii{\tfrac{\d w}{\d x}}_{L^2(\R)} \stackrel{\eqref{hv_norm}}{\le} \frac{2}{\eps Z} \vertii{Y}_H \vertii{Y}_V.
 \]
 This implies
 \begin{eqnarray*}
 \lefteqn{\vertii{\cA(t,Y)}_{V^*}^{\frac{\alpha}{\alpha-1}}} \\
 &\stackrel{\eqref{choice_alpha}}{=}& \vertii{\cA(t,Y)}_{V^*}^2 \\
 &\le& 2 \vertii{Y}_V^2 \left(\nu + \frac{4 \eta_4}{\eps Z} \vertii{Y}_H^2\right)^2 + 2 \left(2 \sqrt\eps + \eps \gamma + \eta_4 \left(1+3 \vertii{\hat u}_{L^\infty(\R)}^2\right)\right)^2 \vertii{Y}_H^2 \\
 &\le& \left(1 + \left(2 \sqrt\eps + \eps \gamma + \eta_4 \left(1+3 \vertii{\hat u}_{L^\infty(\R)}^2\right)\right)^4 + 4 \max\left\{\nu^2, \frac{16 \eta_4^2}{\eps^2 Z^2}\right\} \vertii{Y}_V^2\right)\left(1+\vertii{Y}_H^4\right) \\
 &\stackrel{\eqref{choices_var}}{\le}&  \left(F(t) + C_1 \vertii{Y}_V^\alpha\right) \left(1 + \vertii{Y}_H^\beta\right),
 \end{eqnarray*}
 i.e., (LR\ref{item:growth}). Note that here we have not obtained the classical growth condition in~\cite{KrylovRozovskii1979}.
 
 %%%%%%%%%%%%%%%%%%%%%%%%%%%%%%%%%%%%%%%%%%
 \proofstep{Proof of (LR\ref{item:bound1})}
 %%%%%%%%%%%%%%%%%%%%%%%%%%%%%%%%%%%%%%%%%%
 We have $\vertii{\cB(t,Y)}_{L_2(U;H)}^2 \stackrel{\eqref{h_norm}, \eqref{def_b}}{=} \eps Z \sigma \vertii{\sqrt Q}_{L_2(U)}^2 \le C_2 \left(F(t) + \vertii{Y}_H^2\right)$.
 
 %%%%%%%%%%%%%%%%%%%%%%%%%%%%%%%%%%%%%%%%%%
 \proofstep{Proof of (LR\ref{item:bound2})}
 %%%%%%%%%%%%%%%%%%%%%%%%%%%%%%%%%%%%%%%%%%
 This trivially holds because $g = 0$.
\end{proof}
%%%%%%%%%%%%%%%%%%%%%%%%%%%%%%%%%%%%%%%%%%%

%%%%%%%%%%%%%%%%%%%%%%%%%%%%%%%%%%%%%%%%%%%
\begin{proof}[Proof of Proposition~\ref{prop:reg_var}]
Denote by $X = (u,v)^\trans$ the solution from Proposition~\ref{prop:ex_var}. We first apply \cite[Proposition~G.0.5~(i)]{LiuRoeckner2015} and verify the conditions there to conclude that $X$ meets a mild-solution representation. Notably, $\cB(t,X(t,\cdot))$ (cf.~\eqref{def_b}) does for $t \in [0,T]$ neither depend on $t$ nor on $X$ and takes values in $L_2(U;H)$. Hence, $\cB(t,X(t,\cdot))$ is deterministic and in particular
\[
\PP\left[\int_0^T \vertii{\cB(t,X(t,\cdot))}_{L_2(U;H)} \d t < \infty\right] = 1
\]
trivially holds true. Furthermore, by the regularity of the variational solution stated in Definition~\ref{def:variational}, we have
\[
\E\left[\int_0^T \vertii{X(t,\cdot)}_V \d t\right] \le \sqrt T \left(\E\left[\int_0^T \vertii{X(t,\cdot)}_V^2 \d t\right]\right)^{\frac 1 2} < \infty,
\]
which is why obviously also
\[
\PP\left[\int_0^T \vertii{X(t,\cdot)}_H \d t < \infty\right] = 1
\]
is valid. In view of the above estimate, the mixing of $u$ and $v$ in the two components of $\cA(t,X(t))$ (cf.~\eqref{def_a}) is immaterial, so that it remains show that
\begin{equation}\label{equivalence_var_mild}
\PP\left[\int_0^T \vertii{f\left(u(t,\cdot)+\hat u(\cdot+st)\right) - f\left(\hat u(\cdot+st)\right)}_{L^2(\R)} \d t < \infty\right] = 1
\end{equation}
holds true. Indeed, we can estimate
\begin{eqnarray*}
\lefteqn{\E\left[ \int_0^T \vertii{f\left(u(t,\cdot)+\hat u(\cdot+st)\right) - f\left(\hat u(\cdot+st)\right)}_{L^2(\R)} \d t\right]^{\frac 1 2}} \\
&\stackrel{\eqref{cond_f_diff_4}}{\le}& \eta_4 \E \left[\sup_{t \in [0,T]} \vertii{u(t,\cdot)}_{L^2(\R)} \int_0^T \left(1+2\vertii{\hat u}_{L^\infty(\R)}^2+3\vertii{u(t,\cdot)}^2_{H^1(\R)}\right) \d t\right]^{\frac 1 2} \\
&\le& \frac{\eta_4}{2} \E\left[\sup_{t \in [0,T]} \vertii{u(t,\cdot)}_{L^2(\R)}\right] + \frac{\eta_4 T}{2} \left(1+2\vertii{\hat u}_{L^\infty(\R)}^2\right)  + \frac{3 \eta_4}{2} \E\left[\int_0^T \vertii{u(t,\cdot)}^2_{H^1(\R)} \d t\right]
\end{eqnarray*}
where the Sobolev embedding in form of $\vertii{u(t,\cdot)}_{L^\infty(\R)} \le \vertii{u(t,\cdot)}_{H^1(\R)}$ has been used and finiteness of the terms in the last line follows from the regularity of the variational solution stated in Definition~\ref{def:variational} and Proposition~\ref{prop:ex_var}. Hence, also \eqref{equivalence_var_mild} is satisfied.

\medskip

Next, we prove additional regularity provided $\sqrt Q \in L_2(U;H^1(\R))$, $u^{(0)} \in L^2(\R)$, and $v^{(0)} \in H^1(\R)$. From the first component of \eqref{stoch_evol} or \eqref{var_formula} we derive the mild-solution formula
\begin{eqnarray*}
t u(t,\cdot) &=& \int_0^t K_{t - t'} * \left(f\left(u(t',\cdot)+\hat u(\cdot+s t')\right) - f\left(\hat u(\cdot+st')\right) - v(t',\cdot)\right) t' \d t' \\
&&  + \int_0^t K_{t - t'} * u(t',\cdot) \, \d t' + \int_0^t K_{t-t'} * t' \d W(t',\cdot), \quad \mbox{$\PP$-almost surely},
\end{eqnarray*}
where $K_t(x) = \frac{e^{- \frac{x^2}{4 \nu t}}}{\sqrt{4 \pi \nu t}}$ denotes the heat kernel generated by $\nu \partial_x^2$ and $*$ the convolution on the real line. For the second component of \eqref{stoch_evol} or \eqref{var_formula} we get analogously
\[
v(t,\cdot) = e^{- \eps \gamma t} v^{(0)} + \eps \int_0^t e^{- \eps \gamma (t-t')} u(t',\cdot) \, \d t', \quad \mbox{$\PP$-almost surely}.
\]
Differentiation in space yields
\begin{eqnarray}
t (\partial_x u)(t,\cdot) &=& \int_0^t K_{t-t'} * \left( f'(u(t',\cdot) + \hat u(\cdot + s t')) \left(\partial_x u(t',\cdot) + \tfrac{\d \hat u}{\d\xi}(\cdot+st')\right)\right) t' \d t' \nonumber \\
&& - \int_0^t K_{t-t'} * \left(f'\left(\hat u(\cdot+st')\right) \tfrac{\d \hat u}{\d\xi}(\cdot+st') + \partial_x v(t',\cdot) \right) t' \d t' \nonumber \\
&& + \int_0^t K_{t - t'} * \partial_x u(t',\cdot) \, \d t' + \int_0^t K_{t-t'} * \left(t' \d (\partial_x W)(t',\cdot)\right), \label{reg_pr_x_u}
\end{eqnarray}
$\PP$-almost surely, and
\begin{equation}\label{rep_prx_v_u}
\partial_x v(t,\cdot) = e^{- \eps \gamma t} \left(\partial_x v^{(0)}\right) + \eps \int_0^t e^{- \eps \gamma (t-t')} (\partial_x u)(t',\cdot) \, \d t', \quad \mbox{$\PP$-almost surely.}
\end{equation}
We estimate the three lines on the right-hand side of \eqref{reg_pr_x_u} separately:

\medskip

Using $\vertii{K_{t-t'}}_{L^1(\R)} = 1$, we obtain for the first term on the right-hand side of \eqref{reg_pr_x_u} with Young's convolution inequality
\begin{eqnarray*}
\lefteqn{\vertii{\int_0^t K_{t-t'} * \left( f'(u(t',\cdot) + \hat u(\cdot + s t')) \left(\partial_x u(t',\cdot) + \tfrac{\d \hat u}{\d\xi}(\cdot+st')\right)\right) t' \d t'}_{L^2(\R)}} \\
&\stackrel{\eqref{cond_f_diff_3}}{\le}& \eta_3 \int_0^t \left(1 + 2 \vertii{\hat u}_{L^\infty(\R)}^2 + 2 \vertii{u(t',\cdot)}_{H^1(\R)}^2\right) \left(\vertii{\partial_x u(t',\cdot)}_{L^2(\R)} + \vertii{\tfrac{\d \hat u}{\d\xi}}_{L^2(\R)}\right) t' \d t' \\
&\le& \frac{\eta_3 T}{2} \left(1 + 2 \vertii{\hat u}_{L^\infty(\R)}^2\right) \left(T \vertii{\tfrac{\d \hat u}{\d\xi}}_{L^2(\R)} + \frac{2 T^{\frac 1 2}}{\sqrt 3} \vertii{\partial_x u}_{L^2([0,T] \times \R)}\right) \\
&& + 2 \eta_3 T \vertii{\tfrac{\d \hat u}{\d\xi}}_{L^2(\R)} \vertii{u}_{L^2([0,T];H^1(\R))}^2 + \frac{2 \eta_3 T^{\frac 3 2}}{\sqrt 3} \vertii{u}_{C^0([0,T];L^2(\R))}^2 \vertii{\partial_x u}_{L^2([0,T] \times \R)} \\
&& + 2 \eta_3 \int_0^t \vertii{\partial_x u(t',\cdot)}_{L^2(\R)}^2 t' \vertii{\partial_x u(t',\cdot)}_{L^2(\R)} \d t',
\end{eqnarray*}
where the Sobolev embedding on the real line in form of $\vertii{u(t',\cdot)}_{L^\infty(\R)} \le \vertii{u(t',\cdot)}_{H^1(\R)}$ has been applied.

\medskip

For the second line of \eqref{reg_pr_x_u}, we obtain similarly
\begin{eqnarray*}
\lefteqn{\vertii{\int_0^t K_{t-t'} * \left(f'\left(\hat u(\cdot+st')\right) \tfrac{\d \hat u}{\d\xi}(\cdot+st') + \partial_x v(t',\cdot) \right) t' \d t'}_{L^2(\R)}} \\
&\stackrel{\eqref{cond_f_diff_3}}{\le}& \frac{\eta_3 T^2}{2} \left(1+\vertii{\tfrac{\d \hat u}{\d\xi}}_{L^\infty(\R)}^2\right) \vertii{\tfrac{\d \hat u}{\d\xi}}_{L^2(\R)} + \frac{T^{\frac 3 2}}{\sqrt 3} \vertii{\partial_x v}_{L^2([0,T] \times \R)},
\end{eqnarray*}
where we may use
\begin{equation}\label{est_prx_v_u}
\vertii{\partial_x v}_{L^2([0,T] \times \R)} \stackrel{\eqref{rep_prx_v_u}}{\le} \frac{1}{\sqrt{2\eps\gamma}} \vertii{\partial_x v^{(0)}}_{L^2(\R)} + \frac{1}{\gamma} \vertii{\partial_x u}_{L^2([0,T] \times \R)}, \quad \mbox{$\PP$-almost surely.}
\end{equation}
\medskip

Finally, the first integral on third line of \eqref{reg_pr_x_u} yields
\[
\vertii{\int_0^t K_{t - t'} * \partial_x u(t',\cdot) \, \d t'}_{L^2(\R)} \le \sqrt T \vertii{\partial_x u}_{L^2([0,T] \times \R)},
\]
while the second integral gives with help of It\^o's isometry
\begin{align*}
\E\left[\vertii{\int_0^t K_{t-t'} * \left(t' \d (\partial_x W)(t',\cdot)\right)}_{L^2(\R)}^2\right] &\le \int_0^t (t')^2 \vertii{\left(K_{t-t'} *\right) \partial_x \sqrt Q}_{L_2(U)}^2 \d t' \\
&\le \frac{T^3}{3} \vertii{\sqrt Q}_{L_2(U;H^1(\R))}^2. 
\end{align*}
\medskip

Using $X = (u,v)^\trans \in L^2\left([0,T];V\right) \cap C^0\left([0,T];H\right)$, $\PP$-almost surely, and therefore
\[
u \in L^2\left([0,T];H^1(\R)\right) \cap C^0\left([0,T];L^2(\R)\right), \quad \mbox{$\PP$-almost surely},
\]
we conclude that there exists a constant $C < \infty$, $\PP$-almost surely, with $C \to 0$ as $T \searrow 0$, $\PP$-almost surely, such that
\[
\vertii{t (\partial_x u)(t,\cdot)}_{L^2(\R)} \le C + 2 \eta_3 \int_0^t \vertii{\partial_x u(t',\cdot)}_{L^2(\R)}^2 t' \vertii{\partial_x u(t',\cdot)}_{L^2(\R)} \d t, \quad \mbox{$\PP$-almost surely}.
\]
Gr\"onwall's inequality implies
\[
\vertii{t (\partial_x u)(t,\cdot)}_{L^2(\R)} \le C \exp\left(2 \eta_3 \int_0^t \vertii{\partial_x u(t',\cdot)}_{L^2(\R)}^2 \d t'\right) < \infty, \quad \mbox{$\PP$-almost surely,}
\]
whence $t (\partial_x u) \in L^\infty\left([0,T];L^2(\R)\right)$, $\PP$-almost surely. Furthermore, since $C \to 0$ as $T \searrow 0$, $\PP$-almost surely, $[0,T] \owns t \mapsto t \partial_x u(t,\cdot) \in L^2(\R)$ is, $\PP$-almost surely, continuous in $t = 0$.

\medskip

For proving continuity, observe that for $T \ge t_2 \ge t_1 > 0$ we have
\[
t_2 \partial_x u(t_2,\cdot) - t_1 \partial_x u(t_1,\cdot) = \left(t_2-t_1\right) \partial_x u(t_2,\cdot) + t_1 \left(\partial_x u(t_2,\cdot) - \partial_x u(t_1,\cdot)\right).
\]
Then, it follows $\left(t_2-t_1\right) \partial_x u(t_2,\cdot) \to 0$ as $t_2 \to t_1$ in $L^2(\R)$ by the reasoning before if one translates in time and uses uniqueness. For the remaining term, we derive once more from the first component of \eqref{stoch_evol} or \eqref{var_formula}
\begin{eqnarray}
\partial_x u(t_2,\cdot) - \partial_x u(t_1,\cdot) &=& K_{t_2-t_1} * (\partial_x u)(t_1,\cdot) - \partial_x u(t_1,\cdot) \nonumber \\
&& + \int_{t_1}^{t_2} K_{t_2-t'} * \left(f'\left(u(t',\cdot)+\hat u(\cdot+s t')\right) \left(\partial_x u(t',\cdot) + \tfrac{\d \hat u}{\d\xi}(\cdot+s t')\right)\right) \d t' \nonumber \\
&& - \int_{t_1}^{t_2} K_{t_2-t'} * \left(f'\left(\hat u(\cdot+st')\right) \tfrac{\d\hat u}{\d\xi}(\cdot+st') - \partial_x v(t',\cdot)\right) \d t' \nonumber \\
&& + \int_{t_1}^{t_2} K_{t_2-t'} * \d (\partial_x W)(t',\cdot), \quad \mbox{$\PP$-almost surely}. \label{cont_pr_x_u}
\end{eqnarray}
For the first line of \eqref{cont_pr_x_u} observe that
\[
\vertii{K_{t_2-t_1} * (\partial_x u)(t_1,\cdot) - \partial_x u(t_1,\cdot)}_{L^2(\R)} \to 0 \quad \mbox{as} \quad t_2 \to t_1, \quad \mbox{$\PP$-almost surely},
\]
follows from $t u \in L^\infty([0,T];L^2(\R))$, $\PP$-almost surely, and the fact that $(K_t*)_{t \ge 0} \subset L\left(L^2(\R)\right)$ is strongly continuous in $t = 0$. The other lines of \eqref{cont_pr_x_u} can be estimated as before. Altogether, we obtain
\[
\vertii{t_2 \partial_x u(t_2,\cdot) - t_1 \partial_x u(t_1,\cdot)}_{L^2(\R)} \to 0 \quad \mbox{as} \quad t_2 \to t_1, \quad \mbox{$\PP$-almost surely},
\]
which implies $t u \in C^0\left([0,T];H^1(\R)\right)$, $\PP$-almost surely. Finally, by \eqref{rep_prx_v_u} and \eqref{est_prx_v_u} it follows that $v \in C^0\left([0,T];H^1(\R)\right)$, $\PP$-almost surely.

\medskip

If $u^{(0)} \in H^1(\R)$, an analogous reasoning without time weight yields that $u \in C^0\left([0,T];H^1(\R)\right)$, $\PP$-almost surely, too.
\end{proof}
%%%%%%%%%%%%%%%%%%%%%%%%%%%%%%%%%%%%%%%%%%%

%%%%%%%%%%%%%%%%%%%%%%%%%%%%%%%%%%%%%%%%%%%
\subsection{Linearization around the traveling wave\label{sec:lin_proof}}
%%%%%%%%%%%%%%%%%%%%%%%%%%%%%%%%%%%%%%%%%%%
In this section, we give the proofs of Proposition~\ref{prop:family} and Proposition~\ref{prop:frozen}.
%%%%%%%%%%%%%%%%%%%%%%%%%%%%%%%%%%%%%%%%%%%
\subsubsection{Fixed frame\label{sec:family}}
%%%%%%%%%%%%%%%%%%%%%%%%%%%%%%%%%%%%%%%%%%%
We focus on investigating the linearized evolution generated by the family of operators $\left(\cL_{st}\right)_{t \ge 0}$ defined in \eqref{lin_op}.

%%%%%%%%%%%%%%%%%%%%%%%%%%%%%%%%%%%%%%%%%%%
\begin{proof}[Proof of Proposition~\ref{prop:family}]
Note that in view of \eqref{lin_op} we may write
\begin{equation}\label{decomp_lin_op}
\cL_{st} = \cL_{\pm\infty} + \cR_{st} \quad \mbox{with} \quad \left\{\begin{aligned} \cL_{\pm\infty} &:= \begin{pmatrix} \nu \partial_x^2 + f'(0) && - 1 \\ \eps && - \eps \gamma \end{pmatrix}, \\ \cR_{st} &:= \begin{pmatrix} f'\left(\hat u(\cdot+st)\right) - f'(0) & 0 \\ 0 & 0 \end{pmatrix}. \end{aligned}\right.
\end{equation}
For $Y = (w,q)^\trans \in \left(C_\mathrm{c}^\infty(\R)\right)^2$ we have
\begin{eqnarray*}
\inner{\cL_{\pm\infty} Y}{Y}{\cV} &=& \eps Z \inner{\nu \tfrac{\d^2 w}{\d x^2} + f'(0) w - q}{w}{H^1(\R)} + \eps Z \inner{w - \gamma q}{q}{H^1(\R)} \\
&=& \eps Z \left(- \nu \vertii{\tfrac{\d^2 w}{\d x^2}}_{L^2(\R)}^2 - \left(\nu - f'(0)\right) \vertii{\tfrac{\d w}{\d x}}_{L^2(\R)}^2 + f'(0) \vertii{w}_{L^2(\R)}^2\right) \\
&& - \eps Z \gamma \vertii{q}_{H^1(\R)}^2 \\
&\le& - \min\left\{- f'(0),\eps \gamma\right\} \vertii{Y}_{\cV}^2 - \eps Z \nu \vertii{\tfrac{\d^2 w}{\d x^2}}_{L^2(\R)}^2.
\end{eqnarray*}
By density of $\left(C_\mathrm{c}^\infty(\R)\right)^2$ in $D(\cL_{\pm\infty}) = H^3(\R) \eoperp H^1(\R)$, we infer that $\cL_{\pm\infty} + \kappa \id_{\cV}$ with
\[
\kappa := \min\left\{- f'(0),\eps \gamma\right\} \stackrel{\eqref{sign_f}}{\ge} 0
\]
is dissipative.

\medskip

In order to prove that $\cL_{\pm\infty} + (\kappa-1) \id_{\cV} \colon D\left(\cL_{\pm\infty}\right) \to \cV$ is a bijection, we define
\begin{eqnarray*}
\cM\left(Y_1,Y_2\right) &:=& - \eps Z \nu \inner{\tfrac{\d w_1}{\d x}}{\tfrac{\d w_2}{\d x}}{H^1(\R)} + \eps Z \inner{\left(f'(0) + \kappa - 1\right) w_1 - q_1}{w_2}{H^1(\R)} \\
&& + \eps Z \inner{w_1 + (\kappa - 1 - \gamma) q_1}{q_2}{H^1(\R)} \quad \mbox{for} \quad Y_j := \left(w_j,w_j\right)^\trans \in H^2(\R) \eoperp H^1(\R)
\end{eqnarray*}
and we recognize that $\cM \colon \left(H^2(\R) \eoperp H^1(\R)\right)^2 \to \cV$ is bilinear, continuous, and $-\cM$ is coercive, so that by the Lax-Milgram theorem for any $Y_2 \in \cV$ there exists $Y_1 \in H^2(\R) \eoperp H^1(\R)$ such that $\cM\left(Y_1,\tilde Y\right) = \inner{Y_2}{\tilde Y}{\cV}$ for all $\tilde Y \in H^2(\R) \eoperp H^1(\R)$. This implies $\cL_{\pm\infty} Y_1 + (\kappa-1) Y_1 = Y_2$ distributionally and hence in particular
\[
\nu \frac{\d^2 w_1}{\d x^2} = w_2 + \left(1-\kappa-f'(0)\right) w_1 + q_1 \in H^1(\R),
\]
giving $w_1 \in H^3(\R)$ and thus $Y_1 \in D\left(\cL_{\pm\infty}\right)$. The arguments have shown that $\cL_{\pm\infty} + (\kappa-1) \id_\cV \colon D\left(\cL_{\pm\infty}\right) \to \cV$ is a bijection.

\medskip

The Lumer-Philips theorem \cite[Chapter~1, Theorem~4.3]{Pazy1992} yields that $\cL_{\pm\infty} + \kappa \id_{\cV}$ generates a $C_0$-semigroup of contractions in $\cV$. Hence, $\cL_{\pm\infty}$ generates a $C_0$-semigroup $\left(e^{t \cL_{\pm\infty}}\right)_{t \ge 0}$ in $\cV$ with bound
\[
\vertii{e^{t \cL_{\pm\infty}}}_{L\left(\cV\right)} \le e^{-\kappa t}.
\]
Now, since $\hat u$ and $\tfrac{\d\hat u}{\d\xi}$ are bounded and $f \in C^1(\R)$, the family $\left(\cR_{st}\right)_{t \ge 0}$ is uniformly bounded $\cV \to \cV$ with $\vertii{\cR_{st}}_{L(\cV)} \le \vertii{f'(\hat u) - f'(0)}_{W^{1,\infty}(\R)} < \infty$. By \cite[Chapter~5, Theorem~2.3]{Pazy1992} the family $\left(\cL_{st}\right)_{t \ge 0}$ of linear operators generates an evolution family $\left(P_{st,st'}\right)_{t \ge t' \ge 0}$ of bounded linear operators $\cV \to \cV$ meeting estimate~\eqref{bound_evol_fam}, i.e.,
\[
\vertii{P_{st,st'}}_{L(\cV)} \le e^{\beta (t-t')} \quad \mbox{with} \quad \beta := \vertii{f'(\hat u) - f'(0)}_{W^{1,\infty}(\R)} - \min\left\{- f'(0),\eps \gamma\right\}. \qedhere
\]
\end{proof}
%%%%%%%%%%%%%%%%%%%%%%%%%%%%%%%%%%%%%%%%%%%

%%%%%%%%%%%%%%%%%%%%%%%%%%%%%%%%%%%%%%%%%%%
\subsubsection{Moving frame\label{sec:frozen}}
%%%%%%%%%%%%%%%%%%%%%%%%%%%%%%%%%%%%%%%%%%%
In this section, we prove spectral properties of the frozen-wave operator $\cL^\#$ (cf.~\eqref{fw_op}), as stated in Proposition~\ref{prop:frozen}.

\medskip

We view $\cL^\#$ as a perturbation of the limiting operator $\cL^\#_{\pm\infty}$ (cf.~\eqref{fw_op})
\begin{subequations}\label{decomp_frozen_op}
\begin{equation}
\cL^\# = \cL^\#_{\pm\infty} + \cR^\#,
\end{equation}
where
\begin{equation}\label{decomp_frozen_r}
\cL^\#_{\pm\infty} := \begin{pmatrix} \nu \partial_\xi^2 + f'(0) - s \partial_\xi && - 1 \\ \eps && - \eps \gamma - s \partial_\xi \end{pmatrix} \quad \mbox{and} \quad \cR^\# := \begin{pmatrix} f'\left(\hat u\right) - f'(0) & 0 \\ 0 & 0 \end{pmatrix}.
\end{equation}
\end{subequations}
%

%%%%%%%%%%%%%%%%%%%%%%%%%%%%%%%%%%%%%%%%%%%
\begin{lemma}\label{lem:dissipative}
For $Y \in D\left(\cL^\#\right) = H^2(\R;\C) \eoperp H^1(\R;\C)$ we have
\begin{equation}\label{dissipative}
\inner{\cL_{\pm\infty}^\# Y}{Y}{H_\C} \le - \kappa \vertii{Y}_{H_\C}^2 - \eps Z \nu \vertii{\tfrac{\d w}{\d\xi}}^2_{L^2(\R)} \quad \mbox{with} \quad \kappa \stackrel{\eqref{def_kappa}}{=} \min\left\{-f'(0),\eps\gamma\right\}.
\end{equation}
In particular, $\cL_{\pm\infty}^\#$ generates a $C_0$-semigroup $\left(P_{st}^{\pm\infty}\right)_{t \ge 0}$ of contractions in $H_\C$ satisfying
\begin{equation}\label{est_semigroupL0_contr}
\vertii{P^{\pm\infty}_{st}}_{L\left(H_\C\right)} \le e^{-\kappa t}.
\end{equation}
\end{lemma}
%%%%%%%%%%%%%%%%%%%%%%%%%%%%%%%%%%%%%%%%%%%

%%%%%%%%%%%%%%%%%%%%%%%%%%%%%%%%%%%%%%%%%%%
\begin{proof}
By density, we may assume $Y = \left(w,q\right)^\trans \in \left(C_\mathrm{c}^\infty(\R;\C)\right)^2$ and obtain 
\begin{eqnarray*}
\inner{\cL_{\pm\infty}^\# Y}{Y}{H_\C} &=& \eps Z \inner{\nu \tfrac{\d^2 w}{\d\xi^2} + f'(0) w - s \tfrac{\d w}{\d\xi} - q}{w}{L^2(\R;\C)} + Z \inner{\eps (w - \gamma q) - s \tfrac{\d q}{\d\xi}}{q}{L^2(\R;\C)} \\
&=& \eps Z \left(- \nu \vertii{\tfrac{\d w}{\d\xi}}^2_{L^2(\R;\C)} + f'(0) \vertii{w}^2_{L^2(\R;\C)}\right) - \eps Z \gamma \vertii{q}^2_{L^2(\R;\C)} \\
&& + 2 i \eps Z \, \Im\left(\inner{w}{q}{L^2(\R;\C)}\right) - \eps Z s \int_\R\left(\overline{\tfrac{\d w}{\d\xi}}\right) w \, \d\xi - Z s \int_\R \left(\overline{\tfrac{\d q}{\d\xi}}\right) q \, \d\xi.
\end{eqnarray*}
This implies
\begin{eqnarray*}
\Re \left(\inner{\cL_{\pm\infty}^\# Y}{Y}{H_\C}\right) &\stackrel{\eqref{inner_c}}{\le}& - \min\left\{-f'(0),\eps\gamma\right\} \vertii{Y}_{H_\C}^2 - \frac{Z s}{2} \int_\R \tfrac{\d}{\d\xi} \left(\eps \verti{w}^2 + \verti{q}^2\right) \d\xi \\
&& - \eps Z \nu \vertii{\tfrac{\d w}{\d\xi}}^2_{L^2(\R)}\\
&=& - \kappa \vertii{Y}_{H_\C}^2 - \eps Z \nu \vertii{\tfrac{\d w}{\d\xi}}^2_{L^2(\R)},
\end{eqnarray*}
which is \eqref{dissipative}. From \eqref{dissipative}, we recognize that
\[
\cL_{\pm\infty}^\# + \kappa \id_{H_\C} \colon D\left(\cL^\#\right) = H^2(\R;\C) \eoperp H^1(\R;\C) \to H_\C \stackrel{\eqref{hilbert_c}}{=} L^2(\R;\C) \eoperp L^2(\R;\C)
\]
is dissipative, so that, as in the proof of Proposition~\ref{prop:family}, we deduce with the Lax-Milgram and the Lumer-Philips theorem \cite[Chapter~1, Theorem~4.3]{Pazy1992} that $\cL_{\pm\infty}^\#$ generates a $C_0$-semigroup of contractions $(P^{\pm\infty}_{st})_{t\geq 0}$ in $H_\C$ meeting \eqref{est_semigroupL0_contr}.
\end{proof}
%%%%%%%%%%%%%%%%%%%%%%%%%%%%%%%%%%%%%%%%%%%

%%%%%%%%%%%%%%%%%%%%%%%%%%%%%%%%%%%%%%%%%%%
\begin{proof}[Proof of Proposition~\ref{prop:frozen}~\eqref{item:frozen_semi}]
Since $\cR^\#$ is a bounded operator in $H_\C$ with
\[
\vertii{\cR^\#}_{L\left(H_\C\right)} \le \vertii{f'\left(\hat u\right) - f'(0)}_{L^\infty(\R)},
\]
we conclude with Lemma~\ref{lem:dissipative} and \cite[Chapter~5, Theorem~2.3]{Pazy1992} that $\cL^\#$ generates a $C_0$-semigroup $\left(P^\#_{st}\right)_{t\geq0}$ in $H_\C$ meeting the bound
\begin{equation}\label{semi_fw}
\vertii{P_{st}^\#}_{L\left(H_\C\right)} \le e^{- \rho t} \quad \mbox{with} \quad \rho := \kappa - \vertii{f'\left(\hat u\right) - f'(0)}_{L^\infty(\R)}.
\end{equation}
Since $\cL^\#$ has real coefficients (cf.~\eqref{fw_op}), the statement for $H$ instead of $H_\C$ is immediate.
\end{proof}
%%%%%%%%%%%%%%%%%%%%%%%%%%%%%%%%%%%%%%%%%%%

%%%%%%%%%%%%%%%%%%%%%%%%%%%%%%%%%%%%%%%%%%%
\begin{proof}[Proof of Proposition~\ref{prop:frozen}~\eqref{item:semi_2}]
We have
\[
\left(\partial_t - \cL^\#\right) \cT_{-st} P_{st,st'} \cT_{st'} \stackrel{\eqref{lst_lf}}{=} \cT_{-st} \left(\partial_t - \cL_{st}\right) P_{st,st'} \cT_{st'} = 0,
\]
so that with
\[
\left. \cT_{-st} P_{st,st'} \cT_{st'} \right|_{t = t'} = \cT_{-st'} \id_\cV \cT_{st'} = \left. \cT_{-st'} \cT_{st'} \right|_\cV = \id_\cV
\]
the claim follows by uniqueness of the evolution family.
\end{proof}
%%%%%%%%%%%%%%%%%%%%%%%%%%%%%%%%%%%%%%%%%%%

Note that \eqref{semi_fw} just provides a rough estimate on the action of the semigroup $\left(P_{st}^\#\right)_{t \ge 0}$. In what follows, we provide a more detailed spectral analysis of $\cL^\#$, leading to the proofs of Proposition~\ref{prop:frozen}~\eqref{item:spectrum} and \eqref{item:riesz}.

%%%%%%%%%%%%%%%%%%%%%%%%%%%%%%%%%%%%%%%%%%%
\begin{lemma}\label{lem:spec_approx}
We have
\[
\sigma_\mathrm{ess}(\cL_{\pm\infty}^\#) \subseteq \left\{\lambda \in \C \colon \Re \lambda \le - \kappa\right\}, \quad \mbox{where} \quad \kappa \stackrel{\eqref{def_kappa}}{=} \min\left\{-f'(0),\eps\gamma\right\}.
\]
In particular, $\sigma_\mathrm{ess}(\cL_{\pm\infty}^\#)$ lies to the left of the imaginary axis.
\end{lemma}
%%%%%%%%%%%%%%%%%%%%%%%%%%%%%%%%%%%%%%%%%%%

%%%%%%%%%%%%%%%%%%%%%%%%%%%%%%%%%%%%%%%%%%%
\begin{proof}
We first use that $\sigma_\mathrm{ess}\left(\cL_{\pm\infty}^\#\right) \subseteq \cR_{H_\C}\left(\cL_{\pm\infty}^\#\right)$ (cf.~Definition~\ref{def:spectrum}~\eqref{item:essential} and \eqref{item:numeric}, and \cite[Lemma~4.1.9]{KapitulaPromislow2013}). Then the result is immediate from \eqref{sign_f}, Definition~\ref{def:spectrum}~\eqref{item:numeric}, and Lemma~\ref{lem:dissipative}.
\end{proof}
%%%%%%%%%%%%%%%%%%%%%%%%%%%%%%%%%%%%%%%%%%%

Lemma~\ref{lem:spec_approx} can be lifted to obtain the essential spectrum of $\cL^\#$ using the following compactness argument:

%%%%%%%%%%%%%%%%%%%%%%%%%%%%%%%%%%%%%%%%%%%
\begin{lemma}[compactness]\label{lem:compact}
The operator (cf.~\eqref{decomp_frozen_r})
\[
\cR^\# \colon H_\C \supset H^1(\R;\C) \eoperp L^2(\R;\C) \to H_\C
\]
is compact. In particular,
\[
\cL^\# \stackrel{\eqref{decomp_frozen_op}} = \cL^\#_{\pm\infty} + \cR^\# \colon H_\C \supset D(\cL^\#) = H^2(\R;\C) \eoperp H^1(\R;\C) \to H_\C
\]
is a relatively compact perturbation of $\cL^\#_{\pm\infty} \colon H_\C \supset D(\cL^\#) \to H_\C$.
\end{lemma}
%%%%%%%%%%%%%%%%%%%%%%%%%%%%%%%%%%%%%%%%%%%

%%%%%%%%%%%%%%%%%%%%%%%%%%%%%%%%%%%%%%%%%%%
\begin{proof}
Though the proof is standard, for the sake of a self-contained presentations we provide all necessary details. Suppose that $\left(Y_n\right)_n \in \left(H^1(\R;\C) \eoperp L^2(\R;\C)\right)^\N$ with $Y_n = \left(w_n,q_n\right)^\trans$ and $\sup_{n \in \N} \vertii{w_n}_{H^1(\R;\C)} < \infty$. Because of
\[
\cR^\# Y_n \stackrel{\eqref{decomp_frozen_r}}{=} \begin{pmatrix} \left(f'(\hat u) - f'(0)\right) w_n \\ 0 \end{pmatrix},
\]
we first prove that $\left(\left(f'(\hat u) - f'(0)\right) w_n\right)_n$ has a convergent subsequence in $L^2(\R;\C)$. Therefore, note that $f''(\hat u(\xi)) \tfrac{\d\hat u}{\d\xi}$ and $\verti{f'(\hat u(\xi)) - f'(0)} \le \sup_{w \in [0,\hat u(\xi)] \cup [\hat u(\xi),0]} \verti{f''(w)} \verti{\hat u(\xi)}$ decay exponentially as $\verti{\xi} \to \pm\infty$. Now, setting $y := \arctan\xi$ and $\phi_n(y) := \sqrt{1+\xi^2} \left(f'(\hat u(\xi)) - f'(0)\right) w_n(\xi)$, we may compute that $\d y = \frac{1}{1+\xi^2} \, \d\xi$, $\tfrac{\d}{\d y} = \left(1+\xi^2\right) \tfrac{\d}{\d\xi}$, 
\begin{eqnarray*}
\tfrac{\d \phi_n}{\d y} &=& \xi \sqrt{1+\xi^2} \left(f'(\hat u(\xi)) - f'(0)\right) w_n(\xi) + \left(1+\xi^2\right)^{\frac 3 2} f''(\hat u(\xi)) \tfrac{\d\hat u}{\d\xi}(\xi) w_n(\xi) \\
&& + \left(1+\xi^2\right)^{\frac 3 2} \left(f'(\hat u(\xi)) - f'(0)\right) \tfrac{\d w_n}{\d\xi}(\xi),
\end{eqnarray*}
and therefore
\[
\vertii{\phi_n}_{H^1\left(\left(-\frac \pi 2,\frac \pi 2\right);\C\right)}^2 = \int_{-\frac \pi 2}^{\frac \pi 2} \left(\verti{\phi_n}^2 + \verti{\tfrac{\d\phi_n}{\d y}}^2\right) \d y
\le C \int_\R \left(\verti{w_n}^2 + \verti{\tfrac{\d w_n}{\d\xi}}^2\right) \d\xi = C \vertii{w_n}_{H^1(\R;\C)}^2
\]
for an $n$-independent constant $C < \infty$. By the Rellich-Kondrachov theorem, $(\phi_n)_n$ has a convergent subsequence in $L^2\left(\left(-\tfrac \pi 2, \tfrac \pi 2\right);\C\right)$ and because
\[
\vertii{\phi_n - \phi_m}_{L^2\left(\left(-\tfrac \pi 2, \tfrac \pi 2\right);\C\right)} = \vertii{\left(f'(\hat u(\xi)) - f'(0)\right) \left(w_n-w_m\right)}_{L^2(\R;\C)},
\]
we infer that $\left( \left(f'(\hat u(\xi)) - f'(0)\right) w_n\right)_n$ has a convergent subsequence in $L^2(\R;\C)$.

\medskip

Suppose that $(Y_n)_n \in \left(D(\cL^\#)\right)^\N$ meets $\sup_{n \in \N} \left(\vertii{Y_n}_{H_\C} + \vertii{\cL_{\pm\infty}^\# Y_n}_{H_\C}\right) < \infty$. Then, it suffices to show that $\sup_{n \in \N} \vertii{w_n}_{H^1(\R;\C)} < \infty$. Observe that by interpolation, we have for $Y = \left(w,q\right)^\trans \in \left(C_\mathrm{c}^\infty(\R;\C)\right)^2$
\begin{eqnarray*}
\vertii{\cL^\#_{\pm\infty} Y}^2_{H_\C} &=&  \eps \vertii{\nu \tfrac{\d^2 w}{\d\xi^2} + f'(0) w - s \tfrac{\d w}{\d\xi} - q}^2_{L^2(\R;\C)} + \vertii{\eps w - \eps \gamma q - s \tfrac{\d q}{\d\xi}}^2_{L^2(\R;\C)} \\
&\ge& C_1 \vertii{\tfrac{\d w}{\d\xi}}_{L^2(\R;\C)}^2 - C_2 \left(\vertii{w}_{L^2(\R;\C)}^2 + \vertii{q}_{L^2(\R;\C)}^2\right),
\end{eqnarray*}
with $C_1 > 0$ and $C_2 < \infty$ independent of $n$. Since $\vertii{Y_n}_{H_\C}^2 \stackrel{\eqref{inner_c}}{=} \eps Z \vertii{w_n}_{L^2(\R;\C)}^2 + Z \vertii{q_n}_{L^2(\R;\C)}^2$, we may conclude with \eqref{sign_f} that indeed $\sup_{n \in \N} \vertii{w_n}_{H^1(\R;\C)} < \infty$ holds true.
\end{proof}
%%%%%%%%%%%%%%%%%%%%%%%%%%%%%%%%%%%%%%%%%%%

%%%%%%%%%%%%%%%%%%%%%%%%%%%%%%%%%%%%%%%%%%%
\begin{proof}[Proof of Proposition~\ref{prop:frozen}~\eqref{item:essential-frozen}]
If
\begin{equation}\label{id_ess}
\sigma_\mathrm{ess}(\cL^\#) = \sigma_\mathrm{ess}(\cL_{\pm\infty}^\#)
\end{equation}
holds true, it follows in particular that
\[
\sigma_\mathrm{ess}(\cL^\#) \subseteq \left\{\lambda \in \C \colon \Re \lambda \le - \kappa\right\}, \quad \mbox{where} \quad \kappa \stackrel{\eqref{def_kappa}}{=} \min\left\{-f'(0),\eps\gamma\right\}
\]
by Lemma~\ref{lem:spec_approx}. The equality \eqref{id_ess}, on the other hand, follows by Weyl's essential spectrum theorem \cite[Theorem~2.2.6]{KapitulaPromislow2013} because the operators
\[
\cL^\#, \cL^\#_{\pm\infty} \colon D\left(\cL^\#\right) = H^2(\R;\C) \eoperp H^1(\R;\C) \to H_\C \stackrel{\eqref{hilbert_c}}{=} L^2(\R;\C) \eoperp L^2(\R;\C)
\]
are closed (which follows by an interpolation argument as in the proof of Lemma~\ref{lem:compact}) and $\cL^\#$ is a compact pertubation of $\cL^\#$ by Lemma~\ref{lem:compact}.
\end{proof}
%%%%%%%%%%%%%%%%%%%%%%%%%%%%%%%%%%%%%%%%%%%

%%%%%%%%%%%%%%%%%%%%%%%%%%%%%%%%%%%%%%%%%%%
\begin{proof}[Proof of Proposition~\ref{prop:frozen}~\eqref{item:point}]
It suffices to prove that for $\lambda \in \sigma_\mathrm{p}(\cL^\#)$ any corresponding eigenvector $Y = (w,q)^\trans \in D(\cL^\#)$ is bounded and has infinitely many bounded derivatives to conclude that in $H_\C$ there are no additional eigenvalues compared to the ones obtained by Yanagida~\cite[\S5.1]{Yanagida1985} (cf.~\cite[\S1, Theorem]{Jones1984} in the case of the cubic polynomial without cut off).

\medskip

Indeed, for $Y$ as above satisfying $\cL^\# Y = \lambda Y$, we have
\begin{align*}
\nu \tfrac{\d^2 w}{\d\xi^2} &\stackrel{\eqref{fw_op}}{=} \underbrace{\lambda w - f'(\hat u) w +s \tfrac{\d w}{\d\xi} + q}_{\in H^1(\R;\C)}, \\
s \tfrac{\d q}{\d\xi} &\stackrel{\eqref{fw_op}}{=} \underbrace{ -\lambda q + \eps w - \eps \gamma q}_{\in H^1(\R;\C)},
\end{align*}
which immediately yields $w \in H^3(\R;\C)$ and $q \in H^2(\R;\C)$. Inductively, we obtain that $w,q \in H^k(\R;\C)$ for any $k \in \N$, giving smoothness and boundedness of $w$, $q$, and all derivatives.

\medskip

With the stability analysis in~\cite[\S5.1]{Yanagida1985} we conclude that, except for the simple eigenvalue $0$, all eigenvalues are to the left of the imaginary axis with real part bounded by $\lambda^*(\eps)$. We further remark that in \eqref{frozen_tw_0} it has already been noted that $\frac{\d \hat X}{\d \xi}$ is an eigenvector to the eigenvalue $0$.
\end{proof}
%%%%%%%%%%%%%%%%%%%%%%%%%%%%%%%%%%%%%%%%%%%

In what follows, we need a regularizing effect of the semigroup $\left(P_{st}^{\pm\infty}\right)_{t \ge 0}$ of Lemma~\ref{lem:dissipative} in the first component.
%%%%%%%%%%%%%%%%%%%%%%%%%%%%%%%%%%%%%%%%%%%
\begin{lemma}\label{lem:reg_frozen}
For $Y^{(0)} = \left(w^{(0)}, q^{(0)}\right)^\trans \in H_\C$ we write $Y_{\pm\infty}(t,\cdot) := \left(w_{\pm\infty}(t,\cdot), q_{\pm\infty}(t,\cdot)\right)^\trans := P_{st}^{\pm\infty} Y^{(0)}$. Then, it holds $w_{\pm\infty}, q_{\pm\infty}, \frac{\sqrt t \partial_\xi w_{\pm\infty}}{1+\sqrt t} \in L^\infty\left([0,\infty);L^2(\R;\C)\right)$ with
\[
\vertii{w_{\pm\infty}}_{L^\infty\left([0,\infty);L^2(\R;\C)\right)} + \vertii{q_{\pm\infty}}_{L^\infty\left([0,\infty);L^2(\R;\C)\right)} + \vertii{\frac{\sqrt t \partial_\xi w_{\pm\infty}}{1+\sqrt t}}_{L^\infty\left([0,\infty);L^2(\R;\C)\right)} \le C \vertii{Y^{(0)}}_{H_\C},
\]
where $C < \infty$ is independent of $Y^{(0)}$.
\end{lemma}
%%%%%%%%%%%%%%%%%%%%%%%%%%%%%%%%%%%%%%%%%%%

%%%%%%%%%%%%%%%%%%%%%%%%%%%%%%%%%%%%%%%%%%%
\begin{proof}
First, assume $Y^{(0)} = \left(w^{(0)}, q^{(0)}\right)^\trans \in \left(C_\mathrm{c}^\infty(\R;\C)\right)^2$. Then, by construction
\begin{equation}\label{eq_ypm}
\partial_t Y_{\pm\infty} - \cL_{\pm\infty}^\# Y_{\pm\infty} = 0,
\end{equation}
which implies
\[
\partial_t^j \left(\cL_{\pm\infty}^\#\right)^k Y_{\pm\infty} = \left(\cL_{\pm\infty}^\#\right)^{j+k} Y_{\pm\infty} = P_{st}^{\pm\infty} \left(\cL_{\pm\infty}^\#\right)^{j+k} Y^{(0)} \quad \mbox{for all} \quad j,k \in \N_0
\]
and hence $\left(\cL_{\pm\infty}^\#\right)^k Y_{\pm\infty} \in C^\infty\left([0,\infty);H_\C\right)$ for any $k \in \N_0$. The coercivity/dissipativity estimate \eqref{dissipative} of Lemma~\ref{lem:dissipative} then yields $w_{\pm\infty} \in C^\infty\left([0,\infty);H^k(\R)\right)$ for any $k \in \N_0$.

\medskip

Having established qualitative regularity for regular initial data, we next prove a-priori estimates. Indeed, by testing \eqref{eq_ypm} it follows
\[
\frac{1}{2} \frac{\d}{\d t} \vertii{Y_{\pm\infty}}^2_{H_\C} - \inner{\cL_0^\# Y_{\pm\infty}}{Y_{\pm\infty}}{H_\C} = 0,
\]
so that with \eqref{dissipative} of Lemma~\ref{lem:dissipative} we deduce that
\[
\frac{1}{2} \frac{\d}{\d t} \vertii{Y(t,\cdot)}^2_{H_\C} + \kappa \vertii{Y(t,\cdot)}^2_{H_\C} + \eps Z \nu \vertii{\partial_\xi w(t,\cdot)}^2_{L^2(\R;\C)} \le 0.
\]
Integrating in time yields
\begin{equation}\label{est_reg_p1}
\frac{1}{2} \vertii{Y_{\pm\infty}(t,\cdot)}^2_{H_\C} + \int_0^t \left(\kappa \vertii{Y_{\pm\infty}(t',\cdot)}^2_{H_\C} + \eps Z \nu \vertii{\partial_\xi w_{\pm\infty}(t',\cdot)}^2_{L^2(\R;\C)}\right) \d t' \le \frac{1}{2} \vertii{Y^{(0)}}^2_{H_\C}.
\end{equation}
To obtain a point-wise bound in time on $\vertii{\partial_\xi w_{\pm\infty}(t,\cdot)}^2_{L^2(\R)}$, observe that from \eqref{decomp_frozen_r} and \eqref{eq_ypm} it follows
\[
\partial_t \partial_\xi w_{\pm\infty} - \nu \partial_\xi^3 w_{\pm\infty} - f'(0) \partial_\xi w_{\pm\infty} + s \partial_\xi^2 w_{\pm\infty} + \partial_\xi q_{\pm\infty} = 0.
\]
Testing with $t \partial_\xi w_{\pm\infty}$ gives with
\[
\Re \inner{\partial_\xi^2 w_{\pm\infty}(t,\cdot)}{\partial_\xi w_{\pm\infty}(t,\cdot)}{L^2(\R;\C)} = \frac 1 2 \Re \int_\R \partial_\xi \verti{\partial_\xi w_{\pm\infty}(t,\xi)}^2 \d\xi = 0
\]
and
\[
\Re \inner{\partial_\xi q_{\pm\infty}(t,\cdot)}{\partial_\xi w_{\pm\infty}(t,\cdot)}{L^2(\R)} \le \frac \nu 2 \vertii{\partial_\xi^2 w_{\pm\infty}(t,\cdot)}_{L^2(\R;\C)}^2 + \frac{1}{2 \nu} \vertii{q_{\pm\infty}(t,\cdot)}_{L^2(\R;\C)}^2
\]
that
\begin{align*}
& \frac{t}{2} \vertii{\partial_\xi w_{\pm\infty}(t,\cdot)}^2_{L^2(\R;\C)}
+ \int_0^t t' \left(\frac \nu 2 \vertii{\partial_\xi^2 w_{\pm\infty}(t',\cdot)}_{L^2(\R)}^2 - f'(0) \vertii{\partial_\xi w_{\pm\infty}(t',\cdot)}_{L^2(\R)}^2\right) \d t' \\
& \quad \le \frac{t}{2 \nu} \int_0^t \vertii{q_{\pm\infty}(t',\cdot)}_{L^2(\R;\C)}^2 \d t' + \frac 1 2 \int_0^t \vertii{\partial_\xi w_{\pm\infty}(t',\cdot)}_{L^2(\R;\C)}^2 \d t'.
\end{align*}
The combination with \eqref{inner_c} and \eqref{est_reg_p1} yields that there exists a constant $C < \infty$ such that
\[
\esssup_{t \ge 0} \frac{\sqrt t }{1 + \sqrt t} \vertii{\partial_\xi w_{\pm\infty}(t,\cdot)}_{L^2(\R;\C)} \le C \vertii{Y^{(0)}}^2_{H_\C}.
\]
The statement of the lemma now follows by density of $\left(C_\mathrm{c}^\infty(\R;\C)\right)^2$ in $H_\C$.
\end{proof}
%%%%%%%%%%%%%%%%%%%%%%%%%%%%%%%%%%%%%%%%%%%

%%%%%%%%%%%%%%%%%%%%%%%%%%%%%%%%%%%%%%%%%%%
\begin{proof}[Proof of Proposition~\ref{prop:frozen}~\eqref{item:riesz}]
We loosely follow the approach in \cite[Section 3.2, Proposition 3.5]{ArioliKoch2015}. First observe that by Lemma~\ref{lem:compact} and Lemma~\ref{lem:reg_frozen} the operator $\cR^\# P_{st}^{\pm\infty} \Pi^\# \in L(H_\C)$ is compact for any $t > 0$. This compactness implies thanks to \cite[Proposition 3.4]{ArioliKoch2015} that for every $t \ge 0$ the operator $P^\#_{st} \Pi^\# - P_{st}^{\pm\infty} \Pi^\#$ is compact as well. By estimate \eqref{est_semigroupL0_contr} of Lemma~\ref{lem:dissipative} and the Neumann series, we recognize that the operator $P_{st}^{\pm\infty} \Pi^\#$ has no spectrum outside the disc $\left\{ \mu \in \C: \verti{\mu} \leq e^{-\kappa t} \vertii{\Pi^\#}_{L(H_\C)}\right\}$. Now, since $P^\#_{st} \Pi^\#$ is a compact perturbation of $P_{st}^{\pm\infty} \Pi^\#$, the spectrum of $P^\#_{st} \Pi^\#$ in
\[
\left\{ \mu \in \C: \verti{\mu} > e^{-\kappa t} \vertii{\Pi^\#}_{L(H_\C)} \right\}
\]
only contains point spectrum $\sigma_\mathrm{p}\left(P^\#_{st} \Pi^\#\right)$ (cf.~Definition~\ref{def:spectrum} and \cite[Theorem~2.2.6]{KapitulaPromislow2013}). Using \cite[Chapter~2, Theorem~2.4]{Pazy1992}, we infer that
\[
\sigma_\mathrm{p}\left(P^\#_{st} \Pi^\#\right) \subseteq \{0\} \cup \left\{e^{\lambda t} \colon \lambda \in \sigma_\mathrm{p}\left(\cL^\# \Pi^\#\right)\right\} \subseteq \left\{ \mu \in \C: \verti{\mu} \le e^{\max\{-\kappa,\lambda^*(\eps)\} t} \right\},
\]
where we have used Proposition~\ref{prop:frozen}~\eqref{item:point} in the last inclusion and that $\sigma_\mathrm{p}\left(\cL^\# \Pi^\#\right) = \sigma_\mathrm{p}\left(\cL^\#\right) \setminus \{0\}$ since $0$ is not an eigenvalue of $\cL^\# \Pi^\#$. Altogether, we have
\[
\sigma\left(P^\#_{st} \Pi^\#\right) \subseteq \left\{ \mu \in \C: \verti{\mu} \le e^{\max\{-\kappa,\lambda^*(\eps)\} t} \vertii{\Pi^\#}_{L(H_\C)} \right\}
\]
because $\vertii{\Pi^\#}_{L(H_\C)} \ge 1$.

\medskip

Now, let $\max\{-\kappa,\lambda^*(\eps) \} < - \vartheta < 0$. Since the spectral radius $\lim_{n \to \infty} \vertii{\left(P^\#_{st} \Pi^\#\right)^n}^{\frac{1}{n}}_{L\left(H_\C\right)}$ of $P^\#_{st} \Pi^\#$ meets
\[
\lim_{n \to \infty} \vertii{\left(P^\#_{st} \Pi^\#\right)^n}^{\frac{1}{n}}_{L\left(H_\C\right)} = \lim_{n \to \infty} \vertii{P^\#_{st n} \Pi^\#}^{\frac{1}{n}}_{L\left(H_\C\right)}
\]
and
\[
\lim_{n \to \infty} \vertii{\left(P^\#_{st} \Pi^\#\right)^n}^{\frac{1}{n}}_{L\left(H_\C\right)} = \lim_{n \to \infty} \max\left\{\verti{\mu} \colon \mu \in \sigma\left(P^\#_{s t} \Pi^\#\right)\right\} \le e^{\max\{-\kappa,\lambda^*(\eps)\} t} \vertii{\Pi^\#}_{L(H_\C)},
\]
we have $\vertii{P^\#_{s t n} \Pi^\#}^{\frac{1}{n}}_{L\left(H_\C\right)} \leq e^{-\vartheta t}$ for $t$ and $n$ large enough. Thus there exists $C_{\vartheta} < \infty$ such that \eqref{est_semigroupFrozenWave} holds true.
\end{proof}
%%%%%%%%%%%%%%%%%%%%%%%%%%%%%%%%%%%%%%%%%%%

%%%%%%%%%%%%%%%%%%%%%%%%%%%%%%%%%%%%%%%%%%%
\begin{proof}[Proof of Proposition~\ref{prop:frozen}~\eqref{item:riesz2}]
Suppose that $Y \in H$ is real-valued. Then
\[
\Pi^{\#,0} Y \stackrel{\eqref{riesz_projection}}{=} \frac{1}{2 \pi i} \ointctrclockwise_{\verti{\lambda} = r} X_\lambda \, \d\lambda = \frac{r}{2 \pi} \int_0^{2\pi} X_{r e^{i\tau}} e^{i\tau} \, \d\tau, \quad \mbox{where} \quad X_\lambda := \left(\lambda \id_H - \cL^\#\right)^{-1} Y.
\]
From the equation
\[
\lambda X_\lambda - \cL^\# X_\lambda = Y, \quad \mbox{where} \quad \verti{\lambda} = r,
\]
and because the coefficients of $\cL^\#$ are real (cf.~\eqref{fw_op}), it follows that
\[
\overline\lambda \, \overline{X_\lambda} - \cL^\# \overline{X_\lambda} = Y.
\]
Because $r > 0$ us chosen sufficiently small we have $\lambda \in \rho(\cL^\#)$ (cf.~Definition~\ref{def:spectrum}~\eqref{item:resolvent_set}) and it follows due to uniqueness $\overline{X_\lambda} = X_{\overline\lambda}$. Hence,
\begin{eqnarray*}
\overline{\Pi^{\#,0} Y} &=& \frac{r}{2 \pi} \int_0^{2\pi} \overline{X_{r e^{i\tau}}} e^{-i\tau} \, \d\tau = \frac{r}{2 \pi} \int_0^{2\pi} X_{r e^{-i\tau}} e^{-i\tau} \, \d\tau = \frac{r}{2 \pi} \int_{-2\pi}^0 X_{r e^{i\tau}} e^{i\tau} \, \d\tau \\
&=& \frac{r}{2 \pi} \int_0^{2\pi} X_{r e^{i\tau}} e^{i\tau} \, \d\tau = \Pi^{\#,0} Y.
\end{eqnarray*}
By \eqref{riesz_projection} also $\overline{\Pi^\# Y} = \Pi^\# Y$, i.e., $\Pi^{\#,0},\Pi^\# \colon H \to H$ are well-defined and the estimates of the operator norms are trivial, too.
\end{proof}
%%%%%%%%%%%%%%%%%%%%%%%%%%%%%%%%%%%%%%%%%%%

%
\bibliography{eichinger_gnann_kuehn_fhn_v2} %your .bib file

\begin{thebibliography}{10}

\bibitem{AntonopoulouBatesBloemkerKarali}
D.~C. Antonopoulou, P.~W. Bates, D.~Bl{\"{o}}mker, and G.~D. Karali.
\newblock Motion of a droplet for the stochastic mass-conserving {Allen-Cahn}
  equation.
\newblock {\em SIAM J. Math. Anal.}, 48(1):670--708, 2016.

\bibitem{ArioliKoch2015}
G.~Arioli and H.~Koch.
\newblock Existence and stability of traveling pulse solutions of the
  {F}itz{H}ugh-{N}agumo equation.
\newblock {\em Nonlinear Anal.}, 113:51--70, 2015.

\bibitem{BashkirtsevaRyashko}
I.~Bashkirtseva and L.~Ryashko.
\newblock Analysis of excitability for the fitzhugh-nagumo model via a
  stochastic sensitivity function technique.
\newblock {\em Phys. Rev. E}, 83(6):061109, 2011.

\bibitem{BatesLuWang}
P.~W. Bates, K.~Lu, and B.~Wang.
\newblock Random attractors for stochastic reaction-diffusion equations on
  unbounded domains.
\newblock {\em J. Differen. Equat.}, 246(2):845--869, 2009.

\bibitem{BenArtziGohberg}
A.~Ben-Artzi and I.~Gohberg.
\newblock Dichotomy of systems and invertibility of linear ordinary
  differential operators.
\newblock In {\em Time-variant systems and interpolation}, volume~56 of {\em
  Oper. Theory Adv. Appl.}, pages 90--119. Birkh\"{a}user, Basel, 1992.

\bibitem{BerglundGentz10}
N.~Berglund and B.~Gentz.
\newblock Sharp estimates for metastable lifetimes in parabolic {SPDEs}:
  Kramers' law and beyond.
\newblock {\em Electronic J. Probability}, 18(24):1--58, 2013.

\bibitem{BerglundGentzKuehn1}
N.~Berglund, B.~Gentz, and C.~Kuehn.
\newblock From random {Poincar{\'e}} maps to stochastic mixed-mode-oscillation
  patterns.
\newblock {\em J. Dyn. Diff. Equat.}, 27(1):83--136, 2015.

\bibitem{BerglundKuehn}
N.~Berglund and C.~Kuehn.
\newblock Regularity structures and renormalisation of {FitzHugh-Nagumo SPDEs}
  in three space dimensions.
\newblock {\em Electron. J. Probab.}, 21(18):1--48, 2016.

\bibitem{BerglundLandon}
N.~Berglund and D.~Landon.
\newblock Mixed-mode oscillations and interspike interval statistics in the
  stochastic {FitzHugh-Nagumo} model.
\newblock {\em Nonlinearity}, 25:2303--2335, 2012.

\bibitem{Bloemker}
D.~Bl{\"{o}}mker.
\newblock {\em Amplitude Equations for Stochastic Partial Differential
  Equations}.
\newblock World Scientific, 2007.

\bibitem{BloemkerHairerPavliotis}
D.~Bl{\"{o}}mker, M.~Hairer, and G.~A. Pavliotis.
\newblock Multiscale analysis for stochastic partial differential equations
  with quadratic nonlinearities.
\newblock {\em Nonlinearity}, 20(7):1721--1744, 2007.

\bibitem{BonaccorsiMastrogiacomo}
S.~Bonaccorsi and E.~Mastrogiacomo.
\newblock Analysis of the stochastic {Fitzhugh-Nagumo} system.
\newblock {\em Infin. Dimens. Anal. Quantum Probab. Relat. Top.},
  11(3):427--446, 2008.

\bibitem{Carpenter1974}
G.~A. Carpenter.
\newblock {\em Traveling-wave solutions of nerve impulse equations}.
\newblock ProQuest LLC, Ann Arbor, MI, 1974.
\newblock Thesis (Ph.D.)--The University of Wisconsin - Madison.

\bibitem{CarterdeRijkSandstede}
P.~Carter, B.~de~Rijk, and B.~Sandstede.
\newblock Stability of traveling pulses with oscillatory tails in the
  {FitzHugh-Nagumo} system.
\newblock {\em J. Nonlinear Science}, 26(5):1369--1444, 2016.

\bibitem{CarterSandstede}
P.~Carter and B.~Sandstede.
\newblock Fast pulses with oscillatory tails in the fitzhugh-nagumo system.
\newblock {\em SIAM J. Math. Anal.}, 47(5):3393--3441, 2015.

\bibitem{Sneydetal}
A.~R. Champneys, V.~Kirk, E.~Knobloch, B.~E. Oldeman, and J.~Sneyd.
\newblock {When Shil'nikov meets Hopf in excitable systems}.
\newblock {\em SIAM J. Appl. Dyn. Syst.}, 6(4):663--693, 2007.

\bibitem{ChenChoi2015}
C.-N. Chen and Y.~S. Choi.
\newblock Traveling pulse solutions to {F}itz{H}ugh-{N}agumo equations.
\newblock {\em Calc. Var. Partial Differential Equations}, 54(1):1--45, 2015.

\bibitem{Chen1}
X.~Chen.
\newblock Existence, uniqueness, and asymptotic stability of travelling waves
  in nonlocal evolution equations.
\newblock {\em Adv. Differential Equations}, 2:125--160, 1997.

\bibitem{ChiconeLatushkin}
C.~Chicone and Y.~Latushkin.
\newblock {\em Evolution Semigroups in Dynamical Systems and Differential
  Equations}.
\newblock Amer. Math. Soc., 1999.

\bibitem{Conley1975}
C.~C. Conley.
\newblock On traveling wave solutions of nonlinear diffusion equations.
\newblock In {\em Dynamical systems, theory and applications ({R}encontres,
  {B}attelle {R}es. {I}nst., {S}eattle, {W}ash., 1974)}, pages 498--510.
  Lecture Notes in Phys., Vol. 38. Springer, Berlin, 1975.

\bibitem{ConleyEaston1971}
C.~C. Conley and R.~Easton.
\newblock Isolated invariant sets and isolating blocks.
\newblock {\em Trans. Amer. Math. Soc.}, 158:35--61, 1971.

\bibitem{CornwellJones}
P.~Cornwell and C.~K. R.~T. Jones.
\newblock On the existence and stability of fast traveling waves in a doubly
  diffusive {FitzHugh-Nagumo} system.
\newblock {\em SIAM J. Appl. Dyn. Syst.}, 17(4):754--787, 2018.

\bibitem{DaPratoZabczyk}
G.~Da~Prato and J.~Zabczyk.
\newblock {\em Stochastic Equations in Infinite Dimensions}.
\newblock Cambridge University Press, 1992.

\bibitem{BouardDebussche2007}
A.~de~Bouard and A.~Debussche.
\newblock Random modulation of solitons for the stochastic {K}orteweg-de
  {V}ries equation.
\newblock {\em Ann. Inst. H. Poincar\'{e} Anal. Non Lin\'{e}aire},
  24(2):251--278, 2007.

\bibitem{BouardFukuizumi2009}
A.~de~Bouard and R.~Fukuizumi.
\newblock Modulation analysis for a stochastic {NLS} equation arising in
  {B}ose-{E}instein condensation.
\newblock {\em Asymptot. Anal.}, 63(4):189--235, 2009.

\bibitem{ErmentroutTerman}
G.~B. Ermentrout and D.~H. Terman.
\newblock {\em Mathematical Foundations of Neuroscience}.
\newblock Springer, 2010.

\bibitem{Evans1971_72}
J.~W. Evans.
\newblock Nerve axon equations. {I}. {L}inear approximations.
\newblock {\em Indiana Univ. Math. J.}, 21:877--885, 1971/72.

\bibitem{Evans1972_73}
J.~W. Evans.
\newblock Nerve axon equations. {II}. {S}tability at rest.
\newblock {\em Indiana Univ. Math. J.}, 22:75--90, 1972/73.

\bibitem{Evans1972_73_2}
J.~W. Evans.
\newblock Nerve axon equations. {III}. {S}tability of the nerve impulse.
\newblock {\em Indiana Univ. Math. J.}, 22:577--593, 1972/73.

\bibitem{Evans1974_75}
J.~W. Evans.
\newblock Nerve axon equations. {IV}. {T}he stable and the unstable impulse.
\newblock {\em Indiana Univ. Math. J.}, 24(12):1169--1190, 1974/75.

\bibitem{Evans1976_2}
J.~W. Evans.
\newblock Errata: ``{N}erve axon equations. {III}. {S}tability of the nerve
  impulse'' ({I}ndiana {U}niv. {M}ath. {J}. {\bf 22} (1972/73), 577--593).
\newblock {\em Indiana Univ. Math. J.}, 25(3):301, 1976.

\bibitem{Evans1976}
J.~W. Evans.
\newblock Erratum: ``{N}erve axon equations. {II}. {S}tability at rest''
  ({I}ndiana {U}niv. {M}ath. {J}. {\bf 22} (1972/73), 75--90).
\newblock {\em Indiana Univ. Math. J.}, 25(3):301, 1976.

\bibitem{EvansFenichelFeroe}
J.~W. Evans, N.~Fenichel, and J.~A. Feroe.
\newblock Double impulse solutions in nerve axon equations.
\newblock {\em SIAM J. Appl. Math.}, 42(2):219--234, 1982.

\bibitem{FitzHugh}
R.~FitzHugh.
\newblock Mathematical models of threshold phenomena in the nerve membrane.
\newblock {\em Bull. Math. Biophysics}, 17:257--269, 1955.

\bibitem{GardnerSmoller1983}
R.~Gardner and J.~Smoller.
\newblock The existence of periodic travelling waves for singularly perturbed
  predator-prey equations via the {C}onley index.
\newblock {\em J. Differential Equations}, 47(1):133--161, 1983.

\bibitem{GhazaryanLatushkinSchecter}
A.~Ghazaryan, Y.~Latushkin, and S.~Schecter.
\newblock Stability of traveling waves for degenerate systems of reaction
  diffusion equations.
\newblock {\em Indiana Uni. Math. J.}, 60(2):443--471, 2011.

\bibitem{GnannKuehnPein}
M.~V. Gnann, C.~Kuehn, and A.~Pein.
\newblock Towards sample path estimates for fast-slow {SPDEs}.
\newblock {\em Euro. J. Appl. Math.}, 30(5):1004--1024, 2019.

\bibitem{GuckenheimerKuehn1}
J.~Guckenheimer and C.~Kuehn.
\newblock {Homoclinic orbits of the FitzHugh-Nagumo equation: The singular
  limit}.
\newblock {\em DCDS-S}, 2(4):851--872, 2009.

\bibitem{GuckenheimerKuehn3}
J.~Guckenheimer and C.~Kuehn.
\newblock {Homoclinic orbits of the FitzHugh-Nagumo equation: Bifurcations in
  the full system}.
\newblock {\em SIAM J. Appl. Dyn. Syst.}, 9:138--153, 2010.

\bibitem{HamsterHupkes2020SIADS}
C.~H.~S. Hamster and H.~J. Hupkes.
\newblock Stability of traveling waves for reaction-diffusion equations with
  multiplicative noise.
\newblock {\em SIAM J. Appl. Dyn. Syst.}, 18(1):205--278, 2019.

\bibitem{HamsterHupkes2019}
C.~H.~S. Hamster and H.~J. Hupkes.
\newblock Stability of travelling waves for reaction-diffusion equations with
  multiplicative noise.
\newblock {\em SIAM J. Appl. Dyn. Syst.}, 18(1):205--278, 2019.

\bibitem{HamsterHupkes2020SIMA}
C.~H.~S. Hamster and H.~J. Hupkes.
\newblock Stability of traveling waves for systems of reaction-diffusion
  equations with multiplicative noise.
\newblock {\em SIAM J. Math. Anal.}, 52(2):1386--1426, 2020.

\bibitem{HamsterHupkes2020PhysD}
C.~H.~S. Hamster and H.~J. Hupkes.
\newblock Travelling waves for reaction-diffusion equations forced by
  translation invariant noise.
\newblock {\em Phys. D}, 401:132233, 35, 2020.

\bibitem{Hastings1976}
S.~P. Hastings.
\newblock On the existence of homoclinic and periodic orbits for the
  {F}itzhugh-{N}agumo equations.
\newblock {\em Quart. J. Math. Oxford Ser. (2)}, 27(105):123--134, 1976.

\bibitem{HodgkinHuxley4}
A.~L. Hodgkin and A.~F. Huxley.
\newblock A quantitative description of membrane current and its application to
  conduction and excitation in nerve.
\newblock {\em J. Physiol.}, 117:500--544, 1952.

\bibitem{InglisMacLaurin2016}
J.~Inglis and J.~MacLaurin.
\newblock A general framework for stochastic traveling waves and patterns, with
  application to neural field equations.
\newblock {\em SIAM J. Appl. Dyn. Syst.}, 15(1):195--234, 2016.

\bibitem{Izhikevich1}
E.~Izhikevich.
\newblock {\em Dynamical Systems in Neuroscience}.
\newblock MIT Press, 2007.

\bibitem{Jones1984}
C.~K. R.~T. Jones.
\newblock Stability of the travelling wave solution of the
  {F}itz{H}ugh-{N}agumo system.
\newblock {\em Trans. Amer. Math. Soc.}, 286(2):431--469, 1984.

\bibitem{JonesKaperKopell}
C.~K. R.~T. Jones, T.~J. Kaper, and N.~Kopell.
\newblock Tracking invariant manifolds up to exponentially small errors.
\newblock {\em SIAM J. Math. Anal.}, 27(2):558--577, 1996.

\bibitem{JonesKopell}
C.~K. R.~T. Jones and N.~Kopell.
\newblock Tracking invariant manifolds with differential forms in singularly
  perturbed systems.
\newblock {\em J. Differential Equat.}, 108(1):64--88, 1994.

\bibitem{JonesKopellLanger}
C.~K. R.~T. Jones, N.~Kopell, and R.~Langer.
\newblock Construction of the {FitzHugh-Nagumo} pulse using differential forms.
\newblock In {\em Multiple-Time-Scale Dynamical Systems}, pages 101--113.
  Springer, 2001.

\bibitem{KapitulaPromislow2013}
T.~Kapitula and K.~Promislow.
\newblock {\em Spectral and dynamical stability of nonlinear waves}, volume 185
  of {\em Applied Mathematical Sciences}.
\newblock Springer, New York, 2013.
\newblock With a foreword by Christopher K. R. T. Jones.

\bibitem{KaratzasShreve1991}
I.~Karatzas and S.~E. Shreve.
\newblock {\em Brownian motion and stochastic calculus}, volume 113 of {\em
  Graduate Texts in Mathematics}.
\newblock Springer-Verlag, New York, second edition, 1991.

\bibitem{KruegerStannat2014}
J.~Kr\"uger and W.~Stannat.
\newblock Front propagation in stochastic neural fields: a rigorous
  mathematical framework.
\newblock {\em SIAM J. Appl. Dyn. Syst.}, 13(3):1293--1310, 2014.

\bibitem{KruegerStannat2017}
J.~Kr\"{u}ger and W.~Stannat.
\newblock A multiscale-analysis of stochastic bistable reaction-diffusion
  equations.
\newblock {\em Nonlinear Anal.}, 162:197--223, 2017.

\bibitem{KrupaSandstedeSzmolyan}
M.~Krupa, B.~Sandstede, and P.~Szmolyan.
\newblock Fast and slow waves in the {FitzHugh-Nagumo} equation.
\newblock {\em J. Differ. Equat.}, 133:49--97, 1997.

\bibitem{KrylovRozovskii1979}
N.~V. Krylov and B.~L. Rozovski\u\i.
\newblock Stochastic evolution equations.
\newblock In {\em Current problems in mathematics, {V}ol. 14 ({R}ussian)},
  pages 71--147, 256. Akad. Nauk SSSR, Vsesoyuz. Inst. Nauchn. i Tekhn.
  Informatsii, Moscow, 1979.

\bibitem{KuehnBook}
C.~Kuehn.
\newblock {\em Multiple Time Scale Dynamics}.
\newblock Springer, 2015.

\bibitem{KuehnBook1}
C.~Kuehn.
\newblock {\em PDE Dynamics: An Introduction}.
\newblock SIAM, 2019.

\bibitem{KuehnSPDEwaves}
C.~Kuehn.
\newblock Travelling waves in monostable and bistable stochastic partial
  differential equations.
\newblock {\em Jahresber. Dtsch. Math.-Ver.}, 122(2):73--107, 2020.

\bibitem{KuehnNeamtuPein}
C.~Kuehn, A.~Neamtu, and A.~Pein.
\newblock Random attractors for stochastic partly dissipative systems.
\newblock {\em arXiv:1906.08594}, pages 1--25, 2019.

\bibitem{LiYin}
Y.~Li and J.~Yin.
\newblock A modified proof of pullback attractors in a {Sobolev} space for
  stochastic {Fitzhugh-Nagumo} equations.
\newblock {\em Discr. Cont. Dyn. Syst. B}, 21(4):1203--1223, 2016.

\bibitem{Lindneretal}
B.~Lindner, J.~Garcia-Ojalvo, A.~Neiman, and L.~Schimansky-Geier.
\newblock Effects of noise in excitable systems.
\newblock {\em Physics Reports}, 392:321--424, 2004.

\bibitem{LindnerSchimansky-Geier}
B.~Lindner and L.~Schimansky-Geier.
\newblock Analytical approach to the stochastic {FitzHugh-Nagumo} system and
  coherence resonance.
\newblock {\em Phys. Rev. E}, 60(6):7270--7276, 1999.

\bibitem{LiuRoeckner2010}
W.~Liu and M.~R\"ockner.
\newblock S{PDE} in {H}ilbert space with locally monotone coefficients.
\newblock {\em J. Funct. Anal.}, 259(11):2902--2922, 2010.

\bibitem{LiuRoeckner2015}
W.~Liu and M.~R\"{o}ckner.
\newblock {\em Stochastic partial differential equations: an introduction}.
\newblock Universitext. Springer, Cham, 2015.

\bibitem{LordPowellShardlow}
G.~J. Lord, C.~E. Powell, and T.~Shardlow.
\newblock {\em An Introduction to Computational Stochastic PDEs}.
\newblock CUP, 2014.

\bibitem{LordThuemmler2012}
G.~J. Lord and V.~Th\"ummler.
\newblock Computing stochastic traveling waves.
\newblock {\em SIAM J. Sci. Comput.}, 34(1):B24--B43, 2012.

\bibitem{MacLaurinBressloff2020}
J.~N. MacLaurin and P.~C. Bressloff.
\newblock Wandering bumps in a stochastic neural field: a variational approach.
\newblock {\em Phys. D}, 406:132403, 9, 2020.

\bibitem{SchimanskyGeierMikhailovEbeling}
A.~S. Mikhailov, L.~Schimansky-Geier, and W.~Ebeling.
\newblock Effect of fluctuation on plane front propagation in bistable
  nonequilibrium systems.
\newblock {\em Ann. Physik}, 495(4):277--286, 1983.

\bibitem{MikhailovSchimanskyGeierEbeling}
A.~S. Mikhailov, L.~Schimansky-Geier, and W.~Ebeling.
\newblock Stochastic motion of the propagating front in bistable media.
\newblock {\em Physics Letters A}, 96(9):453--456, 1983.

\bibitem{MuratovVanden-Eijnden}
C.~B. Muratov and E.~Vanden-Eijnden.
\newblock Noise-induced mixed-mode oscillations in a relaxation oscillator near
  the onset of a limit cycle.
\newblock {\em Chaos}, 18:015111, 2008.

\bibitem{Nagumo}
J.~Nagumo, S.~Arimoto, and S.~Yoshizawa.
\newblock An active pulse transmission line simulating nerve axon.
\newblock {\em Proc. IRE}, 50:2061--2070, 1962.

\bibitem{Pazy1992}
A.~Pazy.
\newblock {\em Semigroups of Linear Operators and Applications to Partial
  Differential Equations}.
\newblock Applied Mathematical Sciences. Springer New York, 1992.

\bibitem{PrevotRoeckner}
C.~Pr{\'{e}}vot and M.~R{\"{o}}ckner.
\newblock {\em A Concise Course on Stochastic Partial Differential Equations},
  volume 1905 of {\em Lecture Notes in Mathematics}.
\newblock Springer, 2008.

\bibitem{GGR}
C.~Rocsoreanu, A.~Georgescu, and N.~Giurgiteanu.
\newblock {\em The FitzHugh-Nagumo Model - Bifurcation and Dynamics}.
\newblock Kluwer, 2000.

\bibitem{RottmannMatthes2010}
J.~Rottmann-Matthes.
\newblock {\em Computation and Stability of Patterns in Hyperbolic-Parabolic
  Systems}.
\newblock PhD thesis, Bielefeld University, Bielefeld, Germany, 2010.

\bibitem{RottmannMatthes2011}
J.~Rottmann-Matthes.
\newblock Linear stability of traveling waves in first-order hyperbolic {PDE}s.
\newblock {\em J. Dynam. Differential Equations}, 23(2):365--393, 2011.

\bibitem{Sandstede2002}
B.~Sandstede.
\newblock Stability of travelling waves.
\newblock In {\em Handbook of dynamical systems}, volume~2, pages 983--1055.
  Elsevier, 2002.

\bibitem{SauerStannat}
M.~Sauer and W.~Stannat.
\newblock Analysis and approximation of stochastic nerve axon equations.
\newblock {\em Math. Comp.}, 85(301):2457--2481, 2016.

\bibitem{SauerStannat2}
M.~Sauer and W.~Stannat.
\newblock Reliability of signal transmission in stochastic nerve axon
  equations.
\newblock {\em J. Comp. Neurosci.}, 40(1):103--111, 2016.

\bibitem{Seidler}
J.~Seidler.
\newblock {Da Prato-Zabczyk's} maximal inequality revisited {I}.
\newblock {\em Math. Bohem.}, 118(1):67--106, 1993.

\bibitem{Shardlow}
T.~Shardlow.
\newblock Numerical simulation of stochastic {PDEs} for excitable media.
\newblock {\em J. Comput.. Appl. Math.}, 175(2):429--446, 2005.

\bibitem{Stannat2014}
W.~Stannat.
\newblock Stability of travelling waves in stochastic bistable
  reaction-diffusion equations.
\newblock {\em arXiv:1404.3853}, 2014.

\bibitem{Szmolyan1}
P.~Szmolyan.
\newblock Transversal heteroclinic and homoclinic orbits in singular
  perturbation problems.
\newblock {\em J. Differential Equat.}, 92:252--281, 1991.

\bibitem{Tuckwell1}
H.~C. Tuckwell.
\newblock Analytical and simulation results for the stochastic spatial
  {Fitzhugh-Nagumo} model neuron.
\newblock {\em Neural Computation}, 20(12):3003--3033, 2008.

\bibitem{Tuckwell2}
H.~C. Tuckwell.
\newblock Stochastic partial differential equations in neurobiology: linear and
  nonlinear models for spiking neurons.
\newblock In {\em Stochastic Biomathematical Models}, pages 149--173. Springer,
  2013.

\bibitem{TuckwellRodriguez}
H.~C. Tuckwell and R.~Rodriguez.
\newblock Analytical and simulation results for stochastic {Fitzhugh-Nagumo}
  neurons and neural networks.
\newblock {\em J. Comput. Neurosci.}, 5(1):91--113, 1998.

\bibitem{Veraar2010}
M.~C. Veraar.
\newblock Non-autonomous stochastic evolution equations and applications to
  stochastic partial differential equations.
\newblock {\em Journal of Evolution Equations}, 10(1):85--127, Mar 2010.

\bibitem{Wang4}
B.~Wang.
\newblock Random attractors for the stochastic {FitzHugh-Nagumo} system on
  unbounded domains.
\newblock {\em Nonlinear Anal. Theor.}, 71(7):2811--2828, 2009.

\bibitem{Yanagida1985}
E.~Yanagida.
\newblock Stability of fast travelling pulse solutions of the
  {F}itz{H}ugh-{N}agumo equations.
\newblock {\em J. Math. Biol.}, 22(1):81--104, 1985.

\bibitem{Yurov}
V.~Yurov.
\newblock {\em Stability estimates for semigroups and partly parabolic reaction
  diffusion equations}.
\newblock PhD thesis, University of Missouri, Columbia, USA, 2013.

\end{thebibliography}
\bibliographystyle{plain} %the .bst file
%

%%%%%%%%%%%%%%%%%%%%%%%%%%%%%%%%%%%%%%%%%%%
\end{document}